\documentclass[twoside,11pt]{article}

\usepackage{subfigure}
\usepackage{float}
\usepackage{color}
\usepackage{jmlr2e}
\usepackage{graphics}
\usepackage{epsfig}
\usepackage[ruled]{algorithm2e}
\usepackage{booktabs}
\usepackage{url}
\usepackage{psfrag}
\usepackage{relsize}
\usepackage{amsmath,amsfonts,mathrsfs,mathtools,amssymb}
\usepackage[colorlinks,linkcolor=blue,anchorcolor=green,citecolor=green]{hyperref}
\usepackage{tikz,pgfplots}
\usepackage{tkz-tab}
\usepackage[format=hang,justification=justified,singlelinecheck=false,font=footnotesize]{caption}
\usepackage{enumerate}

\usepackage[T1]{fontenc}
\usepackage[utf8]{inputenc}
\usepackage[font=small,labelfont=bf,tableposition=top]{caption}
\usepackage{booktabs}
\usepackage{threeparttable}

\usepackage{multirow}

\newcommand{\eins}{\mathbf{1}}

\DeclareSymbolFont{wideparensymbol}{OMX}{yhex}{m}{n}
\DeclareMathAccent{\wideparen}{\mathord}{wideparensymbol}{"F3}

\newcommand{\supp}{\text{supp}}

\jmlrheading{}{}{~}{~}{~}{2019 Hanyuan Hang}

\usepackage{tikz,pgfplots}
\pgfplotsset{compat=newest}
\pgfplotsset{plot coordinates/math parser=false,trim axis left}
\newlength\figureheight
\newlength\figurewidth

\begin{document}

\title{Histogram Transform Ensembles for Density Estimation}

\author{\name Hanyuan Hang \email hanyuan.hang@samsung.com \\
\addr AI Lab \\  
Samsung Research China - Beijing \\
100028 Beijing, China}

%\editor{ }

\maketitle

%\tableofcontents
%\newpage

\begin{abstract}We investigate an algorithm named histogram transform ensembles (HTE) density estimator whose effectiveness is supported by both solid theoretical analysis and significant experimental performance. On the theoretical side, by decomposing the error term into approximation error and estimation error, we are able to conduct the following analysis: 
First of all, we establish the universal consistency under $L_1(\mu)$-norm. 
Secondly, under the assumption that the underlying density function resides in the H\"{o}lder space $C^{0,\alpha}$, we prove almost optimal convergence rates for both single and ensemble density estimators under $L_1(\mu)$-norm and $L_{\infty}(\mu)$-norm for different tail distributions, whereas in contrast, for its subspace $C^{1,\alpha}$ consisting of smoother functions, almost optimal convergence rates can only be established for the ensembles and the lower bound of the single estimators illustrates the benefits of ensembles over single density estimators. 
In the experiments, we first carry out simulations to illustrate that histogram transform ensembles surpass single histogram transforms, which offers powerful evidence to support the theoretical results in the space $C^{1,\alpha}$.
Moreover, to further exert the experimental performances, we propose an adaptive version of HTE and study the parameters by generating several synthetic datasets with diversities in dimensions and distributions.
Last but not least, real data experiments with other state-of-the-art density estimators demonstrate the accuracy of the adaptive HTE algorithm.
\end{abstract}


\section{Introduction}

Density estimation is now ubiquitous and vital in modelling more complex tasks due to the fact that once an explicit estimate of the density function is obtained, various kinds of statistical inference can be subsequently conducted, such as assessing the multimodality, skewness, or any other structure in the distribution of the data  \citep{scott2015multivariate, silverman1986density}, novelty detection \citep{Pimentel2014A}, anomaly detection \citep{Breunig2000LOF, Angiulli2002Fast}, summarizing the Bayesian posteriors, classification and discriminant analysis  \citep{simonoff1996smoothing, Rodr2006Rotation}, conducting single-level density-based clustering  \citep{Har1975}, and being proved useful in Monte Carlo computational methods like bootstrap and particle filter \citep{doucet2001an}. Other real-world applications, especially in the computer vision society, include image detection  \citep{ma2015small, liu2016highway, wang2018manifoldbased}, gesture recognition  \citep{chang2016nonparametric}, image reconstruction  \citep{ihsani2016a}, deformable 3D shape matching  \citep{vestner2017product}, image defogging \citep{jiang2017fog}, hyper-spectral unmixing \citep{zhou2018a}, geographical epidemiology \citep{Bithell2010An}, pattern recognition \citep{Lissack76Error}, just to name a few.

Formally speaking, the density estimation problem considered in this paper can be stated as follows. 
Based on i.i.d.~observations $D = \{x_1, \ldots, x_n\}$ drawn from an unknown distribution $\mathrm{P}$, we aim to estimate the underlying density $f$.
Over the decades, there has been a wealth of literature focusing on finding different appropriate methods to solve density estimation problems as well as verifying the theoretical results on the consistency and convergence rates. Among all the attempts made in this direction, nonparametric density estimators prevail since weaker assumptions are applied to the underlying probability distribution \citep{hwang1994nonparametric, hardle2012nonparametric}. Typical nonparametric density estimators include the partition-based and kernel-based density estimators and each category has its own merits in its own regimes. In this study, we are interested in the former one, namely, partition-based density estimator, due to the simple idea, fairly convenient implementation and the nature of getting rid of outliers (see \cite{freedman1981on, lugosi1996consistency}). To be specific, denote $\pi = (A_j)_{j=1}^m$ as a partition of the input space $\mathcal{X} \subset \mathbb{R}^d$, then for $x \in A_j$, the partition-based density estimator is expressed as
\begin{align}\label{equ::Kernel}
f_{\mathrm{D}, \pi}(x) := \frac{1}{n\mu (A_j)}\sum_{i=1}^n \eins_{A_j} (x_i)
\end{align}
where $\mu$ is the Lebesgue measure.
Unfortunately, although consistency \cite{glick1973samplebased, gordon1978asymptotically} and a strong universal consistency \cite{devroye1983distributionfree} of the histograms are established, they are sub-optimal for not being smooth. Moreover, the non-smoothness of histogram and therefore insufficient accuracy brings great obstacles to practical applications. 
In order to conquer this problem and taking more smoothness into account, another popular algorithm called the kernel density estimation (KDE), also known as \emph{Parzen-Rosenblatt estimation} comes into the public view. \cite{hang2018kernel} verifies that almost optimal convergence rates can be achieved with kernel and bandwidth chosen appropriately when dealing with cases where the density is assumed to be smooth. Theoretically, these optimal rates depend on the order of smoothness of density function on the entire input space. In the actual world, however, the smoothness of density function varies from areas to areas. That is to say, due to the lack of adaptivity, KDE can be largely sensitive to outliers, letting spurious bumps to appear, and tending to flatten the peaks and valleys of the density (see \cite{terrell1992variable, botev2010kernel}).
With the aim of filling possible gaps, a wealth of literature is pulled into finding desirable density estimators based on appropriate partitions (see e.g., \cite{klemela2009multivariate}, \cite{Density2014}
%, \cite{Hang2019Best}
), wavelet \citep{David96Wavelet}, extended Engle's ARCH model \citep{Bruce94Autoregressive}, Bayesian methods \citep{Escobar95Bayesian, liu2014multivariate, liu2015multivariate}, just to name a few. 
Nevertheless, as far as we know, it is a challenge for an algorithm to have the theoretical availability for both local and global analysis, the experimental advantages of achieving efficient and accurate prediction results on real data, and stronger resistance to the curse of dimensionality compared to the existing common algorithms.

This study is conducted under such background, aiming at solving these tough problems mentioned earlier. To be specific, motivated by the random rotation ensemble algorithms proposed in \cite{Rodr2006Rotation, Ezequiel2013HT, JMLR:v17:blaser16a}, we investigated a density estimator named histogram transform ensembles which takes full advantage of the histogram methods and ensemble learning. Specifically, its merits can be stated as twofold.
First, the algorithm can be local adaptive by applying adaptive stretching with respect to samples of each dimension for a piecewise constant function to approximate the true density. Second, the global smoothness of our obtained density estimator is attributed to the randomness of different partitions together with the integration of multiple histograms. 
The algorithm starts with mapping the input space into transformed feature space under a certain histogram transform. Then, the process proceeds by partitioning the transformed space into non-overlapping cells with the unit bandwidth where the bin indices are chosen as the round points. Our histogram transform ensembles density estimator is finally specified by \eqref{equ::Kernel} in the corresponding cell in the input space.
Last but not least, by integrating estimators generated by the above procedure, we obtain a density ensemble with satisfying asymptotic smoothness.

The contributions of this paper come from both the theoretical and experimental aspects: \emph{(i)} Considering from a theoretical perspective, 
our histogram transform ensembles density estimator shows its power in achieving universal consistency and almost optimal convergence rates under some mild conditions. Firstly, we show that the universal consistency is obtained under $L_1(\mu)$-norm. Secondly, by decomposing the error term into approximation error and estimation error, which correspond to data-free and data-dependent error terms, respectively, we obtain the convergence rate in the space $C^{0,\alpha}$ in the sense of $L_1(\mu)$-norm and $L_{\infty}(\mu)$-norm both optimal up to a logarithmic factor.
Thirdly, we present a theorem in the space $C^{1,\alpha}$ to illustrate the benefits of histogram transform ensembles over single histogram transform density estimators from a theoretical point of view. 
More precisely, we show that in the space $C^{1,\alpha}$ the histogram transform ensembles can attain the almost optimal rate $O((n/\log n)^{-(1+\alpha)/(2(1+\alpha)+d)})$ 
whereas a single histogram transform density estimator 
fails to achieve this rate 
whose lower bound is of order $O(n^{-1/(2+d)})$ under certain conditions.
\emph{(ii)} In the experimental results, we first verify the estimation power of ensemble estimators over the single ones, which corresponds our former theoretical demonstrations, based on HTE with independent splittings, also referred to as na\"{i}ve histogram transform ensembles (NHTE). Note that we base the entire theoretical analysis on this NHTE algorithm. Then in order to further exert the prediction ability of this histogram-based density estimator, we take more sample information into consideration, and propose the adaptive histogram transform ensembles (AHTE). 
It is worth mentioning that the randomness of partitions originated from the distribution of histogram transform together with the inherent nature of ensembles allow us to build a density estimator with satisfying asymptotic smoothness, which greatly improves the progress of prediction.
As a result, when conducting both synthetic and real data comparisons, our AHTE algorithm is predominant in accuracy compared with other state-of-the-art density estimators, including NHTE and the most commonly used KDE method.

This paper is organized as follows. Section \ref{sec::methodology} is a preliminary section covering some required fundamental notations, definitions and technical histogram transform that is related to our one ensemble variation of histogram density estimator. Section \ref{sec::TheoreticalResults} is concerned with main results on the consistency and convergence rates under different norms of our density estimator. To be specific, universal consistency under $L_1(\mu)$-norm with relatively weak constrains is shown in Section \ref{sec::consistency}, and the convergence rates 
under $L_1(\mu)$-norm and $L_{\infty}(\mu)$-norm
 are derived in Section \ref{subsec::C0}. 
 Moreover, a more complete theory is obtained by conducting upper and lower bounds to illustrate the benefits of histogram transform ensembles $f_{\mathrm{P}, \mathrm{E}}$ over single histogram transform density estimators $f_{\mathrm{P}, H}$ in Section \ref{sec::LowerBoundSingles}. Some comments and discussions related to the main results will also be presented in this section. Section \ref{sec::error_analysis} provides a detailed exposition of bounding decomposed error terms named approximation error and estimation error respectively. In Section \ref{sec::numerical_experiments}, a numerical example is shown in Section \ref{sec::counter} to illustrate the benefits of ensemble estimators over single ones, which coincides with the former theoretical analysis. For the rest of Section \ref{sec::numerical_experiments}, we provide an adaptive version of histogram transform ensembles, namely AHTE, and conduct
numerical comparisons among different density estimation methods based on both synthetic and real data sets. For the sake of clarity, we place all the proofs of Section \ref{sec::TheoreticalResults} and Section \ref{sec::error_analysis} in Section \ref{sec::proofs}. In Section \ref{sec::conclusion}, we close this paper with a conclusive summary, a brief discussion and several remarks.

\section{Methodology} \label{sec::methodology}

In this section, we aim to study an algorithm named histogram transform ensembles (HTE) for density estimation problem. Firstly, in Section \ref{sec::prelims}, we introduce some mathematical notations to be used throughout the entire paper.
Then in Section \ref{sec::HTE} we present the so called histogram transform approach through defining every crucial element such as rotation matrix $R$, stretching matrix $S$ and translator vector $b$. Based on the partition of the input space induced by the histogram transforms, we are then able to formulate the HTE for density estimation in section \ref{sec::alHTE}.

\subsection{Notations}\label{sec::prelims}

Let $\mathcal{X} \subset \mathbb{R}^d$ be a subset, $\mu := \lambda^d$ be the Lebesgue measure with $\mu(\mathcal{X}) > 0$, and $\mathrm{P}$ be a probability measure with support $\mathcal{X}$ which is absolute continuous with respect to $\mu$ with density $f$. 
We denote $B_r$ as the centered ball of $\mathbb{R}^d$ with radius $r$, that is
\begin{align*}
B_r 
:= [-r, r]^d
:= \{ x = (x_1, \ldots, x_d) \in \mathbb{R}^d : x_i \in [-r, r], i = 1, \ldots, d \},
\end{align*}
write $B_r^c := \mathbb{R}^d \setminus B_r$ for the complement of $B_r$, and
for any $\delta \in (0, r/2)$,
\begin{align*}
B_{r,\delta}^+ 
:= [\delta, r - \delta]^d
:= \{ x = (x_1, \ldots, x_d) \in \mathbb{R}^d : x_i \in [\delta, r - \delta], i = 1, \ldots, d \}.
\end{align*}
Moreover, for any $x \in \mathbb{R}^d$ and $r > 0$, $B_r(x)$ denote the ball with center $x$ and radius $r$.
Recall that for $1\leq p < \infty$, the $L_p$-norm is defined as $\| x \|_p := (x_1^p + \ldots + x_d^p)^{1/p}$, and the $L_{\infty}(\mu)$-norm is defined as $\| x \|_{\infty} := \max_{i=1,\ldots,d} |x_i|$. 
Throughout this paper, we shall make frequent use of the following multi-index notations.
For any vector $x = (x_i)_{i=1}^d \in \mathbb{R}^d$ and numbers $a, \underline{x}_0, \overline{x}_0 \in \mathbb{R}$, we write
$\lfloor x \rfloor := (\lfloor x_i \rfloor)_{i=1}^d$,
$x^{-1} := (x_i^{-1})_{i=1}^d$,
$a - x := (a - x_i)_{i=1}^d$,
$\log(x) := (\log x_i)_{i=1}^d$ and 
$x \in [\underline{x}_0, \overline{x}_0]$ indicates $x_i \in [\underline{x}_0, \overline{x}_0]$, $i=1,\ldots,d$.
In the sequel, we use the notation $a_n \lesssim b_n$ to denote that there exists a positive constant $c$ such that $a_n \leq c b_n$, for all $n \in \mathbb{N}$.

\subsection{Histogram Transforms} \label{sec::HTE}
To give a clear description of one possible construction procedure of histogram transforms, we introduce a random vector $(R, S, b)$ which represents the rotation matrix, stretching matrix and translation vector, respectively. To be specific, 
\begin{itemize}
    \item 
[$R$] denotes the rotation matrix which is a real-valued $d \times d$ orthogonal square matrix with unit determinant, that is
\begin{align}\label{RotationMatrix}
R^{\top} = R^{-1} \quad \text{ and } \quad \det(R) = 1.
\end{align}
    \item 
[$S$] stands for the stretching matrix which is a positive real-valued $d \times d$ diagonal scaling matrix with diagonal elements $(s_i)_{i=1}^d$ that are certain random variables.
Obviously, there holds
\begin{align}\label{StretchingMatrix}
\det(S) = \prod_{i=1}^d s_i.
\end{align}
Moreover, we denote 
\begin{align}\label{equ::s}
{s} = (s_i)_{i=1}^d,
\end{align}
and the bin width vector measured on the input space is given by
\begin{align}\label{equ::h}
h =  s^{-1}.
\end{align}
    \item 
[$b$] $\in [0,1]^d$ is a $d$ dimensional vector named translation vector. 
\end{itemize}

\begin{figure}[H]
\begin{minipage}[t]{0.99\textwidth}  
\centering  
\includegraphics[width=\textwidth]{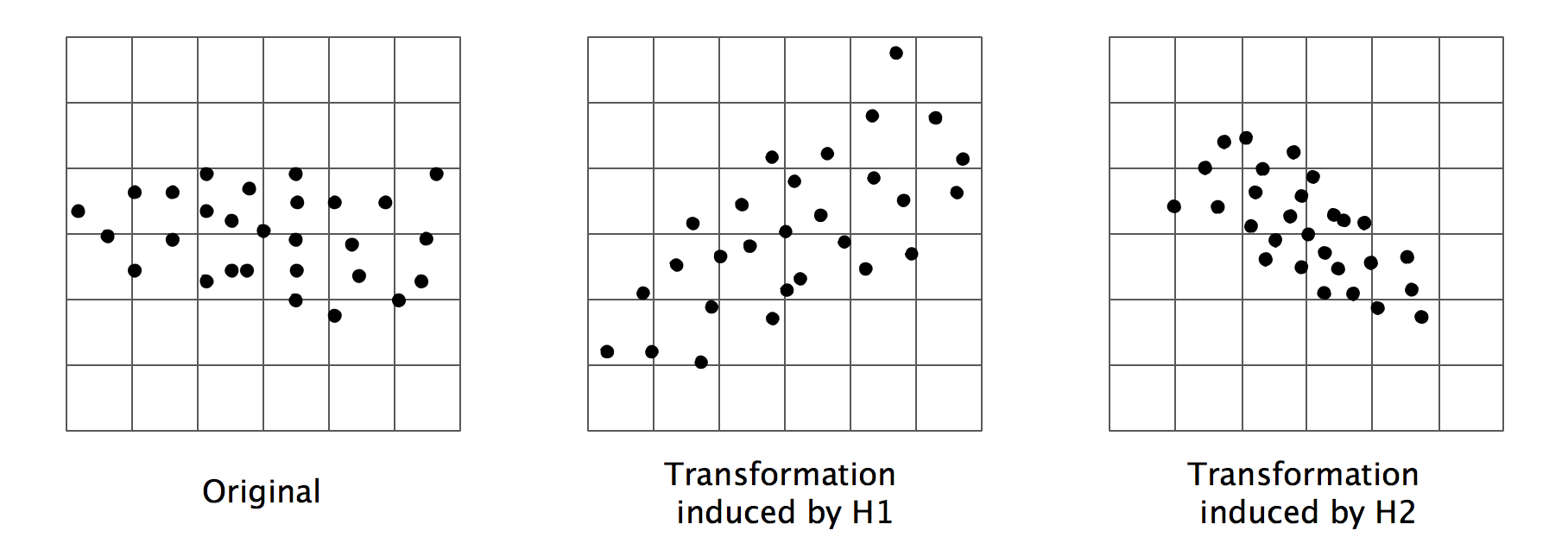}  
\end{minipage}  
\centering  
\caption{Histogram transforms, composed of random rotations, stretchings, and translations.}
\label{fig:RHT}
\end{figure}

Here we describe a practical method for their construction we are confined to in this study. Starting with an $d \times d$ square matrix $M$, consisting of $d^2$ independent univariate standard normal random variates, a Householder $Q R$ decomposition  \cite{Householder1958Unitary} is applied to obtain a factorization of the form $M = R \cdot W$, with orthogonal matrix $R$ and upper triangular matrix $W$ with positive diagonal elements. The resulting matrix $R$ is orthogonal by construction and can be shown to be uniformly distributed. Unfortunately, if $R$ does not feature a positive determinant then it is not a proper rotation matrix according to definition (\ref{RotationMatrix}). However, if this is the case then we can flip the sign on one of the column vectors of $M$ arbitrarily to obtain $M^+$ and then repeat the Householder decomposition. The resulting matrix $R^+$ is identical to the one obtained earlier but with a change in sign in the corresponding column and $\det(R^+) = 1$, as required for a proper rotation matrix, see \cite{JMLR:v17:blaser16a} for a brief account of the existed algorithms to generate random orthogonal matrices.

After that, we build a diagonal scaling matrix with the signs of the diagonal of $S$ where the elements $s_k$ are the well known Jeffreys prior \citep{Jeffreys1946An}, that is, we draw $\log (s_i)$ from the uniform distribution over certain interval of real numbers $[\log (\underline{s}_0), \log (\overline{s}_0)]$ for fixed constants $\underline{s}_0$ and $\overline{s}_0$ with $0<\underline{s}_0 < \overline{s}_0 < \infty$. By \eqref{equ::h}, there holds $h_i \in [\overline{s}_0^{-1}, \underline{s}_0^{-1}]$, $i=1,\ldots,d$. For simplicity and uniformity of notations, in the sequel, we denote $\overline{h}_0 = \underline{s}_0^{-1}$ and $\underline{h}_0 = \overline{s}_0^{-1}$, then we can say $h_i \in [\underline{h}_0, \overline{h}_0]$, $i=1,\ldots,d$.

Moreover, the translation vector $b$ is drawn from the uniform distribution over the hyper-cube $[0, 1]^d$.

Based on the above notations, we define the histogram transform $H : \mathcal{X} \to \mathcal{X}$ by
\begin{align}\label{HistogramTransform}
H(x) := R \cdot S \cdot x + b
\end{align}
as is shown in Figure \ref{fig:RHT}, and the corresponding distribution by $\mathrm{P}_H := \mathrm{P}_R \otimes \mathrm{P}_S \otimes \mathrm{P}_b$, where $\mathrm{P}_R$, $\mathrm{P}_S$ and $\mathrm{P}_b$ represent the distribution for rotation matrix $R$, stretching matrix $S$ and translation vector $b$ respectively.

Moreover, we denote $H'$ as the affine matrix $R \cdot S$. Clearly, there holds
\begin{align}
\det(H') = \det(R) \cdot \det(S) = \prod_{i=1}^d s_i.
\end{align} 
The histogram probability $p (x | H', b)$ is defined by considering the bin width $h = 1$ in the transformed space. It is important to note that there is no point in using $h \neq 1$, since the same effect can be achieved by scaling the transformation matrix $H'$. Therefore, let $\lfloor H(x) \rfloor$ be the transformed bin indices, then the transformed bin is given by 
\begin{align}\label{TransBin}
A'_H(x) := \{ H(x') \ | \ \lfloor H(x') \rfloor = \lfloor H(x) \rfloor \}.
\end{align}
The corresponding histogram bin containing $x \in \mathcal{X}$ is 
\begin{align}\label{equ::InputBin}
A_H(x) := \{ x' \ | \ H(x') \in A'_H(x) \}
\end{align}
whose volume is $\mu(A_H(x)) = (\det(H'))^{-1}$.

For a fixed histogram transform $H$, we specify the partition of $B_r$ induced by the histogram rule \eqref{equ::InputBin}. 
Let $(A'_j)$ be the set of all cells generated by $H$, denote $\mathcal{I}_H$ as the index set for $H$ such that $A'_j \cap B_r \neq \emptyset$. As a result, the set
\begin{align*}
\pi_H 
:= (A_j)_{j \in \mathcal{I}_H}
:= (A'_j \cap B_r)_{j \in \mathcal{I}_H}
\end{align*}
forms a partition of $B_r$. For notational convenience, if we substitute $A_0$ for $B_r^c$, then 
\begin{align*}
\pi'_H := (A_j)_{j \in \mathcal{I}_H \cup \{ 0 \}}
\end{align*}
forms a partition of $\mathbb{R}^d$.

\subsection{Histogram Transform Ensembles (HTE) for Density Estimation}\label{sec::alHTE}

\subsubsection{Histogram Transform Density Estimator}

The histogram transform density of a probability measure $\mathrm{Q}$ can be defined as follows:

\begin{definition}[Histogram Transform Density]
Let $\mathrm{Q}$ be a probability measure on $\mathbb{R}^d$ and $H$ be the histogram transform defined as in \eqref{HistogramTransform}. Then the function $f_{\mathrm{Q},H} : \mathbb{R}^d  \to [0, \infty)$ defined by
\begin{align}\label{RHdensity}
f_{\mathrm{Q},H}(x)
= \sum_{j \in \mathcal{I}_H \cup \{ 0 \}} \frac{\mathrm{Q}(A_j)\eins_{A_j}(x)}{\mu(A_j)}  
\end{align}
is called a \emph{histogram transform density} of $\mathrm{Q}$.
\end{definition}

We demonstrate that $f_{\mathrm{Q},H}$ defines the density of a probability measure on $\mathbb{R}^d$ for $f_{\mathrm{Q},H}$ is measurable and 
\begin{align*}
\int_{\mathbb{R}^d} f_{\mathrm{Q},H}(x) \, d\mu(x)
& = \int_{\mathbb{R}^d} \sum_{j \in \mathcal{I}_H \cup \{ 0 \}} \frac{\mathrm{Q}(A_j)\eins_{A_j}(x)}{\mu(A_j)} \, d\mu(x)
   = \sum_{j \in \mathcal{I}_H \cup \{ 0 \}} \int_{A_j} \frac{\mathrm{Q}(A_j)}{\mu(A_j)} \, d\mu(x)  
\\
& = \sum_{j \in \mathcal{I}_H \cup \{ 0 \}} \frac{\mathrm{Q}(A_j) \mu(A_j)}{\mu(A_j)}  
   = \sum_{j \in \mathcal{I}_H \cup \{ 0 \}} \mathrm{Q}(A_j) 
   = \mathrm{Q}(\mathbb{R}^d)   
   = 1.
\end{align*}

Recalling that $\mathrm{P}$ is a probability measure on $\mathbb{R}^d$ with the corresponding density function $f$, by taking $\mathrm{Q} = \mathrm{P}$ with $d\mathrm{P} = f \, d\mu$, then for $x \in A_j$, we have
\begin{align} \label{equ::RHdensity}
f_{\mathrm{P},H}(x) 
= \frac{\mathrm{P}(A_j)}{\mu(A_j)} 
= \frac{\int_{A_j} f(x') \, d\mu(x')}{\mu(A_j)}. 
\end{align}
Specifically, when $\mathrm{Q}$ is the empirical measure $\mathrm{D}_n = \frac{1}{n} \sum_{i=1}^n \delta_{x_i}$, then $\mathrm{D}_n(A)$ is the expectation of $\eins_A$ with respect to $\mathrm{D}_n$, which is
\begin{align*}
\mathrm{D}_n(A)
= \mathbb{E}_{\mathrm{D}_n} \eins_A
= \frac{1}{n} \sum_{i=1}^n \delta_{x_i}(A)
= \frac{1}{n} \sum_{i=1}^n \eins_A(x_i),
\end{align*}
then the histogram transform density in this study can be expressed as
\begin{align}\label{equ::EMRHdensity}
f_{\mathrm{D}_n,H}(x)
= \frac{\mathrm{D}_n(A_j)}{\mu(A_j)}
= \frac{1}{n\mu(A_j)} \cdot \sum_{i=1}^n \eins_{A_j}(x_i),
\end{align}
which can also be expressed in the transformed space as 
\begin{align*}
f_{\mathrm{D}_n,H}(x)
= \frac{1}{n\mu(A_j)} \cdot \sum_{i=1}^n \eins_{\{ \lfloor H(x) \rfloor = \lfloor H(x_i) \rfloor \}}(x_i).
\end{align*}

From now on, for notational simplicity, we will suppress the subscript $n$ of $\mathrm{D}_n$ and denote $\mathrm{D} := \mathrm{D}_n$, e.g., $f_{\mathrm{D},H} := f_{\mathrm{D}_n,H}$. The map from the training data to $f_{\mathrm{D},H}$ is called the histogram transform density rule with histogram transform $H$.

\subsubsection{Histogram Transform Ensembles (HTE) Density Estimator}

Now we formulate the histogram transform ensembles for density estimation. Ensembles consisting of different estimators have been highly recognized as an effective technique to improve the performance over single estimator in the study which inspire us to apply them to the histogram transform density estimator.

Let $\{ f_{\mathrm{D},H_t} \}_{t=1}^T$ be $T$ histogram transform density estimators generated by the histogram transforms $\{ H_t \}_{t=1}^T$ respectively, which is defined by
\begin{align*} 
f_{\mathrm{D},H_t}(x)
= \sum_{j\in \mathcal{I}_{H_t} \cup \{ 0 \}} \frac{\mathrm{D}(A_{t,j})\eins_{A_{t,j}}(x)}{\mu(A_{t,j})},
\end{align*}
where $\{A_{t,j} \}_{j \in \mathcal{I}_{H_t}}$ is the random partition of $B_r$ induced by histogram transform $H_t$. Therefore, the histogram transform ensembles for density estimation can be presented as 
\begin{align}\label{histogramensemble}
f_{\mathrm{D},\mathrm{E}}(x)
:= \frac{1}{T} \sum_{t=1}^T f_{\mathrm{D},H_t}(x).
\end{align}

\subsubsection{Main Algorithm}\label{sec::NHTE}

As is shown in Algorithm \ref{alg::NHTE}, histogram transform ensembles (HTE) for density estimation is constructed following the manner of \textit{ensemble learning}. We first generate $T$ histogram density estimators each induced by a random histogram transform partition, and then the ensemble estimator is built simply by taking the average.

\begin{algorithm}[htbp]
%	\SetAlgoNoLine
\caption{HTE for Density Estimation}
\label{alg::NHTE}
\KwIn{
Training data $D:=(X_1, \ldots, X_n)$;
\\
\quad\quad\quad\quad Number of histogram transforms $T$;
\\
\quad\quad\quad\quad Lower and upper bound of bandwidth parameters $\underline{h}_0$ and $\overline{h}_0$.
} 
\For{$t =1 \to T$}{

%		\For{$i = 1 \to M$}{
Generate random affine transform matrix
$H_t=R_t\cdot S_t + b_t$;\\
Apply data independent splitting to the transformed sample space; \\
Apply piecewise constant density estimators to each cell; \\ 
Compute the histogram density estimator $f_{\mathrm{D},H_t}(x)$ induced by $H_t$.
%		}
%		Select the best mapping $f_{\mathrm{D},H_t}(x)$ with the minimal error.
}
\KwOut{The histogram transform ensemble for density estimation is
%		\begin{spacing}{0.55}
\begin{align*}
	f_{\mathrm{D},\mathrm{E}}(x)
	= \frac{1}{T} \sum_{t=1}^T f_{\mathrm{D},H_t}(x).
\end{align*}
%		\end{spacing}
}
\end{algorithm}

\section{Theoretical Results}\label{sec::TheoreticalResults}

In this section, we present main results on the convergence rates of our empirical decision function $f_{\mathrm{D},H}$ and $f_{\mathrm{D},\mathrm{E}}$ to the Bayesian decision function $f$ in the sense of different norms under certain restrictions.

This section is organized as follows. In Section \ref{sec::assump}, we firstly introduce some fundamental assumptions to be utilized in the theoretical analysis. Then we prove the universal consistency of the HTE under relatively weak assumptions in Section \ref{sec::consistency}. In section \ref{subsec::C0}, we show that
almost optimal convergence rates can be attained by both single and ensemble HTE, whereas in Section \ref{sec::LowerBoundSingles}, for the subspace $C^{1,\alpha}$ consisting of smoother functions, almost optimal convergence rates can only be established for the HTE ensembles and the lower bound of single HTE illustrates the benefits of ensembles over singles. Last but not least, we also present some comments and discussions on the obtained main results as is shown in Section \ref{sec::Comments}.

\subsection{Fundamental Assumptions}\label{sec::assump}

To demonstrate theoretical results concerning convergence rates, fundamental assumptions are required respectively for the target function $f$ and the bin width $h$ of stretching matrix $S$.

First of all, we introduce a general function space $C^{k, \alpha}$ consisting of $(k, \alpha)$-H\"{o}lder continuous functions.

\begin{definition}\label{def::Cp}
Let $r \in (0,\infty), k \in \mathbb{N} \cup \{ 0 \}$ and $\alpha \in (0, 1]$. We say that a function $f : B_r \to \mathbb{R}$ is $(k, \alpha)$-H\"{o}lder continuous, if there exists a constant $c_L \in (0, \infty)$ such that 
\begin{itemize}
\item[(i)] 
$\| \nabla^{\ell} f \| \leq c_L$ for all $\ell \in \{ 1, \ldots, k \}$;
\item[(ii)] 
$\| \nabla^k f(x) - \nabla^k f(x') \| \leq c_L \| x - x' \|^{\alpha}$ for all $x, x' \in B_r$.
\end{itemize}
The set of such functions is denoted by $C^{k, \alpha} (B_r)$.
\end{definition}

Note that
for the special case $k = 0$, the resulting function space $C^{0, \alpha} (B_r)$ coincides with the 
commonly used $\alpha$-H\"{o}lder continuous function space $C^{\alpha} (B_r)$.

Now we assume the underlying true density $f$ lies in the space $C^{k, \alpha}$.

\begin{assumption}\label{assumption::ck}
Let $f$ be the underlying true density, and assume that $f \in C^{k,\alpha}$, where $\alpha \in (0,1]$ and $k = 0, 1$.
\end{assumption}

Then we assume the upper and lower bounds of the bin width $h$ are of the same order, that is, in a specific partition, the extent of stretching in each dimension cannot vary too much. Mathematically, we assume that the stretching matrix $S$ is confined into the class with width satisfying the following conditions.

\begin{assumption}\label{assumption::h}
Let the bandwidth $h$ be defined as in \eqref{equ::h}, assume that 
there exists some constant $c_0 \in (0,1)$ such that 
\begin{align*}
c_0 \overline{h}_0 \leq \underline{h}_0 \leq c_0^{-1}\overline{h}_0.
\end{align*}
In the case that the bandwidth $h$ depends on the sample size $n$, assume that 
there exist constants $c_{0,n} \in (0,1)$ such that 
\begin{align*}
c_{0,n} \overline{ {h}}_{0,n} \leq  \underline{ {h}}_{0,n} \leq c_{0,n}^{-1} \overline{ {h}}_{0,n}.
\end{align*}
\end{assumption}

\subsection{Universal Consistency under $L_1(\mu)$-norm}\label{sec::consistency}

In this subsection, we firstly establish the density estimator $f_{\mathrm{D},H_{t,n}}$ induced by a certain histogram transform $H_t$ related with sample number $n$. Note that for the sake of the simplicity and uniformity of notations, we omit the index $t$ for a fixed $t \in \{1,\ldots,T\}$ and substitute $f_{\mathrm{D},H_n}$ for $f_{\mathrm{D},H_{t,n}}$.  
Moreover, for the sake of convenience, we write $\nu_n := \mathrm{P}^n \otimes \mathrm{P}_H$.

We present results on the universal consistency property of the histogram transform density estimator $f_{\mathrm{D},H_n}$ in the sense of $L_1(\mu)$-norm. To clarify, an estimator $f_{\mathrm{D},H_n}$ is said to be universally consistent in the sense of $L_1(\mu)$-norm if $f_{\mathrm{D},H_n}$ converges to $f$ under $L_1(\mu)$-norm $\nu_n$-almost surely for arbitrary distributions of $\mathrm{P}$ and $\mathrm{P}_H$.

It is worth mentioning that we adopt $L_1(\mu)$-norm to show the consistency when the underlying density function lies in the space $C^{0,\alpha}$, because in this way, we are able to provide a more general conclusion under relatively weak assumptions, where H\"{o}lder continuity is not required.

\begin{theorem} \label{ConsistencyL1}
Let the histogram transform $H_n$ be defined as in \eqref{HistogramTransform} with bandwidth $h_n$ satisfying Assumption \ref{assumption::h}. If 
\begin{align*}
\overline{ {h}}_{0,n} \to 0 
\quad \hbox{and} \quad 
\frac{n \underline{ {h}}_{0,n}^d}{\log n} \to \infty, 
\quad \hbox{as} \quad 
n \to \infty,
\end{align*}
then the density estimator $f_{\mathrm{D},H_n}$ is universally consistent in the sense of $L_1(\mu)$-norm.
\end{theorem}

\subsection{Results in the Space $C^{0,\alpha}$}\label{subsec::C0}

\subsubsection{Convergence Rates for Single Estimators under $L_1(\mu)$-Norm} \label{sec::l1_convergence}

We firstly establish the convergence rates of single histogram transform density estimators under $L_1(\mu)$-norm with three different tail assumptions imposed on $\mathrm{P}$. In particular, analysis will be conducted in situations where the tail of the probability distribution $\mathrm{P}$ has a polynomial decay, exponential decay and disappears, respectively.

\begin{theorem} \label{theorem::ConvergenceRatesL1}
Let the histogram transform $H_n$ be defined as in \eqref{HistogramTransform} with bandwidth $h_n$ satisfying Assumption \ref{assumption::h}. Moreover, suppose that the density $f\in C^{0,\alpha}$. We consider the following cases:
\begin{enumerate}
    \item[(i)] 
$\mathrm{P}(B_r^c) \lesssim r^{- \eta d}$ for some $\eta > 0$ and for all $ r \geq 1$;
    \item[(ii)] 
$\mathrm{P}(B_r^c) \lesssim e^{- a r^\eta}$ for some $a > 0$, $\eta > 0$ and for all $r \geq 1$;
    \item[(iii)] 
$\mathrm{P}(B_{r_0}^c) = 0$ for some $r_0 \geq 1$;
\end{enumerate} 
and the sequences $\underline{h}_{0,n}$ are of the following forms:
\begin{enumerate}
    \item[(i)] 
$\underline{h}_{0,n} =  (\log n/n)^{\frac{1+\eta}{(1+\eta)(2\alpha+d)-\alpha}}$;
    \item[(ii)] 
$\underline{h}_{0,n} =  (\log n/n)^{\frac{1}{2\alpha+d}} (\log n)^{- \frac{d}{\eta} \cdot \frac{1}{2\alpha+d}}$;
    \item[(iii)] 
$\underline{h}_{0,n} =  (\log n/n)^{\frac{1}{2\alpha+d}}$;
\end{enumerate}
then with probability $\nu_n$ at least $1 - \frac{1}{n}$, there holds
\begin{align*}
\|f_{\mathrm{D},H_n} - f\|_{L_1(\mu)} \leq \varepsilon_n, 
\end{align*}
where the convergence rates  
\begin{enumerate}
    \item[(i)] 
$\varepsilon_n \lesssim  (\log n/n)^{\frac{\alpha\eta}{(1+\eta)(2\alpha+d)-\alpha}}$;
    \item[(ii)] 
$\varepsilon_n \lesssim  (\log n/n)^{\frac{\alpha}{2\alpha+d}} (\log n)^{\frac{d}{\eta} \cdot \frac{\alpha+d}{2\alpha+d}}$;
    \item[(iii)] 
$\varepsilon_n \lesssim  (\log n/n)^{\frac{\alpha}{2\alpha+d}}$.
\end{enumerate}
\end{theorem}

\subsubsection{Convergence Rates for Single Estimators under $L_{\infty}(\mu)$-Norm} \label{sec::linfty_convergence}

We now state our main results on the convergence rates of single histogram transform density estimators under $L_{\infty}(\mu)$-norm.

\begin{theorem}\label{ConvergenceRatesLInfty}
Let the histogram transform $H_n$ be defined as in \eqref{HistogramTransform} with bandwidth $h_n$ satisfying Assumption \ref{assumption::h}. Moreover, assume that $\mathcal{X} \subset B_r \subset \mathbb{R}^d$ and the density function $f \in C^{0,\alpha}$ with $\|f\|_{L_{\infty}(\mu)} < \infty$. Then for all $x \in B^+_{r, \sqrt{d} \cdot \overline{h}_{0,n}}$ and all  $n \geq 1$, by choosing
\begin{align*}
\underline{h}_{0,n} := (\log n / n)^{\frac{1}{2\alpha+d}},
\end{align*}
there holds
\begin{align}\label{RatesInfty}
\|f_{\mathrm{D},H_n} - f\|_{L_{\infty}(\mu)}
\lesssim (\log n/n)^{\frac{\alpha}{2\alpha+d}}
\end{align}	
with probability $\nu_n$ at least $1 - \frac{1}{n}$.
\end{theorem}

\subsubsection{Convergence Rates for Ensemble Estimators under $L_1(\mu)$-norm}

In this subsection, we state the convergence rates of histogram transform ensembles. First of all, the following theorem establishes the convergence rates of $f_{\mathrm{D},\mathrm{E}}$ in the sense of $L_1(\mu)$-norm under three tail probability distributions.

\begin{theorem}\label{ensembleL1}
Let the histogram transform $H_n$ be defined as in \eqref{HistogramTransform} with bandwidth $h_n$ satisfying Assumption \ref{assumption::h} and $T$ be the number of single estimators contained in the ensembles. Moreover, suppose that the density $f \in C^{0,\alpha}$. We consider the following cases: 
\begin{enumerate}
    \item[(i)] 
$\mathrm{P}(B_r^c) \lesssim r^{- \eta d}$ for some $\eta > 0$ and for all $r \geq 1$;
    \item[(ii)] 
$\mathrm{P}(B_r^c) \lesssim e^{- a r^\eta}$ for some $a > 0$, $\eta > 0$ and for all $r \geq 1$;
    \item[(iii)] 
$\mathrm{P}(B_{r_0}^c) = 0$ for some $r_0 \geq 1$;
\end{enumerate} 
and the sequences $\underline{h}_{0,n}$ are of the following forms:
\begin{enumerate}
    \item[(i)] 
$\underline{h}_{0,n} =  (\log n/n)^{\frac{1+\eta}{(1+\eta)(2\alpha+d)-\alpha}}$;
    \item[(ii)] 
$\underline{h}_{0,n} =  (\log n/n)^{\frac{1}{2\alpha+d}} (\log n)^{- \frac{d}{\eta} \cdot \frac{1}{2\alpha+d}}$;
    \item[(iii)] 
$\underline{h}_{0,n} =  (\log n/n)^{\frac{1}{2\alpha+d}}$;
\end{enumerate}
then with probability $\nu_n$ at least $1 - \frac{T}{n}$, there holds
\begin{align*}
\|f_{\mathrm{D}, \mathrm{E}} - f\|_{L_1(\mu)} \leq \varepsilon_n, 
\end{align*}
where the convergence rates  
\begin{enumerate}
    \item[(i)] 
$\varepsilon_n \lesssim  (\log n/n)^{\frac{\alpha\eta}{(1+\eta)(2\alpha+d)-\alpha}}$;
    \item[(ii)] 
$\varepsilon_n \lesssim  (\log n/n)^{\frac{\alpha}{2\alpha+d}} (\log n)^{\frac{d}{\eta} \cdot \frac{\alpha+d}{2\alpha+d}}$;
    \item[(iii)] 
$\varepsilon_n \lesssim  (\log n/n)^{\frac{\alpha}{2\alpha+d}}$.
\end{enumerate}
\end{theorem}

\subsubsection{Convergence Rates for Ensemble Estimators under $L_{\infty}(\mu)$-norm}

Finally, we present the convergence rates of the ensembles $f_{\mathrm{D},\mathrm{E}}$ under $L_{\infty}(\mu)$-norm.

\begin{theorem}\label{ensembleLinfty}
Let the histogram transform $H_n$ be defined as in \eqref{HistogramTransform} with bandwidth $h_n$ satisfying Assumption \ref{assumption::h} and $T$ be the number of single estimators contained in the ensembles. Moreover, assume that $\mathcal{X} \subset B_r \subset \mathbb{R}^d$ with $r \in (0, \infty)$ and the density function $f \in C^{0,\alpha}$. 
Then for any $x \in B^+_{r, \sqrt{d} \cdot \overline{h}_{0,n}}$, there exists some $N_0 \in \mathbb{N}$ such that for all $n \geq N_0$, by choosing
\begin{align*}
\underline{h}_{0,n} := (\log n / n)^{\frac{1}{2\alpha+d}},
\end{align*}
there holds
\begin{align}\label{RatesInftyEnsemble}
\|f_{\mathrm{D},\mathrm{E}} - f\|_{L_{\infty}(\mu)}
\lesssim (\log n/n)^{\frac{\alpha}{2\alpha+d}}
\end{align}	
with probability $\nu_n$ at least $1 - \frac{T}{n}$.
\end{theorem}

At first glance, the ensembles and singles seem to share the same convergence rate. However, a closer look shows that the ensembles have even worse rates than single estimators, since they converge with probability $1-\frac{T}{n}$ instead of $1-\frac{1}{n}$ for a relatively large $T$. As a matter of fact, we can say that ensembles and singles converge at the roughly same rate only when $T$ is a constant independent of $n$. As a result, with the above analysis in the space $C^{0,\alpha}$
we fail to explain how ensemble histogram transforms exceed single estimators. Therefore,  we are motivated to search for theoretical clues of this remaining puzzle in the subspace $C^{1,\alpha}$.

\subsection{Results in the Space $C^{1,\alpha}$} \label{sec::LowerBoundSingles}

In this subsection, we turn to the subspace $C^{1,\alpha}$ and provide a result that illustrates the benefits of histogram transform ensembles over single density estimators by both verifying the optimal convergence rate for ensembles and establishing the lower bound of the single estimators.

As Theorem \ref{thm::LowerBoundSingles} shows below, single histogram transform density estimators do not benefit the stronger smoothness assumption and can not achieve the same rate as the ensembles due to the fact that the approximation error of single histogram transform density estimators is highly sub-optimal for such functions.

\subsubsection{Convergence Rates for Ensemble Estimators under $L_{\infty}(\mu)$-norm}

This subsection presents a theorem that verifies the optimal convergence rate of histogram transform ensembles $f_{\mathrm{D},\mathrm{E}}$. Note that this result implies that our density estimator indeed benefits from ensembles.

\begin{theorem}\label{cor::ensemblerateP}
Let the histogram transform $H_n$ be defined as in \eqref{HistogramTransform} with bandwidth $h_n$ satisfying Assumption \ref{assumption::h} and $T_n$ be the number of single estimators contained in the ensembles.	Moreover, assume that $\mathcal{X} \subset B_r \subset \mathbb{R}^d$ with $r\in(0,\infty)$
and the density function $f\in C^{1,\alpha}$. 
Then for all $\tau > 0$ and $x \in B^+_{r, \sqrt{d} \cdot \overline{h}_{0,n}}$,
there exists some $N_0 \in \mathbb{N}$ such that for all $n \geq N_0$, by choosing
\begin{align*}
\underline{h}_{0,n} & := (n/\log n)^{-\frac{1}{2 (1+\alpha) + d}},
\\
T_n & := n^{\frac{2 \alpha}{2 (1+\alpha) + d}},
\end{align*}
there holds
\begin{align}\label{RatesLInftyConeEnsemble}
\|f_{\mathrm{D},\mathrm{E}} - f\|_{L_{\infty}(\mu)}
\lesssim   (\log n/n)^{\frac{1+\alpha}{2(1+\alpha)+d}}
\end{align}	
in the sense of $L_2(\mathrm{P}_H)$
with probability $\mathrm{P}^n$ at least $1 - e^{-\tau}$.
\end{theorem}

\subsubsection{Lower Bound of Single Estimators under $L_{\infty}(\mu)$-norm}

As mentioned at the beginning of this subsection, we now present a theorem to illustrate the benefit of ensembles
over single estimators.

\begin{theorem} \label{thm::LowerBoundSingles}
Let the histogram transform $H_n$ be defined as in \eqref{HistogramTransform} with bandwidth $h_n$ satisfying Assumption \ref{assumption::h}. Furthermore, let the density function $f \in C^{1,\alpha}$. For fixed constants $\underline{c}'_f, \underline{c}_f, \overline{c}_f \in (0, \infty)$, let $\mathcal{A}_f$ denote the set 
\begin{align} \label{DegenerateSetF}
\mathcal{A}_f := 
\biggl\{ x \in \mathbb{R}^d : 
\biggl| \frac{\partial f(x)}{\partial x_i} \biggr| \geq \underline{c}'_f \text{ for all } i = 1, \ldots, d  
\text{ and }
f(x) \in [\underline{c}_f, \overline{c}_f]\biggr\}.
\end{align}
Then for all $x \in B_{r, \sqrt{d} \cdot \overline{h}_0}^+ \cap \mathcal{A}_f$ and all $n > N_0$ with
\begin{align} \label{MinimalNumber}
N_0 := \min \biggl\{
n \in \mathbb{N} :
\overline{h}_{0,n} \leq  
\min \biggl\{ 
\biggl( \frac{\sqrt{d} \underline{c}'_f c_{0,n}}{4 \sqrt{3} c_L} \biggr)^{\frac{1}{\alpha}},
\biggl( \frac{d \sqrt{d}}{2} \biggr)^{\frac{1}{\alpha}},
\frac{\underline{c}_f}{2 d \sqrt{d} c_L},
\biggl( \frac{1}{4 \overline{c}_f} \biggr)^{\frac{1}{d}}
\biggr\} \biggr\},
\end{align}
by choosing 
\begin{align*}
	\overline{h}_{0,n} := n^{-\frac{1}{2+d}}, 
\end{align*}
there holds
\begin{align} \label{RatesLInftyConeSingle}
\|f_{\mathrm{D},H_n} - f\|_{L_{\infty}(\mu)}
\gtrsim n^{-\frac{1}{2+d}}
\end{align}
in the sense of $L_2(\nu_n)$-norm.
\end{theorem}

Obviously, for all $\alpha \in (0,1]$, the upper bound of ensemble histogram transforms \eqref{RatesLInftyConeEnsemble}
in Theorem \ref{cor::ensemblerateP} is essentially smaller than the lower bound \eqref{RatesLInftyConeSingle} in Theorem \ref{thm::LowerBoundSingles} for single histogram transforms. This exactly illustrates the benefits for the convergence rates of ensemble predictors over single estimators.

\subsection{Comments and Discussions} \label{sec::Comments}

Through a statistical learning treatment, in this paper we investigated and explored the histogram transform ensembles for density estimation which takes full advantage of the diversity induced by random rotation transform and the ensemble nature. In this section, we present some comments and discussions on the obtained theoretical results on the consistency and convergence rates of $f_{\mathrm{D}, \mathrm{E}}$ and compares them with existing studies.

First of all, we point out that in this paper, theoretical analysis on convergence rates is conducted in the H\"{o}lder spaces $C^{0,\alpha}$ and $C^{1,\alpha}$ under different norms respectively. For the space $C^{0, \alpha}$, the universal consistency is obtained under $L_1(\mu)$-norm. Furthermore, we highlight that the convergence rate $O((\log n / n)^{\alpha/(2\alpha + d)})$ derived in Theorem \ref{theorem::ConvergenceRatesL1} and Theorem \ref{ConvergenceRatesLInfty} in the sense of $L_1(\mu)$-norm and $L_{\infty}(\mu)$-norm are both optimal up to a logarithmic factor and of strong type ``with high probability $\nu_n$'' due to the use of the Bernstein's inequality that takes into account the variance information of the random variables. However, we fail to show the benefits of histogram transform ensembles $f_{\mathrm{P}, \mathrm{E}}$ over single histogram transform density estimators $f_{\mathrm{P}, H}$ theoretically. Therefore, we turn to the subspace $C^{1, \alpha}$ which is confined to a class of smoother functions and prove that in the space $C^{1,\alpha}$ ensemble estimators $f_{\mathrm{P}, \mathrm{E}}$ converges in type of ``in expectation w.r.t.~$\mathrm{P}_H$ and with high probability $\mathrm{P}^n$'' faster than single estimators $f_{\mathrm{P}, H}$ in weaker type of ``in expectation w.r.t.~$\nu_n$''. More precisely, Theorem \ref{cor::ensemblerateP} shows that the histogram transform ensembles $f_{\mathrm{P}, \mathrm{E}}$ can attain the almost optimal rate $O((n/\log n)^{-(1+\alpha)/(2(1+\alpha)+d)})$ whereas Theorem \ref{thm::LowerBoundSingles} shows that a single histogram transform density estimator fails to achieve this rate whose lower bound is of the order $O(n^{-1/(2+d)})$.

Recall that our study focuses on histogram partitions by applying random histogram transform ensembles to address the density estimation problem, see \cite{Ezequiel2013HT}. As a result, comparisons on consistency and convergence rates with other existing theoretical studies of histogram density estimators should be conducted. First of all, \cite{lugosi1996consistency}  investigated histogram density estimations based on data-dependent partitions and established strong consistency in the sense of $L_1(\mu)$-norm under general sufficient conditions which turns out to be stronger than ours. Concerning with the convergence rate of the histogram-based density estimators, \cite{klemela2009multivariate}  presents a multivariate histograms based on data-dependent partitions which are obtained by minimizing a complexity-penalized error criterion. The convergence rates obtained in his study are of the type $O((\log n/n)^{2\sigma/(2\sigma + 1)})$ with respect to the weaker expected $L_2$-norm under the assumption that function belongs to an anisotropic Besov class. By analyzing the proposed sieve maximum likelihood density estimator in \cite{liu2014multivariate} and further developed by a Bayesian method in \cite{liu2015multivariate}, the concentration rate of the type $O(n^{-r/(2r+1)})$ up to a logarithmic factor which does not directly depend on the dimension. However, note that the satisfying dimension-free conclusion heavily relies on the strong assumptions imposed on the approximation errors which is discussed in detail and given the upper bound in our study. \cite{Density2014} presented an algorithm taking advantage of discrepancy, a concept originates from Quasi Monte Carlo analysis, to control the partition process and a piecewise constant estimator defined on a binary partition. The resulting algorithm is proved to have the optimal convergence rate of order $O(n^{-1/2})$ in certain Monte Carlo sense under more restrictive constraints on regions containing larger proportion of the sample. Moreover, there exist a flurry of studies in the literature exploring other density estimation methods rather than histogram estimators. Several representative examples are listed below: under the same assumptions on the $\alpha$-H\"{o}lder continuity of $f$, the convergence rate for the kernel density estimator derived in \cite{jiang2017fog} is as fast as ours as well. \cite{Liu2011Forest} formed kernel density estimates of bivariate and univariate marginals supposed in the $\alpha$-H\"{o}lder class which turns out to be stronger assumptions than ours, and applied Kruskal’s algorithm to estimate the optimal forest on held out data. Formally speaking, the convergence rate 
$O ((\log n)^{1/2} ( (k^* + \widehat{k}) n^{-\beta / (2+2\beta)} + n^{- \beta / (1+2\beta)}))$ in probability are further verified.

\section{Error Analysis} \label{sec::error_analysis}

In this section, we conduct error analysis for the single and ensemble density estimators $f_{\mathrm{D}, H}$ and $f_{\mathrm{D}, \mathrm{E}}$ in the H\"{o}lder spaces $C^{0,\alpha}$ and $C^{1,\alpha}$.

\subsection{Analysis in the Space $C^{0,\alpha}$}\label{sec::subsec::C0}

By introducing the function $f_{\mathrm{P},H}$ in \eqref{equ::RHdensity}, obviously there holds
\begin{align}\label{equ::L1}
\| f_{\mathrm{D},H} - f \|_{L_1(\mu)} 
\leq \| f_{\mathrm{D},H} - f_{\mathrm{P},H} \|_{L_1(\mu)} 
      + \| f_{\mathrm{P},H} - f \|_{L_1(\mu)}.
\end{align}
The consistency and convergence analysis in our study will be mainly conducted in the sense of $L_1(\mu)$-norm with the help of inequality \eqref{equ::RHdensity}. Besides, for some specific case, i.e., when the density $f$ is compactly supported, we are also concerned with the consistency and convergence of $f_{\mathrm{D},H}$ to $f$ under $L_{\infty}(\mu)$-norm. In this case, there also holds the following inequality
\begin{align}\label{equ::Linfty}
\| f_{\mathrm{D},H} - f \|_{L_{\infty}(\mu)}
 \leq \| f_{\mathrm{D},H} - f_{\mathrm{P},H} \|_{L_{\infty}(\mu)} 
       + \| f_{\mathrm{P},H} - f \|_{L_{\infty}(\mu)}.
\end{align}

It is easy to see that the first error term on the right-hand side of \eqref{equ::L1} or \eqref{equ::Linfty} is data-dependent due to the empirical measure $\mathrm{D}$ while the second one is deterministic because of it's sampling-free. Loosely speaking, the first error term corresponds to the variance of the estimator $f_{\mathrm{D},H}$, while the second one can be treated as its bias although \eqref{equ::L1} or \eqref{equ::Linfty} is not an exact error decomposition. In our study, we proceed with the consistency and convergence analysis on $f_{\mathrm{D},H}$ by bounding the two error terms, respectively.

\subsubsection{Bounding the Approximation Error}\label{subsubsec::AppError0}

Our first theoretical result on bounding the approximation error term shows that, the $L_1(\mu)$-distance between $f_{\mathrm{P},H}$ and $f$ can be arbitrarily small by choosing the bandwidth appropriately which leads to the result of universal consistency. Moreover, under the assumption $f \in C^{0,\alpha}$, the $L_{\infty}(\mu)$-distance between the two quantities is shown to decrease polynomially. As a result, convergence rates under $L_{\infty}(\mu)$-norm can be derived.

\begin{proposition} \label{ApproximationError}
Let the histogram transform $H$ be defined as in \eqref{HistogramTransform} with bandwidth $h$ satisfying Assumption \ref{assumption::h}.
\begin{itemize}
    \item[(i)]
For any $\varepsilon > 0$, there exists $h_{\varepsilon} > 0$ such that for $\overline{h}_0  \leq h_{\varepsilon}$, there holds
\begin{align*}
\|f_{\mathrm{P},H} - f\|_{L_1(\mu)} \leq \varepsilon.
\end{align*}
    \item[(ii)]
If $f \in C^{0,\alpha}$, then for $\overline{h}_0 \leq (\varepsilon/c_L)^{1/\alpha}/d$ with $c_L$ as in Definition \ref{def::Cp}, there holds
\begin{align*}
\|f_{\mathrm{P},H} - f\|_{L_{\infty}(\mu)} \leq \varepsilon.
\end{align*}
\end{itemize}
\end{proposition}

We now show that the $L_1(\mu)$-distance between $f_{\mathrm{P},H}$ and $f$ can be upper bounded by their difference on a compact domain $B_r$ together with their difference outside the domain which is the crucial for obtaining the convergence rates under $L_1(\mu)$-norm.

\begin{proposition} \label{L1LInftyRelation}
Let the histogram transform $H$ be defined as in \eqref{HistogramTransform}. Then, for all $r \geq 1$, we have
\begin{align*}
\|f_{\mathrm{P},H} - f\|_{L_1(\mu)} 
\lesssim 2^d r^d \|(f_{\mathrm{P},H} - f) \eins_{B_r}\|_{L_{\infty}(\mu)} + 2 \mathrm{P}(B_r^c).
\end{align*}
\end{proposition}

\subsubsection{Bounding the Estimation Error} \label{subsubsction::VCDimension}

Recall that for any histogram transform $H$, the set $\pi'_H:= \{ A_j \}_{j \in \mathcal{I}_H\cup \{0\}} $ forms a partition of $\mathbb{R}^d$. The following lemma shows that both of the $\|\cdot\|_{L_1(\mu)}$ and $\|\cdot\|_{L_{\infty}(\mu)}$-distance between $f_{\mathrm{D},H}$ and $f_{\mathrm{P},H}$ can be estimated by the quantities $|\mathbb{E}_{\mathrm{D}} \eins_{A_j} - \mathbb{E}_{\mathrm{P}} \eins_{A_j}|$.

\begin{lemma} \label{FundamentalLemma}
Let the histogram transform $H$ be defined as in \eqref{HistogramTransform}. Then the following equalities hold:
\begin{itemize}
    \item[(i)]
$\displaystyle \|f_{\mathrm{D},H} - f_{\mathrm{P},H}\|_{L_1(\mu)} = \sum_{j\in \mathcal{I}_H \cup \{0\}} |\mathbb{E}_{\mathrm{D}} \eins_{A_j} - \mathbb{E}_{\mathrm{P}} \eins_{A_j}|$.
    \item[(ii)]
$\displaystyle \|f_{\mathrm{D},H} - f_{\mathrm{P},H}\|_{L_{\infty}(\mu)} = \sup_{j \in \mathcal{I}_H \cup \{0\}} \frac{|\mathbb{E}_{\mathrm{D}} \eins_{A_j} - \mathbb{E}_{\mathrm{P}} \eins_{A_j}|}{\mu(A_j)}$.
\end{itemize}
\end{lemma}

Now, there is a need for some constraints on the complexity of the function set so that the set will have a finite VC dimension \citep{vapnik71a}, and therefore make the algorithm PAC learnable \citep{valiant84a}, see e.g., \cite[Definition 3.6.1]{Gine2016Infinitedimension}.

\begin{definition}[VC dimension] \label{def::VCdimension}
Let $\mathcal{B}$ be a class of subsets of $\mathcal{X}$ and $A \subset \mathcal{X}$ be a finite set. The trace of $\mathcal{B}$ on $A$ is defined by $\{ B \cap A : B \in  \mathcal{B} \}$. Its cardinality is denoted by $\Delta^{\mathcal{B}}(A)$. We say that $\mathcal{B}$ shatters $A$ if $\Delta^{\mathcal{B}}(A) = 2^{\#(A)}$, that is, if for every $A' \subset A$, there exists a $B \subset \mathcal{B}$ such that $A' = B \cap A$. For $k \in \mathbb{N}$, let 
\begin{align*}
m^{\mathcal{B}}(k)  := \sup_{A \subset \mathcal{X}, \, \#(A) = k}  \Delta^{\mathcal{B}}(A).
\end{align*} 
Then, the set $\mathcal{B}$ is a Vapnik-Chervonenkis class if there exists $k < \infty$ such that $m^{\mathcal{B}}(k) < 2^k$ and the minimal of such $k$ is called the \emph{VC dimension} of $\mathcal{B}$, and abbreviated as $\mathrm{VC}(\mathcal{B})$.
\end{definition}

Recall that $H$ is a histogram transform, $\pi_H := (A_j)_{j\in \mathcal{I}_H}$ is a partition of $B_r$ with the index set $\mathcal{I}_H$ induced by $H$, and $\Pi_H$ is the gathering of all partitions $\pi_H$. To bound the estimation error, we need to introduce some more notations. To this end, let $\pi_h$ denote the collection of all cells in $\pi_H$, that is, 
\begin{align}  \label{WidetildeMathcalB}
\pi_h := \{ A_j :  A_j \in \pi_H \subset \Pi_H \}. 
\end{align} 
Moreover, we define
\begin{align} \label{Bh}
\Pi_h := \biggl\{ B : B = \bigcup_{j \in I} A_j,  I \subset \mathcal{I}_H, A_j \in \pi_H \subset \Pi_H \biggr\}.
\end{align}

\begin{lemma} \label{VCindex}
Let $\pi_h$ and $\Pi_h$ be defined as in \eqref{WidetildeMathcalB} and \eqref{Bh}, respectively. Then we have
\begin{align*} 
\mathrm{VC}(\pi_h) \leq 2^d + 2
\end{align*} 
and
\begin{align} \label{VCMathcalBh}
\mathrm{VC}(\Pi_h)
\leq \bigl( d (2^d - 1) + 2 \bigr) \biggl( \frac{2 r \sqrt{d}}{\underline{h}_0} + 1 \biggr)^d.
\end{align}
\end{lemma}

To bound the capacity of an infinite function set, we need to introduce the following fundamental descriptions which enables an approximation by finite subsets, see e.g. \cite[Definition 6.19]{StCh08}.

\begin{definition}[Covering Numbers]
Let $(X, d)$ be a metric space, $A \subset X$ and $\varepsilon > 0$. We call $A' \subset A$ an $\varepsilon$-net of $A$ if for all $x \in A$ there exists an $x' \in A'$ such that $d(x, x') \leq \varepsilon$. Moreover, the $\varepsilon$-covering number of $A$ is defined as
\begin{align*}
\mathcal{N}(A, d, \varepsilon)
= \inf \biggl\{ n \geq 1 : \exists x_1, \ldots, x_n \in X \text{ such that } A \subset \bigcup_{i=1}^n B_d(x_i, \varepsilon) \biggr\},
\end{align*}
where $B_d(x, \varepsilon)$ denotes the closed ball in $X$ centered at $x$ with radius $\varepsilon$.
\end{definition}

Let $\mathcal{B}$ be a class of subsets of $\mathcal{X}$, denote $\eins_{\mathcal{B}}$ as the collection of the indicator functions of all $B \in \mathcal{B}$, that is, $\eins_{\mathcal{B}} := \{ \eins_B : B \in \mathcal{B} \}$. Moreover, as usual, for any probability measure $\mathrm{Q}$, $L_1(\mathrm{Q})$ is denoted as the $L_1$ space with respect to $Q$ equipped with the norm $\|\cdot\|_{L_1(\mathrm{Q})}$.

\begin{lemma} \label{ScriptBhCoveringNumber}
Let $\pi_h$ and $\Pi_h$ be defined as in \eqref{WidetildeMathcalB} and \eqref{Bh}, respectively. Then, for all $0 < \varepsilon < 1$, there exists a universal constant $K$ such that for any probability measure $\mathrm{Q}$, there hold
\begin{align} \label{CollectionCoveringNumber} 
\mathcal{N}(\eins_{\pi_h}, \|\cdot\|_{L_1(\mathrm{Q})}, \varepsilon) \leq K (2^d+2) (4e)^{2^d+2} \biggl( \frac{1}{\varepsilon} \biggr)^{2^d+1} 
\end{align} 
and
\begin{align} \label{BpCoveringNumber}
\mathcal{N}(\eins_{\Pi_h}, \|\cdot\|_{L_1(\mathrm{Q})}, \varepsilon) \leq K \biggl(\frac{c_dr}{\underline{h}_0}\biggr)^d (4 e)^{(\frac{c_dr}{\underline{h}_0})^d} \biggl( \frac{1}{\varepsilon} \biggr)^{(\frac{c_dr}{\underline{h}_0})^d - 1},
\end{align}
where the constant $c_d := 2^{1+\frac{1}{d}}\cdot 3\cdot d^{\frac{1}{d}+\frac{1}{2}}$.
\end{lemma}

Now, we are able to establish oracle inequalities under $L_1(\mu)$-norm and $L_{\infty}(\mu)$-norm which can be used to derive the convergence rates.

\begin{proposition} \label{OracleInequalityL1}
Let the histogram transform $H_n$ be defined as in \eqref{HistogramTransform} with bandwidth $h_n$ satisfying Assumption \ref{assumption::h} with $\overline{h}_{0,n} \leq 1$. Then, for all $r \geq 2$ and $n \geq 1$, there holds
\begin{align*}
\|f_{\mathrm{D},H_n} - f_{\mathrm{P},H_n}\|_{L_1(\mu)}
\leq \sqrt{\frac{9\log n(1+12(c_dr/\underline{h}_{0,n})^d)}{2n}} + \frac{2 \log n(1+(c_dr/\underline{h}_{0,n})^d)}{n} + \frac{4}{n}
\end{align*}
with probability $\nu_n$  at least $1 - \frac{1}{n}$. 
\end{proposition}

\begin{proposition}\label{OracleInequalityInftyNorm}
Let the histogram transform $H_n$ be defined as in \eqref{HistogramTransform} with bandwidth $h_n$ satisfying Assumption \ref{assumption::h} with $\overline{h}_{0,n} \leq 1$. Moreover, assume that $\mathcal{X} \subset B_r \subset \mathbb{R}^d$ and the density function $f$ satisfies $\|f\|_{L_{\infty}(\mu)} < \infty$. Then for all $\tau > 0$,
all $x \in B^+_{r, \sqrt{d} \cdot \overline{h}_{0,n}}$, 
and all $n \geq N_0$ with $N_0 := \max\{e, 2K, \mu(B_r)\}$ and 
$K$ as in Lemma \ref{ScriptBhCoveringNumber}, there holds
\begin{align*}
\|f_{\mathrm{D},H_n} - f_{\mathrm{P},H_n}\|_{L_{\infty}(\mu)}
& \leq \sqrt{\frac{2 \|f\|_{L_{\infty}(\mu)} (\tau + 2^{d+4} \log n +  2^{d+2}\log(1/\underline{h}_{0,n}^d))}{n\underline{h}_{0,n}^d}}
\nonumber\\
& \phantom{=}
+ \frac{2(\tau+2^{d+4}\log n+2^{d+2}\log(1/\underline{h}_{0,n}^d))}{3n\underline{h}_{0,n}^d} + \frac{2}{n}
\end{align*}
with probability $\nu_n$ at least $1 - e^{-\tau}$.
\end{proposition}

\subsection{Analysis in the Space $C^{1,\alpha}$}\label{subsec::C1}

Recall that in the previous subsection where $f$ lies in the space $C^{0,\alpha}$, a drawback to $L_1(\mu)$-norm is that 
it does not admit an exact bias-variance decomposition and  the usual Taylor expansion involved techniques for error estimation may not apply directly. Thus, we fail to tell how ensembles exceed single estimators in a theoretical point of view, and therefore we turn to the subspace $C^{1,\alpha}$.

In this subsection, we study the convergence rates of $f_{\mathrm{D},\mathrm{E}}$ and $f_{\mathrm{D},H}$ to the true density $f \in C^{1,\alpha}$. To this end, there is a point in introducing some notations. First of all, for any $t \in \{ 1, \ldots, T \}$ and $x \in \mathbb{R}^d$, recall that the population version of $f_{\mathrm{D},H_t}$ can be formulated as
\begin{align*}
f_{\mathrm{P},H_t}(x) = \mathbb{E}_{\mathrm{P}} \bigl( f(X) | A_{H_t}(x) \bigr),
\end{align*}
where $\mathbb{E}_{\mathrm{P}}(\cdot | A_{H_t}(x))$ denotes the conditional expectation with respect to $\mathrm{P}$ on $A_{H_t}(x)$. With the ensembles of the population version
\begin{align}\label{equ::interes}
f_{\mathrm{P}, \mathrm{E}}(x) := \frac{1}{T} \sum^T_{t=1} f_{\mathrm{P},H_t}(x)
\end{align}
we make the $L_2(\nu_n)$-error decomposition
\begin{align}\label{equ::L2de}
\|f_{\mathrm{D},\mathrm{E}}(x) - f(x)\|_{L_2(\nu_n)}^2
= \|f_{\mathrm{D},\mathrm{E}}(x) - f_{\mathrm{P},\mathrm{E}}(x) \|_{L_2(\nu_n)}^2
   + \|f_{\mathrm{P},\mathrm{E}}(x) - f(x)\|_{L_2(\nu_n)}^2.
\end{align}
In our study, the consistency and convergence analysis of the histogram transform ensembles $f_{\mathrm{D},\mathrm{E}}$
in the space $C^{1,\alpha}$ will be mainly conducted with the help of the decomposition \eqref{equ::L2de}.

In particular, in the case that $T = 1$, i.e., when there is only single histogram transform density estimator, we are concerned with the lower bound of $f_{\mathrm{D}, H}$ to $f$. With the population version
\begin{align*}
f_{\mathrm{P},H}(x) = \mathbb{E}_{\mathrm{P}} ( f(X) | A_H(x) )
\end{align*}
we make the $L_2(\nu_n)$-error decomposition
\begin{align} \label{equ::L2Decomposition}
\|f_{\mathrm{D},H}(x) - f(x)\|_{L_2(\nu_n)}^2
= \|f_{\mathrm{D},H}(x) - f_{\mathrm{P},H}(x)\|_{L_2(\nu_n)}^2
   + \|f_{\mathrm{P},H}(x) - f(x)\|_{L_2(\nu_n)}^2.
\end{align}
It is important to note that both of the two terms on the right-hand side of \eqref{equ::L2de} or \eqref{equ::L2Decomposition} are data- and partition-independent due to the expectation with respect to $\mathrm{D}$ and $H$. Loosely speaking, the first error term corresponds to the expected estimation error of the estimators $f_{\mathrm{D},\mathrm{E}}$ or $f_{\mathrm{D},H}$, while the second one demonstrates the expected approximation error.  Besides, note that \eqref{equ::L2de} and \eqref{equ::L2Decomposition} are exact error decompositions, which are different from the error decompositions \eqref{equ::L1} and \eqref{equ::Linfty} conducted in Section \ref{sec::subsec::C0}.

\subsubsection{Bounds on the Approximation Error}\label{subsubsec::AppError1}

In this subsection, we firstly establish the upper bound for the approximation error term of histogram transform ensembles $f_{\mathrm{P},\mathrm{E}}$ and further find a lower bound of the approximation error for single density estimator $f_{\mathrm{P},H}$.

\begin{proposition} \label{ApproximationError::LTwo}
Let the histogram transform $H$ be defined as in \eqref{HistogramTransform} with bandwidth $h$ satisfying Assumption \ref{assumption::h} and $T$ be the number of single estimators contained in the ensembles. Moreover, let the density function satisfy $f \in C^{1,\alpha}$. Then, for all $n \geq 1$, there holds
\begin{align*}
\|f_{\mathrm{P},\mathrm{E}} - f\|_{L_{\infty}(\mu)}
\leq c_L \biggl( c_{0,n}^{-2d} \cdot \overline{h}_0^{2(1+\alpha)} + \frac{d}{T} \cdot  \overline{h}_0^2 \biggr)^{1/2}
\end{align*}
in the sense of $L_2(\mathrm{P}_H)$-norm, where the constant $c_{0,n}$ is as in Assumption \ref{assumption::h}.
\end{proposition}

\begin{proposition} \label{ApproximationError::LTwoCounter}
Let the histogram transform $H$ be defined as in \eqref{HistogramTransform} with bandwidth $h$ satisfying Assumption \ref{assumption::h}. Furthermore, let the density function $f \in C^{1,\alpha}$ and $\mathcal{A}_f$ be the set \eqref{DegenerateSetF}. Then for all $x \in B_{r, \sqrt{d} \cdot \overline{h}_0}^+ \cap \mathcal{A}_f$ and all
\begin{align*}
h_0 \leq \biggl( \frac{\sqrt{d} \underline{c}'_f c_0}{4 \sqrt{3} c_L} \biggr)^{\frac{1}{\alpha}},
\end{align*}
we have
\begin{align}\label{eq::appolower}
\|f_{\mathrm{P},H} - f\|_{L_{\infty}(\mu)}
\geq \frac{\sqrt{d}}{4} \underline{c}_f' c_0 \cdot \overline{h}_0
\end{align}
in the sense of $L_2(\mathrm{P}_H)$-norm, where the constant $c_0$ is as in Assumption \ref{assumption::h}.
\end{proposition}

\subsubsection{Bounds on the Estimation Error}\label{subsubsec::EstError1}

The first proposition analyzes the estimation error term for ensemble density estimators.

\begin{proposition} \label{OracleInequality::LOne}
Let the histogram transform $H_n$ be defined as in \eqref{HistogramTransform} with bandwidth $h_n$ satisfying Assumption \ref{assumption::h} with $\overline{h}_{0,n} \leq 1$.
Moreover, assume that $\mathcal{X} \subset B_r \subset \mathbb{R}^d$ and the density function $f$ satisfies $\|f\|_{L_{\infty}(\mu)} < \infty$.
Then 
for all $x \in B^+_{r, \sqrt{d} \cdot \overline{h}_{0,n}}$, all $\tau > 0$,
and all $n \geq N_0$ with $N_0 := \max\{e, 2K, \mu(B_r)\}$ and 
$K$ as in Lemma \ref{ScriptBhCoveringNumber},
there holds
\begin{align*}
\|f_{\mathrm{D},\mathrm{E}}-f_{\mathrm{P},\mathrm{E}}\|_{L_{\infty}(\mu)}
& \leq \sqrt{\frac{2 \|f\|_{L_{\infty}(\mu)} (\tau + 2^{d+4} \log n +  2^{d+2}\log(1/\underline{h}_{0,n}^d))}{n\underline{h}_{0,n}^d}}
\nonumber\\
& \phantom{=}
+ \frac{2(\tau+2^{d+4}\log n+2^{d+2}\log(1/\underline{h}_{0,n}^d))}{3n\underline{h}_{0,n}^d} + \frac{2}{n}
\end{align*}
with probability $\mathrm{P}^n$ at least $1 - e^{-\tau}$.
\end{proposition}

Here, we further find a lower bound of the estimation error for single density estimators.

\begin{proposition} \label{OracleInequality::LTwoCounter}
Let the histogram transform $H_n$ be defined as in \eqref{HistogramTransform} with bandwidth $h_n$ satisfying Assumption \ref{assumption::h}. Moreover, let the density function $f \in C^{1,\alpha}$ and $\mathcal{A}_f$ be the set \eqref{DegenerateSetF}. Then for all $x \in B_{r, \sqrt{d} \cdot \overline{h}_{0,n}}^+ \cap \mathcal{A}_f$ and all $n \geq N'$ with 
\begin{align} \label{MinimalNumberSample}
N' := \min \biggl\{ n \in \mathbb{N} :
\overline{h}_{0,n} \leq  
\min \biggl\{
\biggl( \frac{d \sqrt{d}}{2} \biggr)^{\frac{1}{\alpha}},
\frac{\underline{c}_f}{2 d \sqrt{d} c_L},
\biggl( \frac{1}{4 \overline{c}_f} \biggr)^{\frac{1}{d}}
\biggr\}
\biggr\},
\end{align}
there holds
\begin{align}	\label{variancetermP}
\|f_{\mathrm{D},H_n} - f_{\mathrm{P},H_n}\|_{L_{\infty}(\mu)}
\geq \sqrt{\frac{\underline{c}_f}{4 n\overline{h}_{0,n}^{d}}}
\end{align}
in the sense of $L_2(\mathrm{P}^n)$-norm.
\end{proposition}

\section{Numerical Experiments} \label{sec::numerical_experiments}

In this section, we present the computational experiments that we have carried out. Firstly, in Section \ref{sec::subsec::experisetup}, we introduce some preliminary work including a brief account for the generation process of our histogram transforms, following by two effective measures of estimation accuracy named \emph{MAE} and \emph{ANLL}. Then we consider the origin algorithm proposed by \cite{Ezequiel2013HT}, referred to as na\"{i}ve histogram transform ensembles (NHTE) which we conduct comprehensive theoretical analysis in the preceding sections. We extend the existing work to an illustrative example, which illustrates how ensemble estimators outperform the single ones in a experimental point of view in Section \ref{sec::counter}.

In Section \ref{sec::AdaptiveHTE}, we introduce an adaptive version of histogram transform ensembles based on our proposed adaptive splitting method, termed as the adaptive histogram transform ensembles (AHTE). In Section \ref{sec::subsec::syntheticdata}, we design four types of distributions, study the behavior of our AHTE depending on the values of tunable parameters, and conduct comparisons with other state-of-the-art density estimators. Finally, in Section \ref{sec::subsec::realdata}, we compare our approach with the counterparts for real data in terms of \textit{ANLL}.

\subsection{Experimental Setup} \label{sec::subsec::experisetup}

Recall that $D = \{x_1, \ldots, x_n\}$ are denoted as observations drawn independently from an unknown distribution $\mathrm{P}$ on $\mathbb{R}^d$ with the density $f$.  For the sake of of constructing data sets, we need to detail the elements of histogram transform defined as in \eqref{HistogramTransform}.

\subsubsection{Generation Process for Histogram Transforms}

Firstly, note that the random rotation matrix $R$ is generated in the manner coinciding with Section \ref{sec::HTE}. For the elements of the scaling matrix $S$, applying the well known Jeffreys prior for scale parameters referred to \cite{Jeffreys1946An}, we draw $\log (s_i)$ from the uniform distribution over certain real-valued interval $[\log(\underline{s}_0),\log(\overline{s}_0)]$ with
\begin{align*}
	\log (\underline{s}_0) & := s_{\min} +\log (\widehat{s}),
\\
	\log (\overline{s}_0) & := s_{\max} +\log (\widehat{s}),
\end{align*}
where $s_{\min}, s_{\max} \in \mathbb{R}$ are tunable parameters with $s_{\min} < s_{\max}$ and the scale parameter $\widehat{s}$ is the inverse of the bin width $\widehat{h}$ measured on the input space, which is defined by
\begin{align*}
\widehat{s} := (\widehat{h})^{-1} = (3.5 \sigma)^{-1} n^{\frac{1}{2+d}}.
\end{align*}
Here, the standard deviation $\sigma := \sqrt{\mathrm{trace}(V)/d}$ with $V := \frac{1}{n-1} \sum_{i=1}^n (x_i - \bar{x})(x_i-\bar{x})^{\top}$ and $\bar{x} := \frac{1}{n} \sum_{i=1}^n x_i$ combines the information from all the dimensions of the input space.

\subsubsection{Performance Evaluation Criteria}

When it comes to the empirical performances for various different density estimators $\widehat{f}$, we introduce two comparisons of options conducted over $m$ test samples $\{x_j, j=1,\ldots,m\}$. In the case where the true density function $f$ is known, by convention, we adopt the following Mean Absolute Error (\emph{MAE}):
\begin{align*}
\emph{\text{MAE}} (\widehat{f}) = \frac{1}{m} \sum_{j = 1}^m |\widehat{f}(x_j) - f(x_j)|.
\end{align*}
Obviously, lower \emph{MAE} implies better performance of a density estimator $\widehat{f}$.

Taking the real data application into consideration, namely, when the true density $f$ is unknown, we turn to the measure of accuracy given by the Average Negative Log-Likelihood (\emph{ANLL}):
\begin{align*}
\emph{\text{ANLL}}(\widehat{f}) = -\frac{1}{m}\sum_{j=1}^m\log \widehat{f}(x_j),
\end{align*}
where $\widehat{f}(x_j)$ represents the estimated probability density for the test sample $x_j$ and the lower the \emph{ANLL} is, the better estimation we obtain. However, the case $\widehat{f}(x) = 0$ would imply $\emph{\text{ANLL}}(\widehat{f}) = \infty$, an undesirable state. We prefer to avoid this issue altogether by substituting all density estimation $\widehat{f}(x)$ with $\widehat{f}(x) + \varepsilon$, where $\varepsilon$ is a infinitesimal number which can be obtained by function {\tt numpy.spacing(1)} in Python. Last, note that the test set ($m$ samples) must be disjoint with the training set ($n$ samples) in all cases.

\subsection{Comparisons between Ensemble and Single Estimators}\label{sec::counter}

In order to give a more comprehensive understanding of this section, the reader will be reminded of the significance to illustrate the benefits of our histogram transform ensembles over single estimator. Therefore, we start this simulation by constructing the above mentioned illustrative example as the synthetic data. To be specific, we base the simulations on one particular distribution construction approach generating a 2-dimensional toy example, where the density function $f : \mathcal{X}^2 \to \mathbb{R}$ is defined by
\begin{align*}
f = \mathrm{Beta}(3,10).
\end{align*}

We emphasize that the bell-shaped density function $f$ is compactly supported and bounded.

In experiment, we generate $1,000$ testing samples, and let the number of training samples vary from $10$ to $1000$. We set hyper-parameters $s_{\min} = 0 $ and $s_{\max} = 1$. Each experiment is repeated  for $50$ times and the average \textit{MAE} is reported as a representative of testing accuracy.

\begin{figure}[H]
\begin{minipage}[t]{0.8\textwidth}  
\centering  
\includegraphics[width=\textwidth]{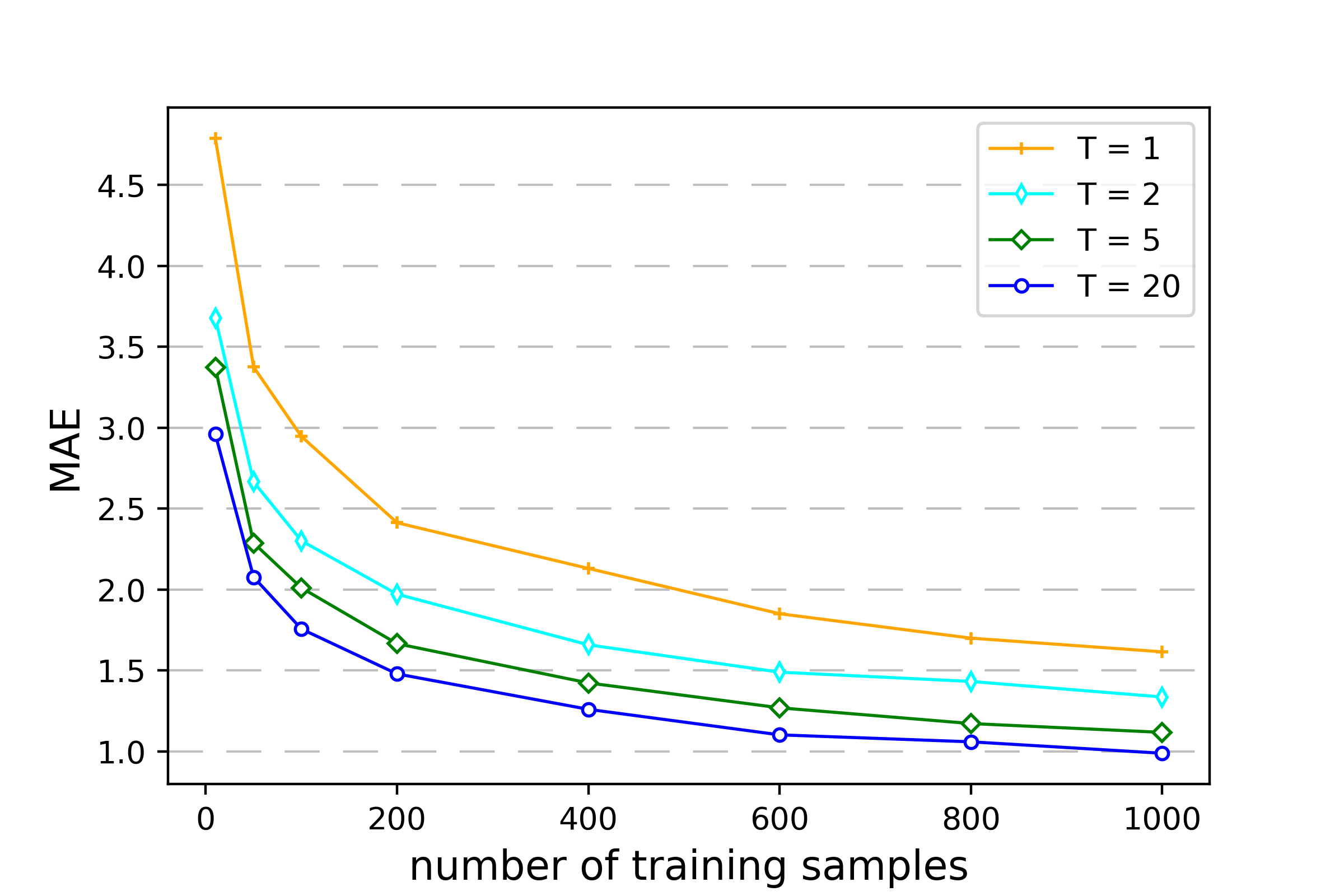}  
\end{minipage}  
\centering  
\caption{\emph{MAE} for different numbers of histogram transforms $T$ applied for the synthetic data.}
\label{fig:ANLLN}
\end{figure}

Figure \ref{fig:ANLLN} captures the \emph{MAE} performance of our model for $T = 1, 2, 5, 20$, respectively. The result is twofold: First of all, lower \textit{ANLL} of the steady state for $T > 1$ states that ensembles behave better than single estimators in terms of accuracy. Moreover, the difference of slope before the curves reach flat illustrates the lower bound of the convergence rate of single estimators to some extent.

\subsection{Adaptive HTE for Density Estimation}\label{sec::AdaptiveHTE}

Till now, the partition processes considered have only performed in an equal-size histogram manner, however, it fails to take into account information resides in data itself, and thus lacks of local adaptivity. Therefore, in the following experiments, we extend the origin histogram transform ensembles, which we will then refer to as the na\"{i}ve histogram transform ensembles (NHTE), to an adaptive version, the adaptive histogram transform ensembles (AHTE). That is, all histogram transforms in the following experiments adopt the adaptive random stretching criterion to significantly improve the balancing property of splits and hence to increase the accuracy.

\subsubsection{Adaptive Splitting}

The \textit{adaptive splitting} technique helps formulate a data dependent partition. In general, adaptive histogram transforms split more on sample-dense areas while split less on sample-sparse areas, for we adopt the recursive-based splittings, that is, the splitting process ceases with the number of samples containing in a cell less than a certain amount, denoted by the hyper-parameter \textit{min\_samples\_split}. However, it's worth pointing out that instead of selecting the bin indices as the round points, where each cell shares the same size, this adaptive method aims at diversifying the samples contained in each cell. In fact, high-dimensional samples are typically sparse. Therefore, it is of great necessity to recognize the area where samples densely cluster and to ``cut it out''. Otherwise, the high volume of the sample-sparse areas will significantly lower the estimated density in the whole cell.

Moreover, we impose a stopping criterion when a cell contains less than $m$ samples. Then we focus on every \textit{qualified} cell with enough sample points, and select the to-be-split dimension as the one with the largest ratio, that is, with the least variance, scaled to $[0.5,2.5]$, per range. In the process of split point selection, we choose the split point as a relatively far position from the sample mean. To be specific, if the sample mean is smaller than the $0.6$ quantile, the split point is selected as the $0.618$ quantile, otherwise as the $0.382$ quantile of the sample points.

By this means, we're able to make full use of the potential information containing in samples and to split at the area with as sparse points as possible. Thus, the sample-dense area will not be split too much, and therefore the local estimation of the underlying density function can be guaranteed. A concrete description of the construction process of \textit{adaptive splitting} is shown in the following Algorithm \ref{alg::adaptiveSplitting}.

\begin{algorithm}[htbp]
%	\SetAlgoNoLine
\caption{Adaptive Splitting}
\label{alg::adaptiveSplitting}
\KwIn{
Transformed sample space $D^\top$ 
%		with dimension $d$
;
\\
\quad\quad\quad\quad Minimal number of samples required to split  $m$; \\
\quad\quad\quad\quad Number of splits $p$ initiated as 1.
\\
}
\Repeat{$\max(\textit{number of samples in all cells}) \leq m$}{
$k_t^p$ is the number of cells before the $p$-th split for the $t$-th partition;	
\\	
\For{$j =1 \to k_t^p$}{
	\If{$\textit{number of samples in the j-th cell} > m$}{
		\For{$i = 1 \to d$}{
			$range_i$ denotes the range of dimension $i$; \\
			$var_i$ denotes the scaled variance of dimension $i$; \\
			$ratio_i = \frac{range_i}{var_i}$.	
		}	
		The to-be-split dimension is with the largest ratio. \\
		\eIf{mean $\leq$ $0.6$ quantile}{
			Select the split point as the $0.618$ quantile of this dimension;
		}
		{
			Select the split point as the $0.382$ quantile of this dimension;
		}
	}
	$p++$. \\
}
}
\KwOut{Adaptive partition of the transformed sample space.}
\end{algorithm}

\subsubsection{Adaptive HTE}

To achieve the adaptive version of HTE, we simply apply \textit{adaptive splittings} instead of data independent partitions, so that we are able to take into account more information of samples.

\begin{algorithm}[htbp]
%	\SetAlgoNoLine
\caption{Adaptive HTE for Density Estimation}
\label{alg::AHTE}
\KwIn{
Training data $D:=(X_1, \ldots, X_n)$;
\\
\quad\quad\quad\quad Number of histogram transforms $T$.
} 
\For{$t =1 \to T$}{

%		\For{$i = 1 \to M$}{
Generate random affine transform matrix
$H_t=R_t$;\\
Apply \textit{adaptive splitting} to the transformed sample space; \\
Apply piecewise constant density estimators to each cell; \\ 
Compute the histogram density estimator $f_{\mathrm{D},H_t}(x)$ induced by $H_t$.
%		}
%		Select the best mapping $f_{\mathrm{D},H_t}(x)$ with the minimal error.
}
\KwOut{The histogram transform ensemble for density estimation is
%		\begin{spacing}{0.55}
\begin{align*}
	f_{\mathrm{D},\mathrm{E}}(x)
	= \frac{1}{T} \sum_{t=1}^T f_{\mathrm{D},H_t}(x).
\end{align*}
%		\end{spacing}
}
\end{algorithm}

Note that in the adaptive version of HTE, the random affine transform matrix $H_t$ only consists of random rotations. However, this doesn't mean that we abandon  stretchings or translations. As a matter of fact, instead of choosing bin indices, we split differently on different local areas. Thus the \textit{adaptive splittings} can be viewed as a combination of \textit{adaptive stretchings} and \textit{adaptive translations}, with each local area stretched and translated to a different extent according to the sample distribution, only the effect of theses two transforms is achieved by adaptively selecting the split points. Also, the randomness of this adaptive partition is solely provided by random rotation transforms, with the implied stretchings and translations determined by sample distributions.

\subsection{Synthetic Data Analysis} \label{sec::subsec::syntheticdata}

\subsubsection{Synthetic Data Settings}

In this subsection, we start by applying our histogram transform ensembles and other density estimation methods including na\"{i}ve histogram transforms and KDE on several artificial examples. To be specific, we base the simulations on four different types of distribution construction approaches with dimension $d = 2, \ 5$, respectively. More precisely, we assume that the sample vector $X = (X_1, \ldots , X_d)$ are independent with each other and with distributions as follows:
\begin{itemize}
    \item
\textbf{Type I:} 
$f_{\mathrm{I}} = 0.3 * \mathrm{Unif}(0.7,1) + 0.7 * \mathrm{Unif}(0,0.4)$ with $X_i \sim f_{\mathrm{I}}$, for $i=1, \dots, d$;
    \item
\textbf{Type II:} 
$f_{\mathrm{II}} = 0.5 * \mathrm{Beta}(2,10) + 0.5 * \mathrm{Unif}(0.5,1)$ with $X_i \sim f_{\mathrm{II}}$, for $i=1, \dots, d$;
    \item
\textbf{Type III:} 
$f_{\mathrm{III}} = 0.5 * \mathrm{Laplace}(0,\frac{1}{2}) + 0.5 * \mathrm{Unif}(2,4)$ with $X_i \sim f_{\mathrm{III}}$, for $i=1, \dots, d$;
    \item
\textbf{Type IV:} 
$f_{\mathrm{IV}} = \prod_{i=1}^{d-1} f(X_i) \cdot g(X_d)$ with $f = \mathrm{Exp}(0.5)$ and $g = \mathrm{Unif}(0,5)$.
\end{itemize}

The above artificial examples represent different kinds of distributions. The first type $f_{\mathrm{I}}$ is piecewise constant, and the following two, $f_{\mathrm{II}}$ and $f_{\mathrm{III}}$ represent distributions with sectionally continuous probability distribution functions. Note that in the first three distributions, marginal distributions of each dimension are identically distributed, while in the last example $f_{\mathrm{IV}}$, $X_d$ obeys a different distribution from the first $d-1$ dimensions.

\subsubsection{Study of the Parameters}\label{sec::subsec::parameter}
In this subsection, we lubricate the study of parameters $T$ and $m$ in Algorithm \ref{alg::adaptiveSplitting} among different datasets and distinct dimensions. As it mentions above, $T$ stands for the number of partitions in an ensemble and $m$ represents the minimum number of samples required to split an internal node.

We conduct experiments based on synthetic datasets following four types of distributions defined as $f_{\mathrm{I}}$, $f_{\mathrm{II}}$, $f_{\mathrm{III}}$ and $f_{\mathrm{IV}}$. For each distribution, we independently take $2000$ and $10000$ samples on $d=2,\ 5$ as training data and test data respectively. In addition, we have carried out experiments with $T\in\{5,20,100\}$ and $m$ varying respectively from $1$ to $30$ and $1$ to $50$ for \textit{ANLL} and \textit{MAE}, the measurements of accuracy under each parameters combination.

\begin{figure}[htbp]
\centering
{
\begin{minipage}{0.450\textwidth}
	\centering
	\includegraphics[width=\textwidth]{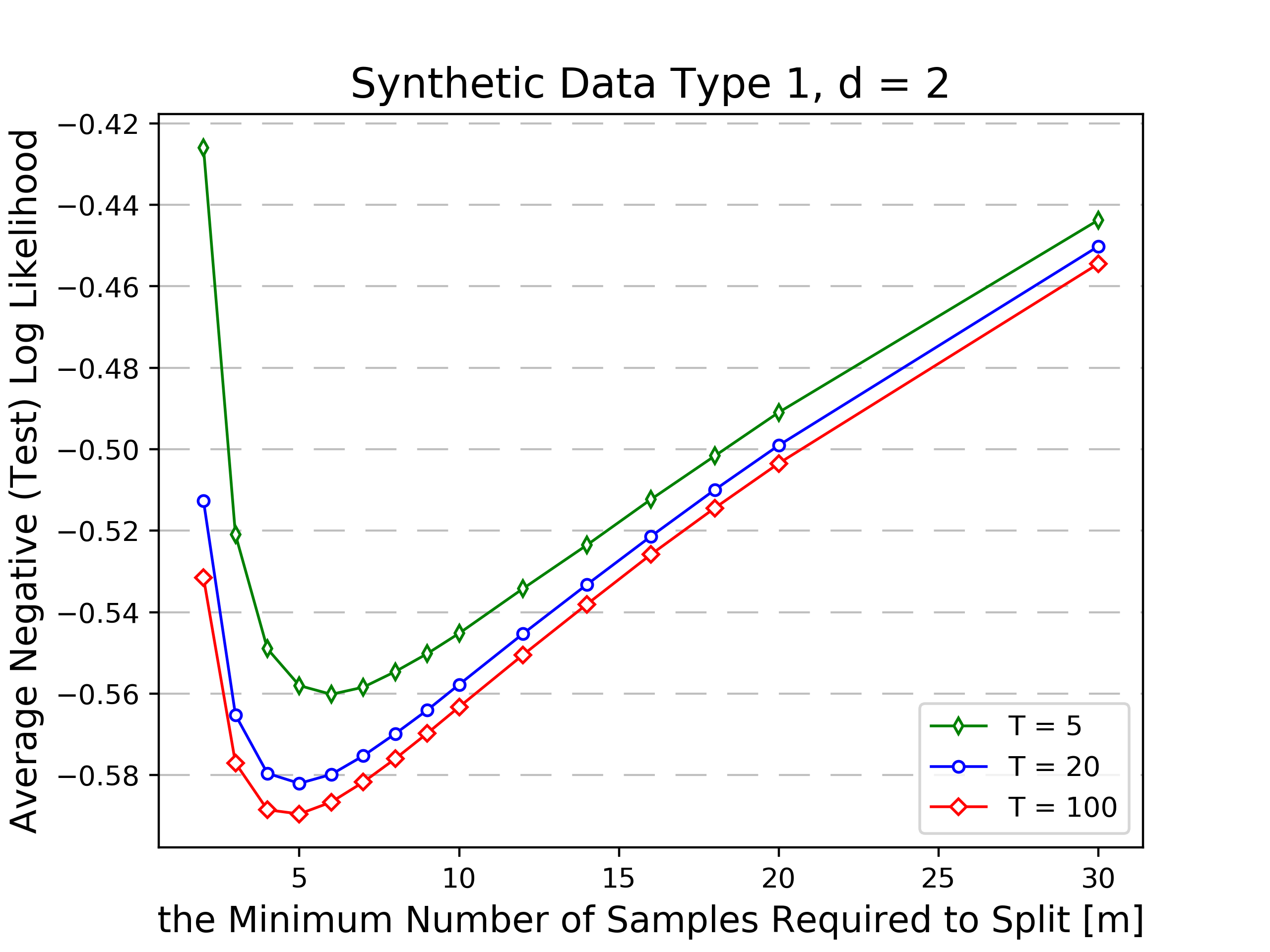}
\end{minipage}
}
{
\begin{minipage}{0.450\textwidth}
	\centering
	\includegraphics[width=\textwidth]{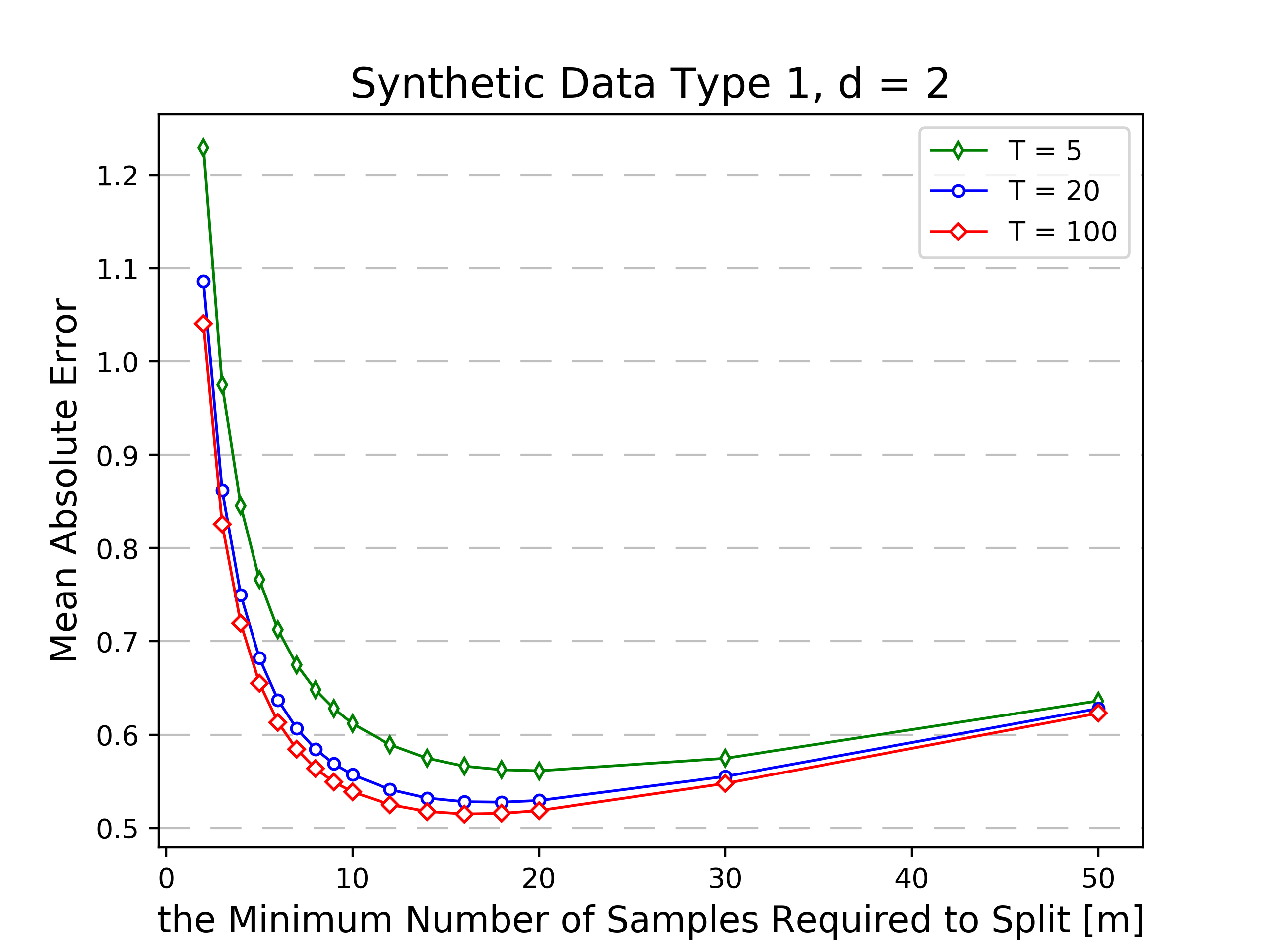}
\end{minipage}
}
{
\begin{minipage}{0.450\textwidth}
	\centering
	\includegraphics[width=\textwidth]{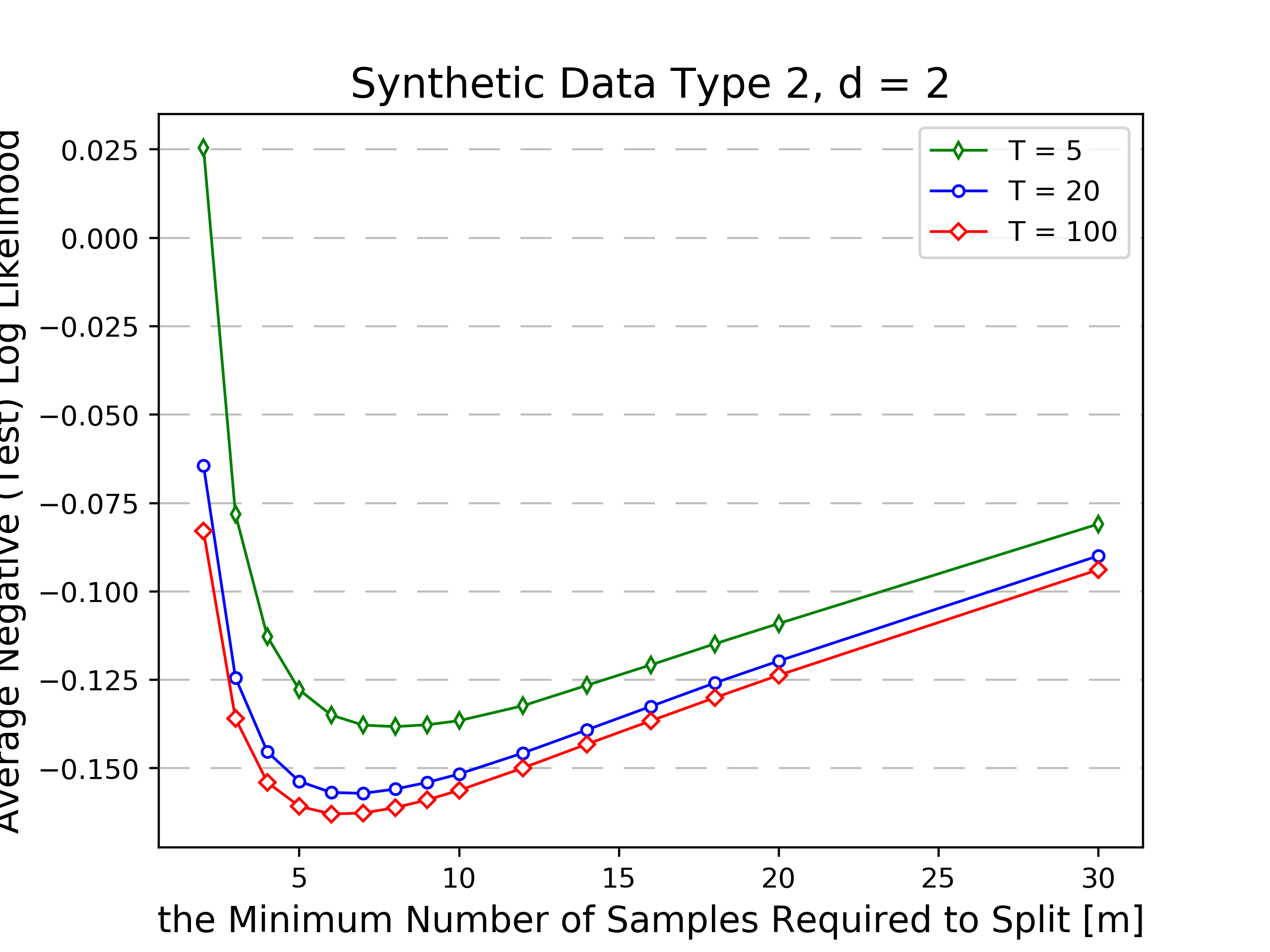}
\end{minipage}
}
{
\begin{minipage}{0.450\textwidth}
	\centering
	\includegraphics[width=\textwidth]{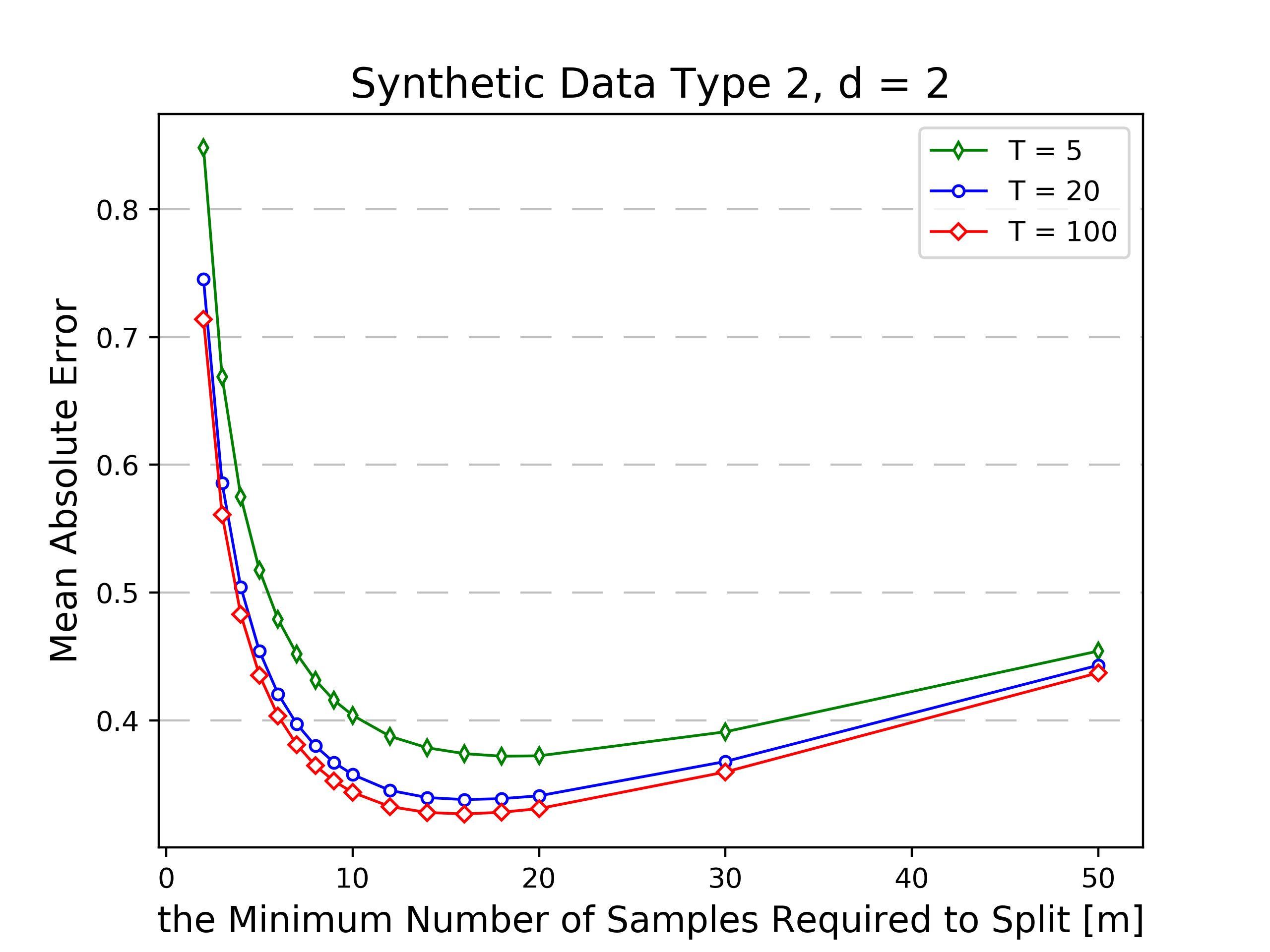}
\end{minipage}
}
{
\begin{minipage}{0.450\textwidth}
	\centering
	\includegraphics[width=\textwidth]{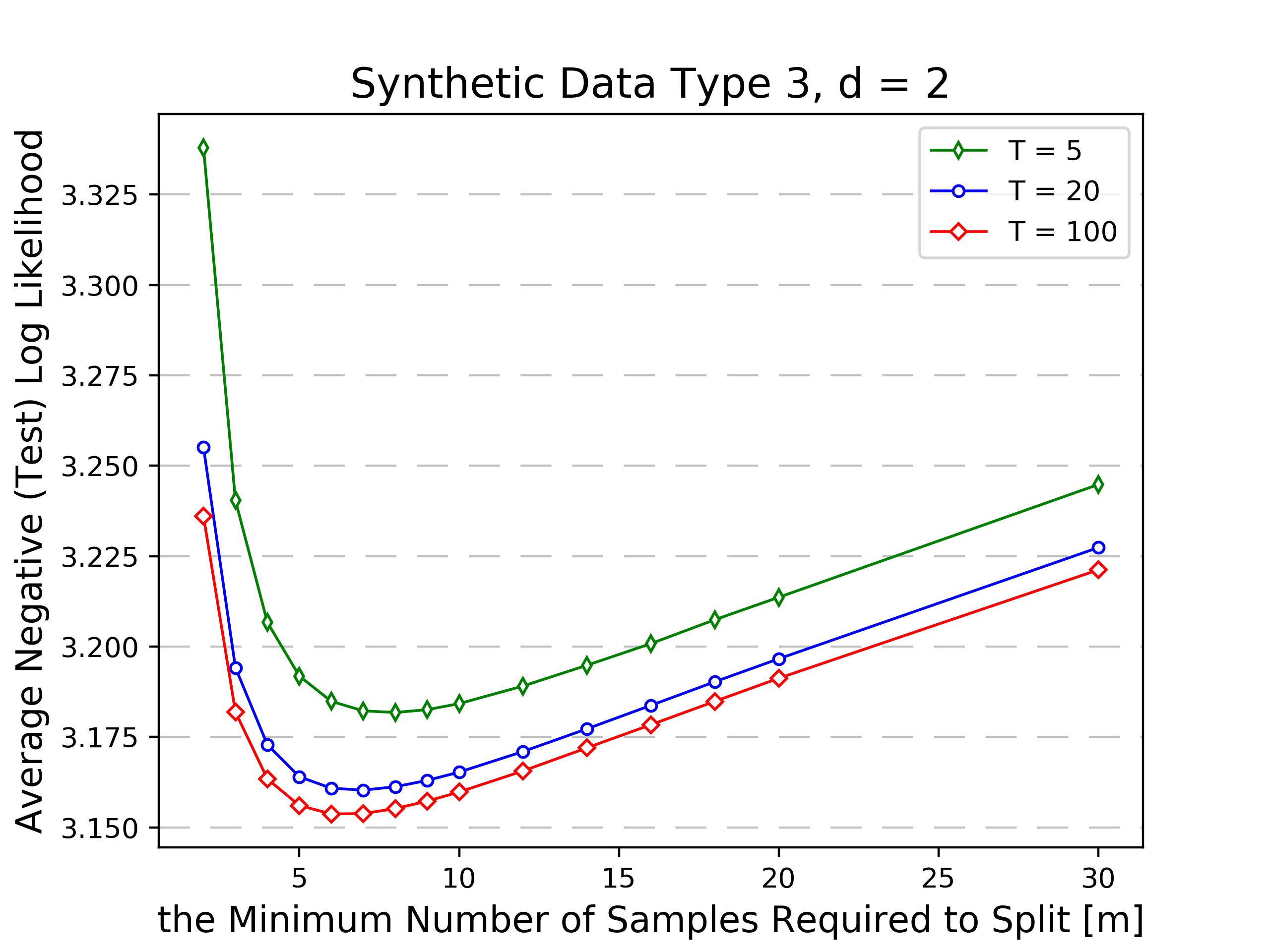}
\end{minipage}
}
{
\begin{minipage}{0.450\textwidth}
	\centering
	\includegraphics[width=\textwidth]{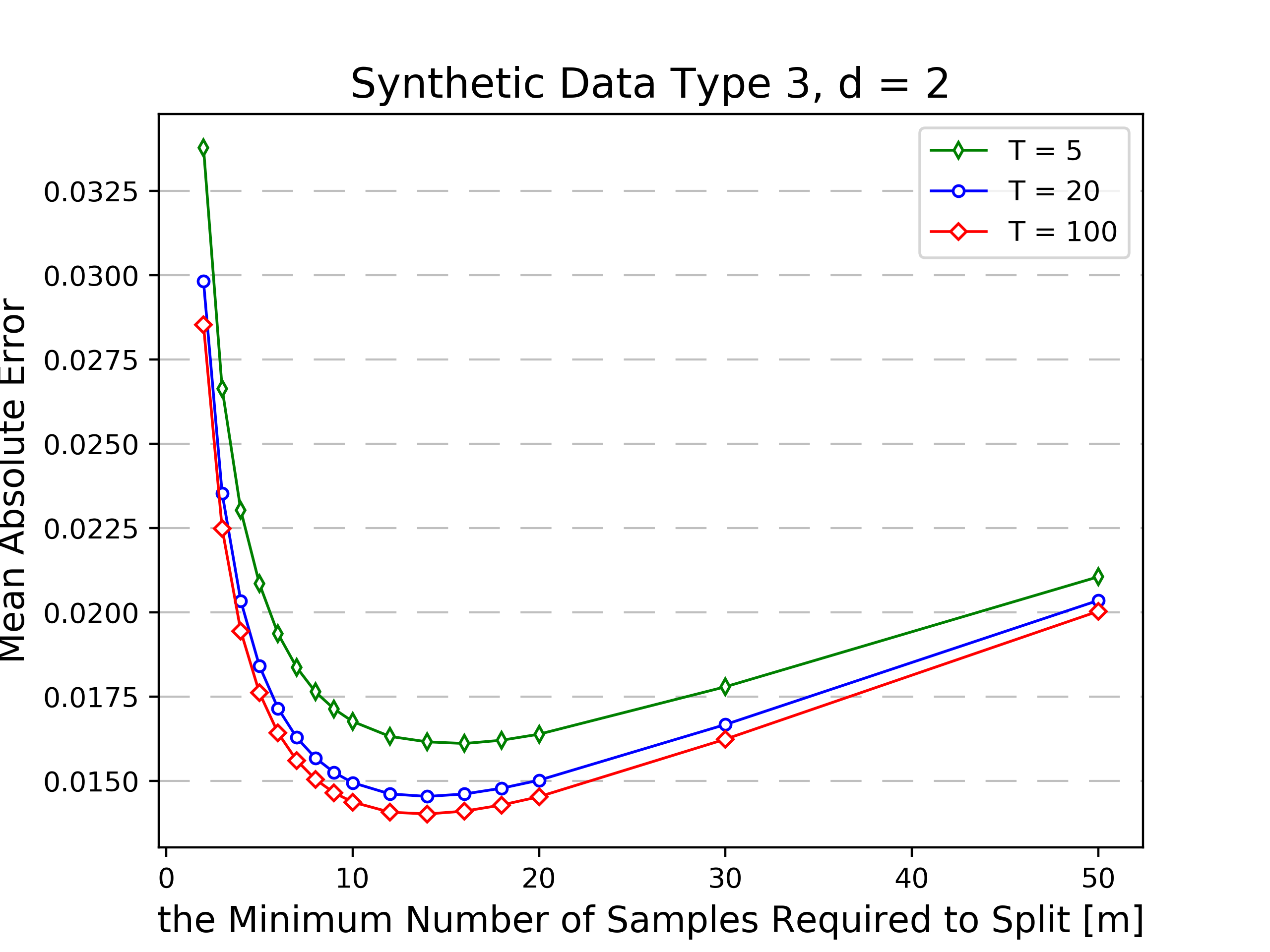}
\end{minipage}
}
\subfigure[The study of parameters with $d=2$, where each row respectively stands for each type of the four distributions. The left column indicates how \textit{ANLL} varies along parameters $m$ and $T$, and the right column shows the variation of \textit{MAE}.]{
{
	\begin{minipage}{0.450\textwidth}
		\centering
		\includegraphics[width=\textwidth]{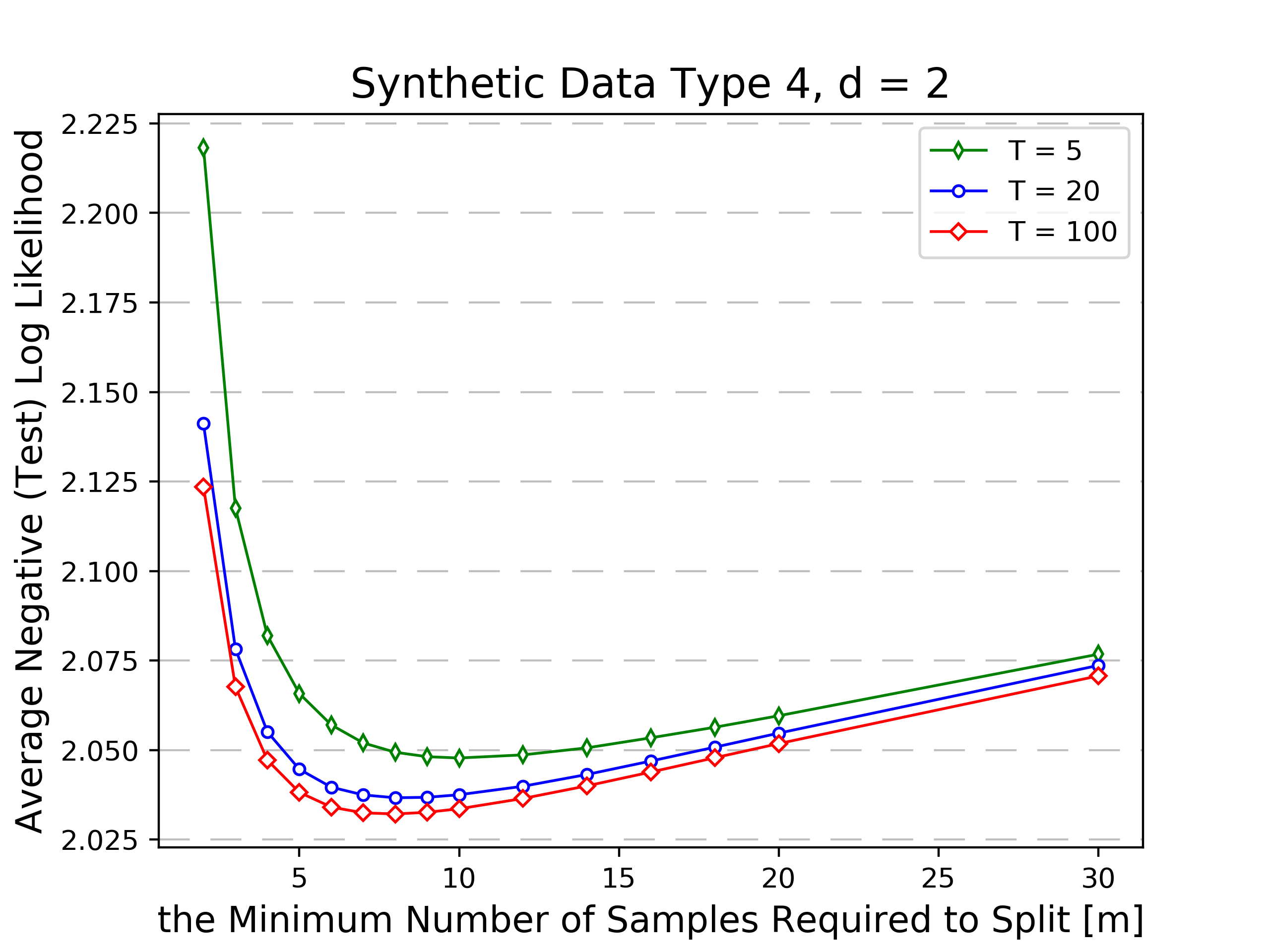}
	\end{minipage}
}
{
	\begin{minipage}{0.450\textwidth}
		\centering
		\includegraphics[width=\textwidth]{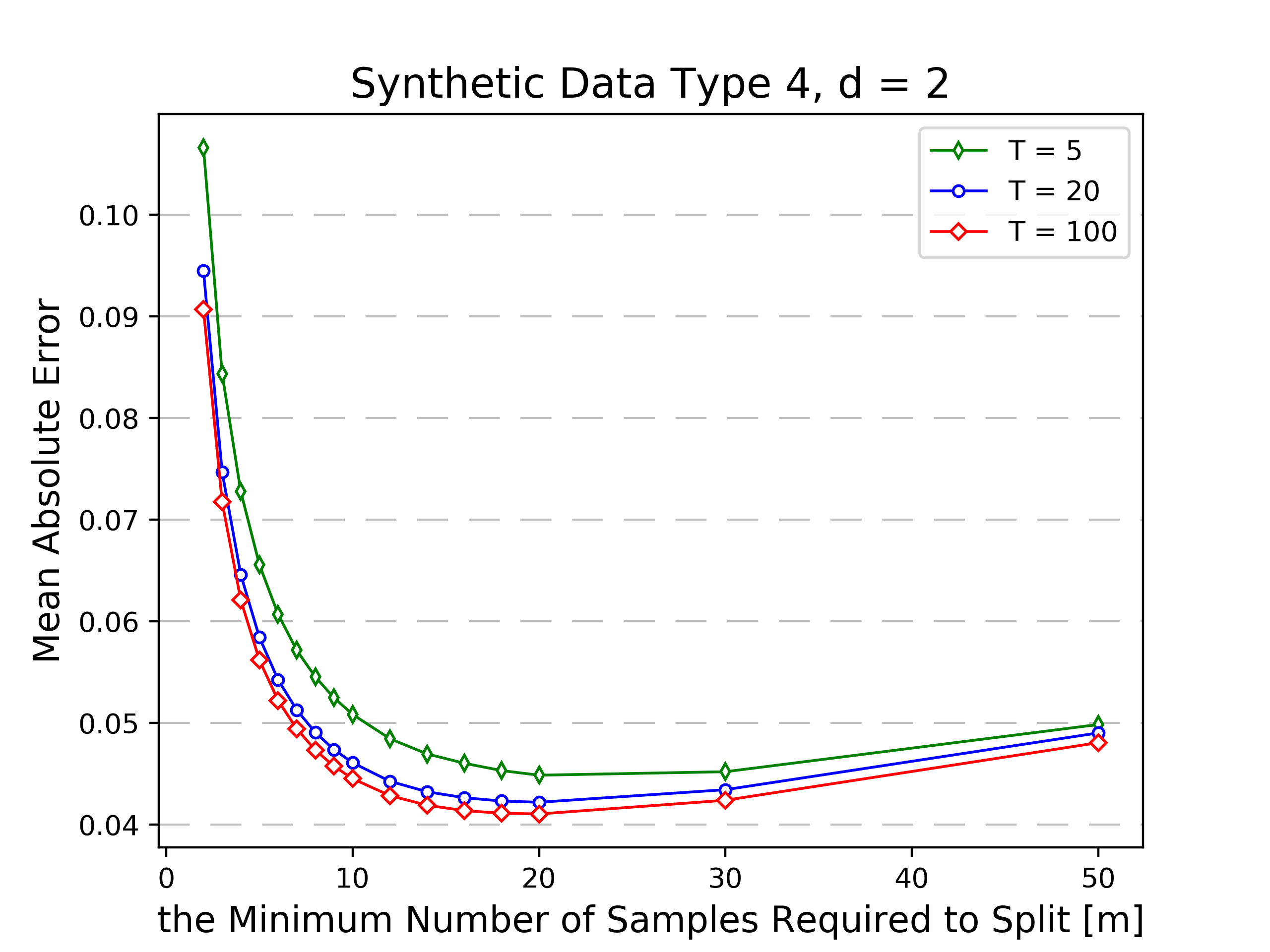}
	\end{minipage}
}
%	\vspace{2.0mm}
}
\end{figure}
\begin{figure}[htbp]
\centering
{
\begin{minipage}{0.450\textwidth}
	\centering
	\includegraphics[width=\textwidth]{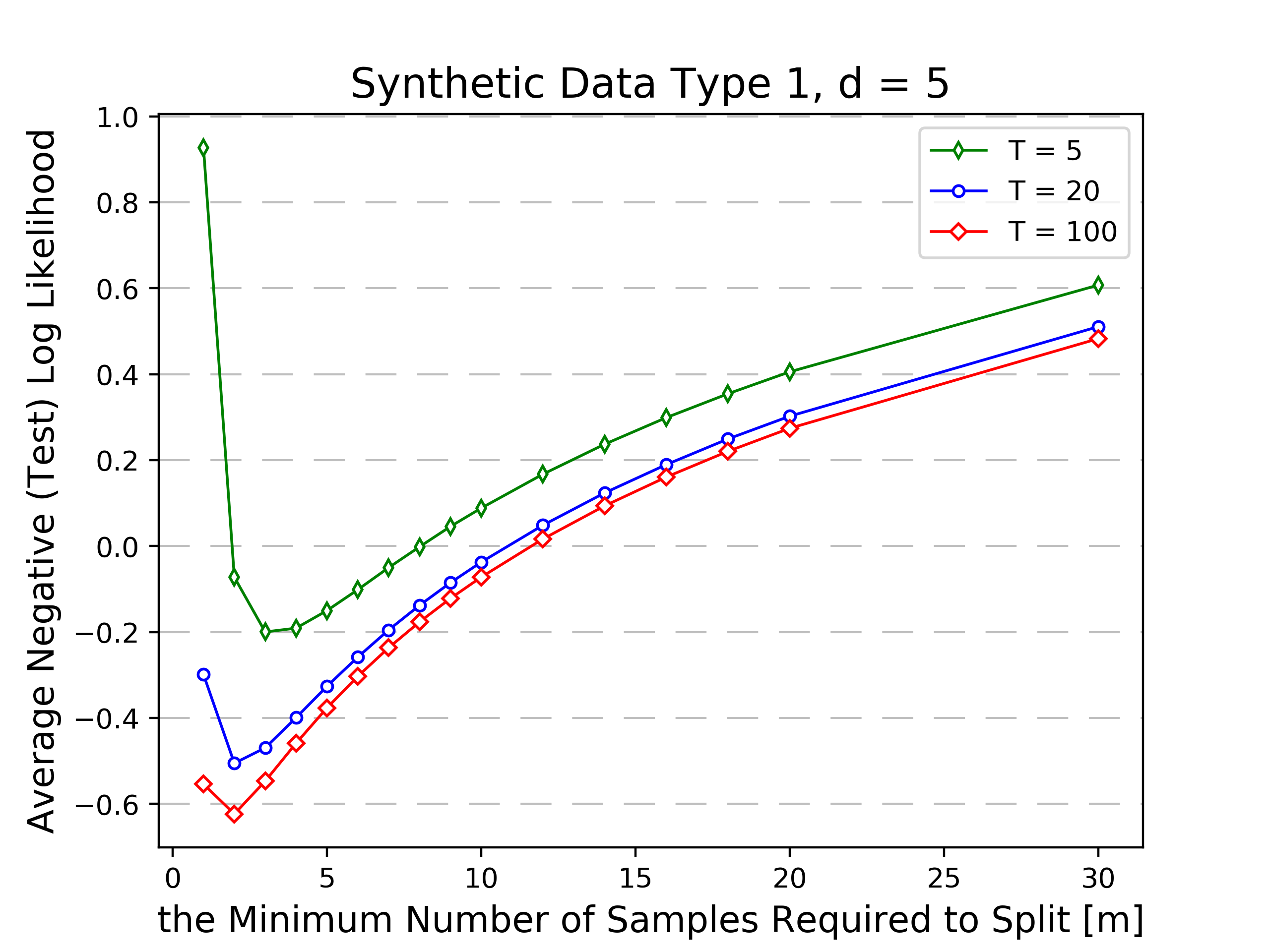}
\end{minipage}
}
{
\begin{minipage}{0.450\textwidth}
	\centering
	\includegraphics[width=\textwidth]{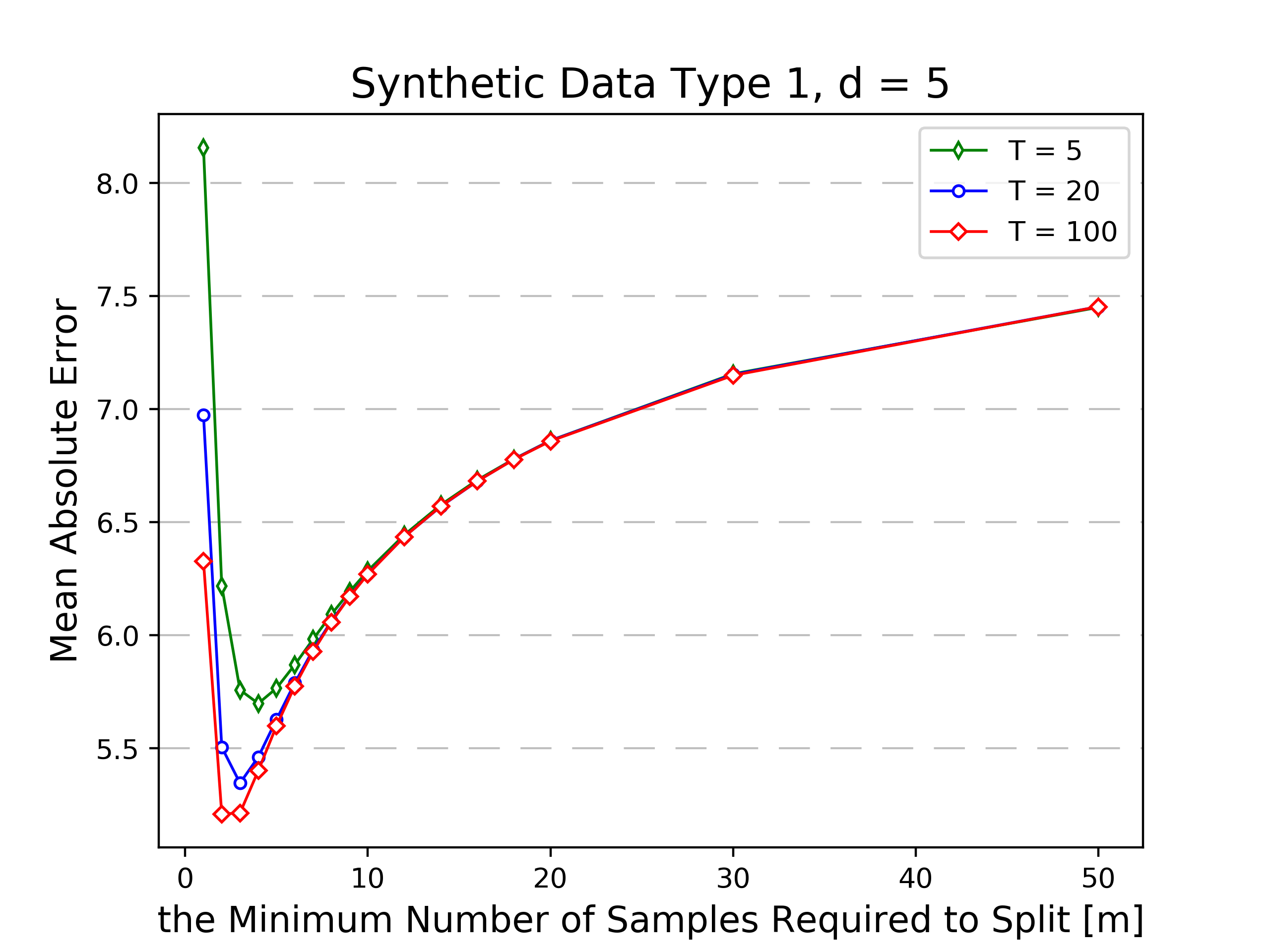}
\end{minipage}
}
{
\begin{minipage}{0.450\textwidth}
	\centering
	\includegraphics[width=\textwidth]{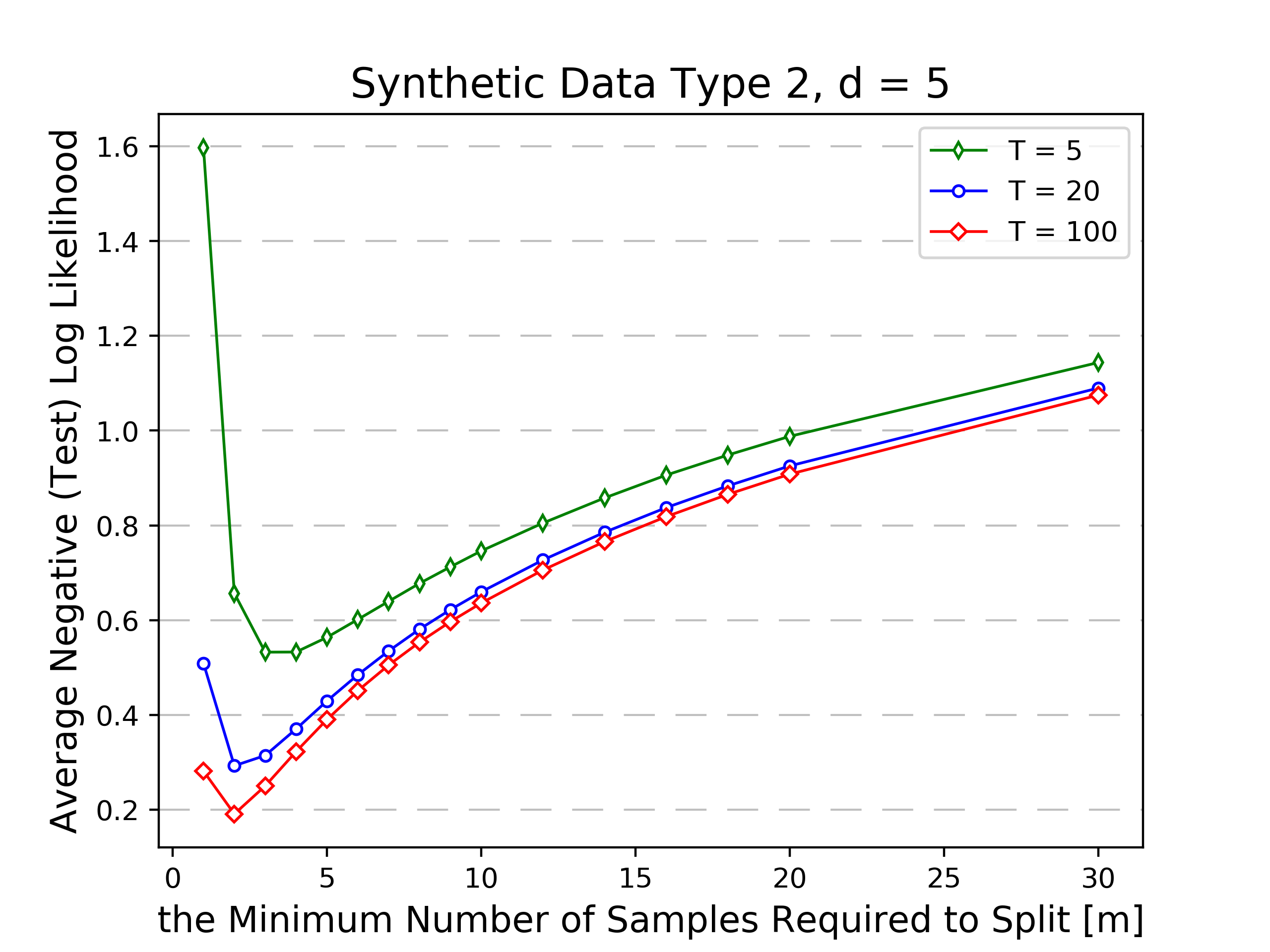}
\end{minipage}
}
{
\begin{minipage}{0.450\textwidth}
	\centering
	\includegraphics[width=\textwidth]{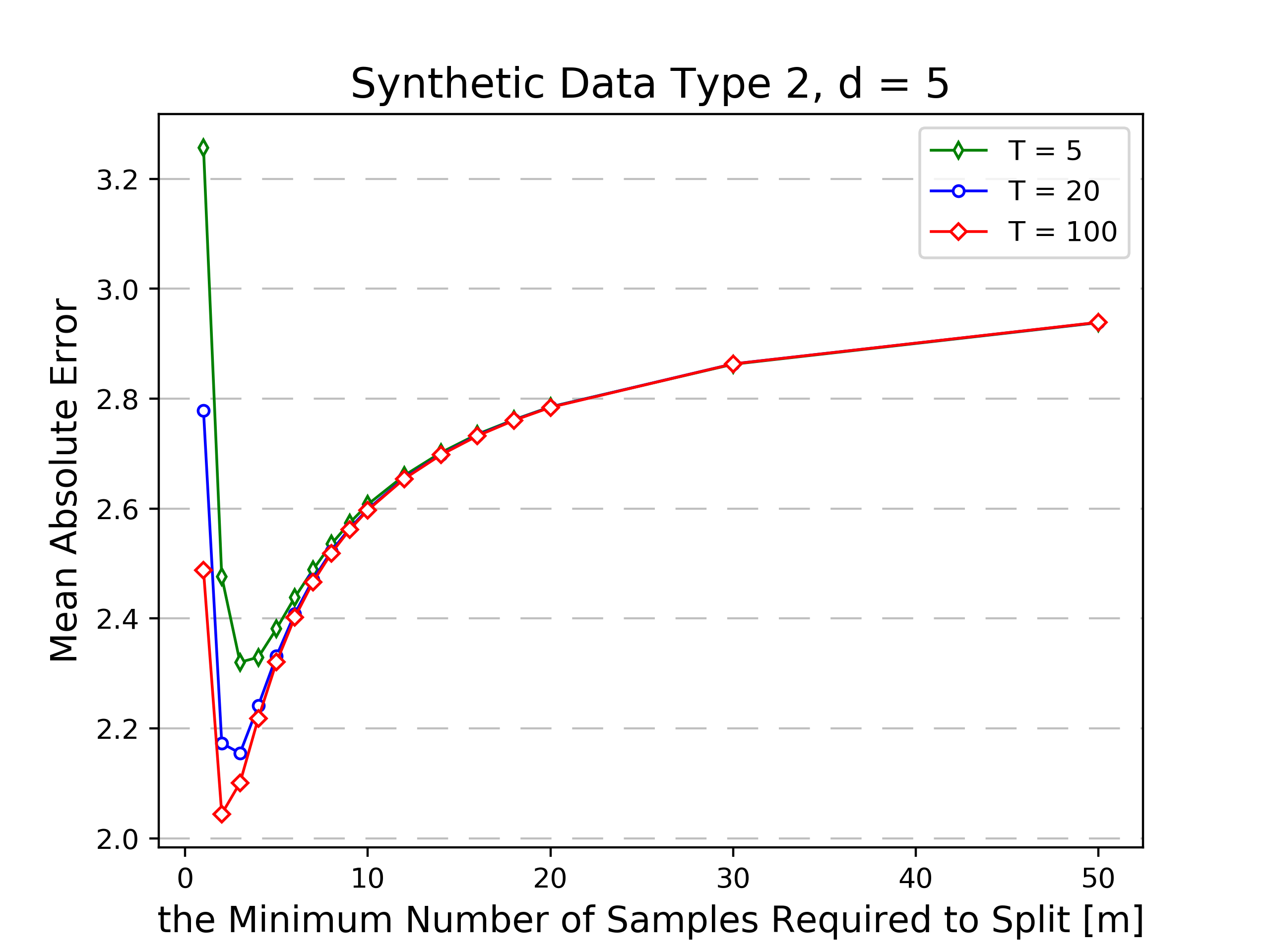}
\end{minipage}
}
{
\begin{minipage}{0.450\textwidth}
	\centering
	\includegraphics[width=\textwidth]{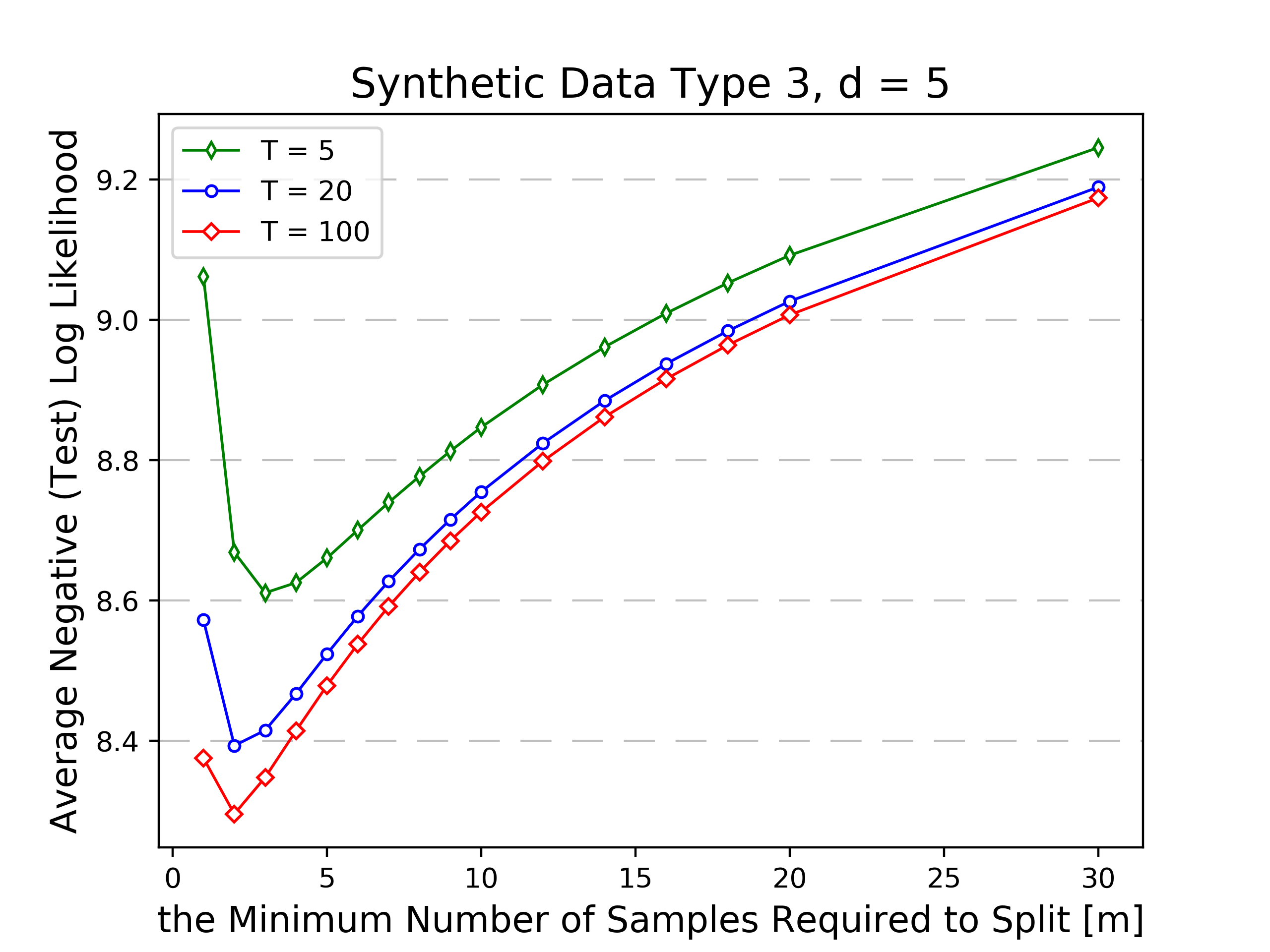}
\end{minipage}
}
{
\begin{minipage}{0.450\textwidth}
	\centering
	\includegraphics[width=\textwidth]{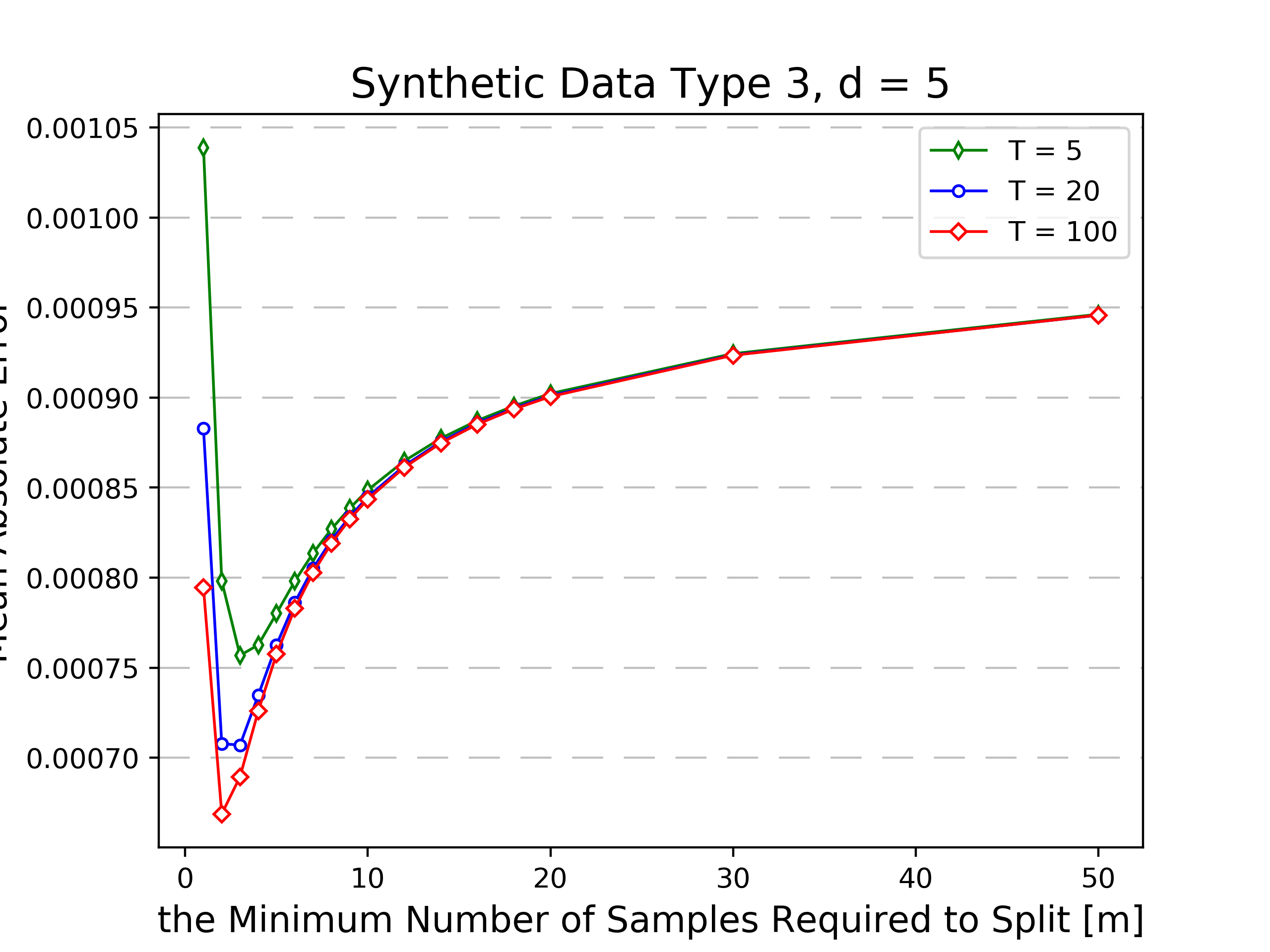}
\end{minipage}
}
\subfigure[The study of parameters with $d=5$, where each row respectively stands for each type of the four distributions. The left column indicates how \textit{ANLL} varies along parameters $m$ and $T$, and the right column shows the variation of \textit{MAE}.]{
{
	\begin{minipage}{0.450\textwidth}
		\centering
		\includegraphics[width=\textwidth]{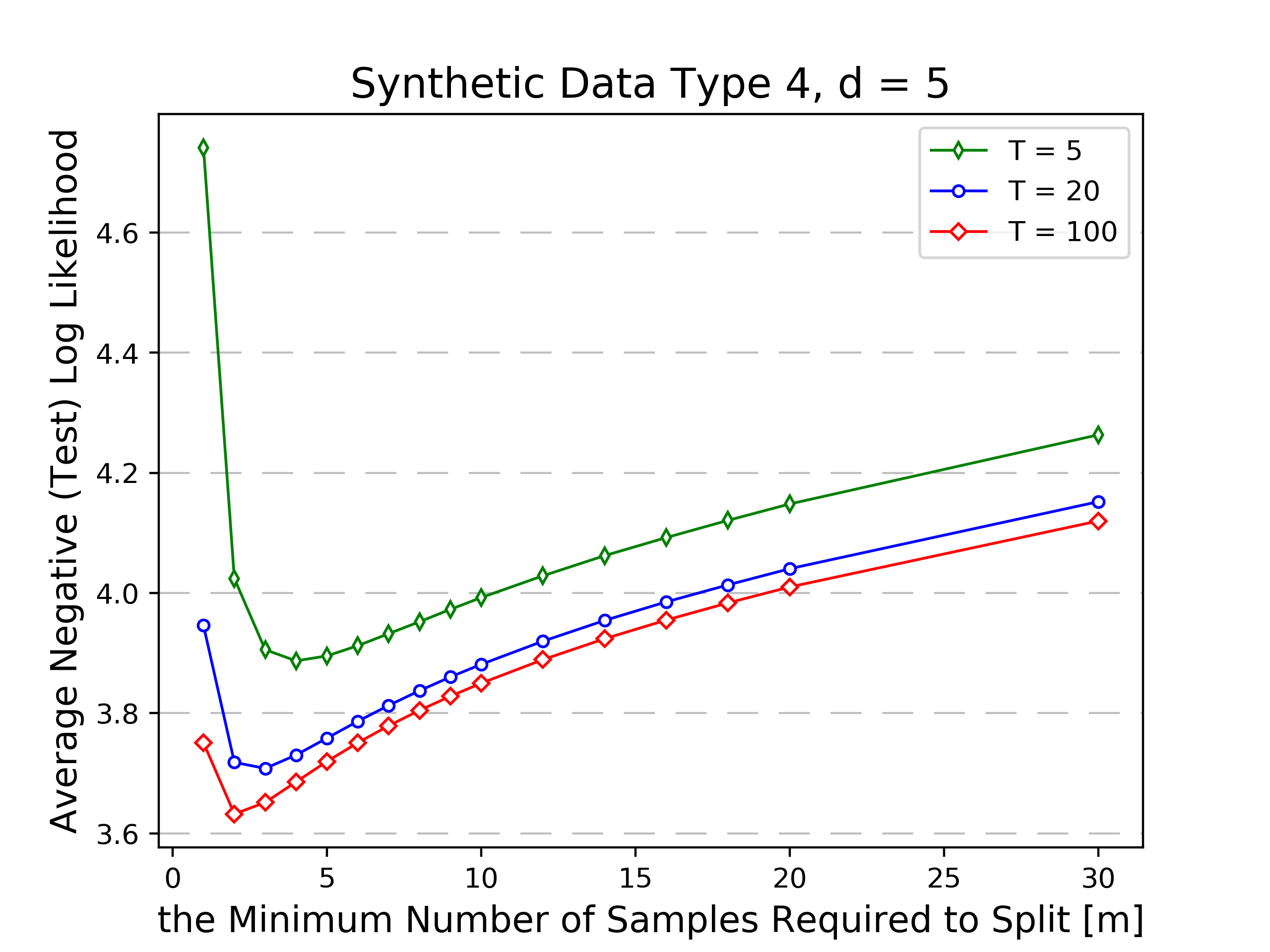}
	\end{minipage}
}
{
	\begin{minipage}{0.450\textwidth}
		\centering
		\includegraphics[width=\textwidth]{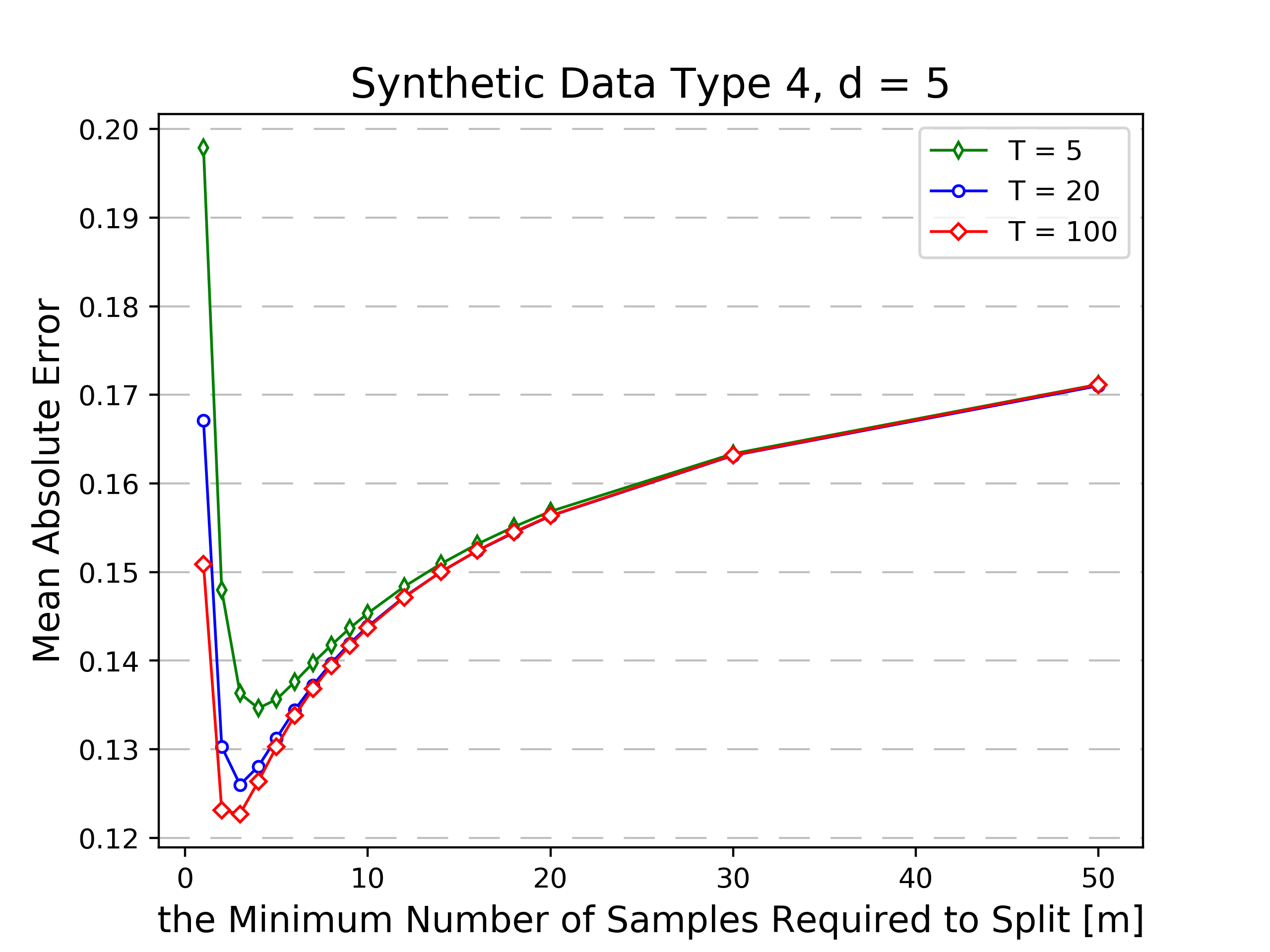}
	\end{minipage}
}
%	\vspace{2.0mm}
}
\caption{Study of parameters for AHTE.}
\label{fig::ParameterStudy}
\end{figure}

As we can see in Figure \ref{fig::ParameterStudy}, the performances of our AHTE estimator have the local optimal results on axis of $m$, that is, we can choose best $m$ parameters to minimize the $ANLL$ and average $MAE$ of each line. On the other hand, the results gain considerable improves when we go from $T=5$ to $T=100$, but then a steady state seems to be reached. Further more, average Negative log-likelihood loss function obtains better ability in tell the differences when $T$ changes. Features above provide enormous convenience to parameter selection in experiments, since it means that we can not only possibly figure out the optimal parameter $m$, but also choose a rather big $T$ as the number of partitions in an ensemble, for the further incrementation has no significant effect beyond certain limit.

\subsubsection{Performance Comparisons}\label{sec::SyntheticComparisons}

In the following experiments, comparisons are conducted among our adaptive histogram transform ensembles (AHTE), kernel density estimators (KDE) and Na\"{i}ve Histogram transform ensembles (NHTE).
\begin{itemize}
    \item 
KDE: the kernel density estimators \citep{PARZEN1962On, Davis1956Remarks} where we take the Gaussian kernel.
    \item 
NHTE: the original HTE density estimator proposed by \cite{Ezequiel2013HT}.
\end{itemize}

\begin{table}[H] 
\setlength{\tabcolsep}{9pt}
\centering
\captionsetup{justification=centering}
\caption{\footnotesize{Average \textit{ANLL} and \textit{MSE} over synthetic data sets}}
\label{tab::SyntheticData}
\resizebox{0.9\textwidth}{40mm}{
\begin{tabular}{cccccccc}
\toprule
\multirow{2}*{\text{Datasets}}
& \multirow{2}*{$d$} 
& \multicolumn{2}{c}{\text{NHTE}}
& \multicolumn{2}{c}{\text{KDE}}
& \multicolumn{2}{c}{\text{AHTE}} \\
\cline{3-8}
\multicolumn{2}{c}{}&\textit{ANLL}&\textit{MAE}
&\textit{ANLL}&\textit{MAE} &\textit{ANLL}&\textit{MAE} 
\\
\hline 
\hline
\multirow{4}*{I}
& \multirow{2}*{2} & $\textbf{-0.5942}$ & $\textbf{0.7029}$ & $-0.3112$ & $0.8002$ & $-0.5768$ & $0.8204$ \\
&   & $(0.0556)$ & $(0.2060)$ & $(0.0082)$ & $(0.0199)$ & $(0.0143)$ & $(0.0463)$ \\
\cline{2-8}
& \multirow{2}*{5} & $0.5870$ & $8.7246$ & $-0.1295$ & $6.7247$ & $\textbf{-0.5727}$ & $\textbf{6.3086}$ \\
&   & $(0.4282)$ & $(1.1090)$ & $(0.0135)$ & $(0.0426)$ & $(0.0217)$ & $(0.1717)$ \\
\hline
\multirow{4}*{II}
& \multirow{2}*{2} & $-0.1339$ & $0.4941$ & $-0.0108$ & $0.5479$ & $\textbf{-0.1526}$ & $\textbf{0.3791}$ \\
&   & $(0.0608)$ & $(0.1639)$ & $(0.0066)$ & $(0.0106)$ & $(0.0181)$ & $(0.0830)$ \\
\cline{2-8}
& \multirow{2}*{5} & $1.4265$ & $3.4790$ &$0.4796$ & $2.5694$ & $\textbf{0.2582}$ & $\textbf{2.4126}$ \\
&   & $(0.6569)$ & $(0.6844)$ & $(0.0102)$ & $(0.0299)$ & $(0.0251)$ & $(0.1723)$ \\
\hline
\multirow{4}*{III}
& \multirow{2}*{2} & $3.2186$ & $0.0200$ & $3.2726$ & $0.0233$ & $\textbf{3.1647}$ & $\textbf{0.0161}$ \\
&   & $(0.0617)$ & $(0.0060)$ & $(0.0096)$ & $(0.0005)$ & $(0.0161)$ & $(0.0035)$ \\
\cline{2-8}
& \multirow{2}*{5} & $9.2211$ & $0.0011$ & $8.5674$ & $0.0008$ & $\textbf{8.3585}$ & $\textbf{0.0008}$ \\
&   & $(0.3222)$ & $(0.0002)$ & $(0.0130)$ & $(0.0000)$ & $(0.0168)$ & $(0.0001)$ \\
\hline
\multirow{4}*{IV}
& \multirow{2}*{2} & $2.1263$ & $0.0641$ & $2.1309$ & $0.0531$ & $\textbf{2.0343}$ & $\textbf{0.0450}$ \\
&   & $(0.0558)$ & $(0.0180)$ & $(0.0103)$ & $(0.0024)$ & $(0.0144)$ & $(0.0045)$ \\
\cline{2-8}
& \multirow{2}*{5} & $4.8812$ & $0.1820$ & $3.7944$ & $0.1489$ & $\textbf{3.6513}$ & $\textbf{0.1227}$ \\
&   & $(0.4528)$ & $(0.0303)$ & $(0.0208)$ & $(0.0032)$ & $(0.0226)$ & $(0.0033)$ \\
\bottomrule	
\end{tabular}}
\begin{minipage}{0.9\textwidth}
\begin{tablenotes}
\footnotesize
\item {*} The best results are marked in \textbf{bold}, and the standard deviation is reported in the parenthesis under each value.
\end{tablenotes}
\end{minipage}
\end{table}

We perform the experiment with $n = 2,000$ training points and $m = 10,000$ test samples. For our AHTE, we set $T=100$. Besides, the hyper-parameter \textit{min\_samples\_split}, the minimal number of samples required to split, is chosen by the grid search method, where the validation set consists of $30\%$ samples randomly selected from the training data, and the optimal parameters are with the minimal validated \textit{ANLL}. The grid is selected as $\{1, 3, 10, 20, 40\}$. For NHTE, the selection of the parameters $s_{\min}$ and $s_{\max}$, is referred to the Nelder-Mead method \citep{nelder1965simplex}. And for KDE, we use the function \emph{gaussian\_kde} in the package \emph{scipy} of python and adopt the default settings for bandwidth selection. For every method we have computed the average \emph{ANLL} and \emph{MAE} over the $100$ runs.

As is shown in Table \ref{tab::SyntheticData}, our adaptive AHTE method outperforms the other two state-of-the-art algorithms NHTE and KDE in terms of \textit{ANLL} and \textit{MAE} measurements, only except for dataset I which was reducted to 2-dimension by PCA. The advantages brought by our AHTE model derives from the high efficiency of information utilization, making it possible for us to establish adaptive domains obtaining sample points with similarities.

Experimental results of synthetic data have so far shown part of its strength of our AHTE model. With data on high dimension space, our model can bring more accuracy to density estimation, proved by \textit{ANLL} and \textit{MAE} loss measurement respectively. Further more, as the synthetic data sets are more \textit{regularized} data, the next step we take is to vertify our model in complex real data situation.

\subsection{Real Data Analysis} \label{sec::subsec::realdata}

\subsubsection{Real Data Settings}

We conduct the real data comparisons based on the following UCI datasets \citep{bache2013uci}, previously used to study the performance of other density estimators\citep{silva2011mixed, tang2012deep, uria2013rnade}. Following \cite{tang2012deep}, we eliminated discrete-valued attributes and an attribute from every pair with a Pearson correlation coefficient greater than $0.98$. Each dimension of the data was normalized to $[0,1]$, and all results are reported on the normalized data.

\begin{itemize}
    \item
{\tt WineQuality}: 
The \textit{Wine Quality Data Set} is available on UCI. Two sub-datasets are included, related to red and white vinho verde wine samples, with continuous attributes representing acidity, pH, alcohol, etc. The {\tt Redwine} data set consist of $1,599$ observations of dimension $11$, and the {\tt Whitewine} data set consist of $4,898$ of dimension $11$. In experiment, each of the two data sets are reduced to $1$, $3$, $6$ and $8$ dimensions respectively through PCA.
    \item
{\tt Parkinsons}: 
The \textit{Parkinsons Telemonitoring Data Set} available on UCI is composed of a range of biomedical voice measurements from $42$ people with early-stage Parkinson's disease. This data set contains $5,875$ observations with $26$ attributes concerning basic information of patients, vocal frequencies and vocal amplitudes. In experiment, instances in this data set are respectively reduced to $2$, $4$, $8$ and $10$ dimensions through PCA.
    \item 
{\tt Ionosphere}: 
This data set is available on the UCI Repository with $351$ observations and $34$ continuous attributes. The radar data was collected by a system in Goose Bay, Labrador, which consists of a phased array of 16 high-frequency antennas. In experiment, instances in this data set are respectively reduced to $3$, $10$, $16$ and $22$ dimensions through PCA.
\end{itemize}

Note that in the data pre-processing, all data sets are reduced to various lower dimensions through PCA. The reasons are intuitive. Since real density often resides in a low-dimension manifold instead of filling the whole high-dimensional space, it becomes more reasonable to solve the density estimation problems after dimensionality reduction. However, the dimension of the specific manifold remains unknown, so we take it as a hyper-parameter and try several possible options. 

In the following experiments, we adopt the same parameter settings as in Section \ref{sec::SyntheticComparisons}. For our AHTE, we set $T=100$. Besides, the hyper-parameter \textit{min\_samples\_split}, the minimal number of samples required to split, is chosen by the grid search method, where the validation set consists of $30\%$ samples randomly selected from the training data, and the optimal parameters are with the minimal validated \textit{ANLL}. The grid is selected as $\{1, 3, 10, 20, 40\}$. For NHTE, the selection of the parameters $s_{\min}$ and $s_{\max}$, is referred to the Nelder-Mead method \citep{nelder1965simplex}. And for KDE, we use the function \emph{gaussian\_kde} in the package \emph{scipy} of python and adopt the default settings for bandwidth selection.

\subsubsection{Real Data Comparisons}

\begin{table}[htbp] 
\setlength{\tabcolsep}{9pt}
\centering
\captionsetup{justification=centering}
\caption{\footnotesize{Average \textit{ANLL} over real data sets}}
\label{tab::Realdata}
\resizebox{0.9\textwidth}{36mm}{
\begin{tabular}{c|c|c|c|c}
\toprule
Datasets & $d$ &\text{NHTE} & \text{KDE} & \text{Our AHTE} \\ \hline \hline
\multirow{4}*{{\tt Redwine}}
& 1 & $\mathbf{-0.9168}$ $(0.1090)$ &  $-0.8933$  $(0.1016)$ & $-0.9029$  $(0.1132)$ \\
& 3 & $-2.3152$ $(0.3605)$ & $-2.5608$  $(0.1517)$ & $\mathbf{-2.7211}$ $(0.1964)$ \\
& 6 & $-3.0261$  $(1.3117)$ & $-5.6829$  $(0.2179)$ & $\mathbf{-6.1486}$ $(0.3027)$ \\
& 8 & $-2.6352$  $(1.2961)$ & $-7.5382$  $(0.2517)$ & $\mathbf{-7.9488}$ $(0.3488)$ \\ \hline
\multirow{4}*{{\tt Whitewine}}
& 1 & $-0.9561$  $(0.1151)$ & $\mathbf{-0.9661}$  $(0.1136)$ & $-0.9402$ $(0.1161)$ \\
& 3 & $-4.0290$ $(0.3888)$ & $-4.1101$  $(0.3724)$ & $\mathbf{-4.3077}$ $(0.3847)$ \\
& 6 & $-4.5044$ $(0.9711)$ & $-6.5178$  $(0.4369)$ & $\mathbf{-6.6920}$ $(0.4314)$ \\
& 8 & $-4.8352$ $(1.0595)$ & $-8.7202$ $(0.4346)$ & $\mathbf{-8.9497}$ $(0.4217)$ \\ \hline
\multirow{4}*{{\tt Parkinsons}}
& 2 & $-0.7765$ $(0.5245)$ & $-0.2940$ $(0.0200)$ & $\mathbf{-3.5291}$ $(0.1230)$ \\
& 4 & $-4.8725$ $(1.3322)$ & $-1.9956$ $(0.0355)$ & $\mathbf{-4.9030}$ $(0.1285)$ \\
& 8 & $-4.8511$ $(0.9005)$ & $-6.0971$ $(0.1133)$ & $\mathbf{-7.1145}$ $(0.1580)$ \\
& 10 & $-5.9335$ $(1.4770)$ & $-10.2285$ $(0.1634)$ & $\mathbf{-11.0099}$ $(0.2096)$ \\ \hline
\multirow{4}*{{\tt Ionosphere}}
& 3 & $-1.1516$ $(0.7316)$ & $-1.5021$ $(0.2243)$ & $\mathbf{-1.7826}$ $(0.5222)$ \\
& 10 & $-6.3828$  $(1.0471)$ & $-7.8793$ $(1.1025)$ & $\mathbf{-11.2895}$ $(1.7528)$ \\
& 16 & $-10.2635$  $(1.9343)$ & $-13.2235$ $(1.8512)$ & $\mathbf{-18.9552}$ $(2.8121)$ \\
& 22 & $-14.7161$ $(2.3818)$ & $-18.7878$ $(2.3641)$ & $\mathbf{-25.3548}$ $(3.6542)$ \\
\bottomrule	
\end{tabular}}
\begin{minipage}{0.9\textwidth}
\begin{tablenotes}
\footnotesize
\item{*} The best results are marked in \textbf{bold}, and the standard deviation is reported in the parenthesis.
\end{tablenotes}
\end{minipage}
\end{table}

In Table \ref{tab::Realdata}, we summarize the comparisons with other two top-notch density estimators that demonstrates the accuracy of our histogram transform algorithm on four real datasets. For almost all of the redacted datasets, our AHTE model shows its superiority on accuracy. Meanwhile, with the incrementation of data dimension, our AHTE algorithm reveals itself with better performance on accuracy, obtaining a non-negligible edge over the popular KDE and NHTE algorithms.

\section{Proofs} \label{sec::proofs}

\subsection{Proofs of Results in the Space $C^{0,\alpha}$}

\subsubsection{Proofs Related to Section \ref{subsubsec::AppError0}}

\begin{proof}[of Proposition \ref{ApproximationError}] 
\emph{(i)} Since the space of continuous and compactly supported functions $C_c(\mathbb{R}^d)$ is dense in $L_1(\mathbb{R}^d)$, we can find $\tilde{f} \in C_c(\mathbb{R}^d)$ such that
for all $\varepsilon > 0$, there holds
\begin{align} \label{CompactApprox}
\|f - \tilde{f}\|_{L_1(\mu)} \leq \varepsilon/3,
\end{align}
and $B := \supp(\tilde{f}) \subset B_r$ is a $d$-dimensional rectangle. Moreover, $\tilde{f}$ is uniformly continuous, since it is continuous and $\mathrm{supp}(\tilde{f})$ is compact. This implies that there exists a $\delta \in (0, 1]$ such that if $\|x - x'\|_1 \leq \delta$, then we have
\begin{align} \label{ftildeUniformlyContinuity}
|\tilde{f}(x) - \tilde{f}(x')|
\leq \frac{\varepsilon}{3 \mu(B_r)}.
\end{align}
Now we let $h_{\varepsilon} := \delta/d$ and define $\bar{f} : \mathbb{R}^d \to \mathbb{R}$ by
\begin{align} \label{fbar}
\bar{f}(x) := 
\frac{1}{\mu(A_H(x))} \int_{A_H(x)} \tilde{f}(x') \, d\mu(x'), 
\end{align}
where $A_H(x)$ is as in \eqref{equ::InputBin}. Then, for all $\varepsilon > 0$, \eqref{CompactApprox} implies that
\begin{align}
\|f - f_{\mathrm{P},H}\|_{L_1(\mu)}
& \leq \|f - \tilde{f}\|_{L_1(\mu)} + \|\tilde{f} - \bar{f}\|_{L_1(\mu)} + \|\bar{f} - f_{\mathrm{P},H}\|_{L_1(\mu)}
\nonumber\\
& \leq \varepsilon/3 + \|\tilde{f} - \bar{f}\|_{L_1(\mu)} + \|\bar{f} - f_{\mathrm{P},H}\|_{L_1(\mu)}.
         \label{L1ErrorDecomposition}
\end{align}
If $x \in B_r^c$, then we have $x \in B^c$ and thus $\tilde{f}(x) = 0$. Moreover, there holds
\begin{align*}
\bar{f}(x) := \frac{1}{\mu(A_H(x))} \int_{A_H(x)} \tilde{f}(x') \, d\mu(x') = 0.
\end{align*}
Therefore, we obtain
\begin{align*}
\|\tilde{f} - \bar{f}\|_{L_1(\mu)}
= \int_{B_r} |\tilde{f}(x) - \bar{f}(x)| \, d\mu(x).
\end{align*}
For $x \in B_r$ with $\mu(A_H(x)) > 0$, there holds
\begin{align*}
|\tilde{f}(x) - \bar{f}(x)|
& = \biggl| \frac{1}{\mu(A_H(x))} \int_{A_H(x)} \tilde{f}(x) \, d\mu(x') - \frac{1}{\mu(A_H(x))} \int_{A_H(x)} \tilde{f}(x') \, d\mu(x') \biggr|
\\
& \leq \frac{1}{\mu(A_H(x))} \int_{A_H(x)} |\tilde{f}(x) - \tilde{f}(x')| \, d\mu(x').
\end{align*}
For all $x' \in A_H(x)$, we have
\begin{align*}
\|x - x'\|_1 \leq d \cdot \overline{h}_0 \leq \delta,
\end{align*}
Consequently,
\begin{align*}
|\tilde{f}(x) - \tilde{f}(x')|
\leq \frac{\varepsilon}{3 \mu(B_r)}.
\end{align*}
As a result, we have
\begin{align} \label{ftildefbarL1}
\|\tilde{f} - \bar{f}\|_{L_1(\mu)}
= \int_{B_r} |\tilde{f}(x) - \bar{f}(x)| \, d\mu(x)
\leq \frac{\varepsilon}{3}.
\end{align}
Finally, \eqref{fbar} yields that
\begin{align*}
\|\bar{f} - f_{\mathrm{P},H}\|_{L_1(\mu)}
& = \sum_{j \in \mathcal{I}_H} \int_{\mathbb{R}^d} |\bar{f}(x) - f_{\mathrm{P},H}(x)|\eins_{A_j}(x) \, d\mu(x)
\\
& = \sum_{j \in \mathcal{I}_H} \int_{A_j} |(\bar{f}(x) - f_{\mathrm{P},H}(x)) \eins_{A_j}(x)| \, d\mu(x)
\\
& = \sum_{j \in \mathcal{I}_H} \int_{A_j} \biggl| \frac{1}{\mu(A_j)} \int_{A_j} \tilde{f}(x') \, d\mu(x')
- \frac{1}{\mu(A_j)} \int_{A_j} f(x') \, d\mu(x') \biggr| \, d\mu(x)
\\
& \leq \sum_{j \in \mathcal{I}_H}  \int_{A_j}\biggl| \tilde{f}(x') - f(x') \, \biggr| d\mu(x')
\\
& \leq \int_{\mathbb{R}^d} |\tilde{f}(x') - f(x')| \, d\mu(x')
\\
& \leq \frac{\varepsilon}{3},
\end{align*}
which proves the assertion by combining the estimates \eqref{L1ErrorDecomposition} and \eqref{ftildefbarL1}.

\emph{(ii)} The assumption $f \in C^{0,\alpha}$ tells us that there exists a constant $c_L$ as in Definition \ref{def::Cp} such that for all $x', x'' \in A_H(x)$, we have
\begin{align*}
|f(x') - f(x'')|
\leq c_L \|x' - x''\|_1^{\alpha}
\leq c_L (\mathrm{diam}(A_H(x)))^{\alpha}
\leq c_L (d \overline{h}_0)^{\alpha}.
\end{align*}
Therefore, 
\begin{align*}
|f_{\mathrm{P},H}(x) - f(x)|
& = \biggl| \frac{1}{\mu(A_H(x))} \int_{A_H(x)} f(x') \, d\mu(x') - f(x) \biggr|
\\
& = \biggl| \frac{1}{\mu(A_H(x))} \int_{A_H(x)} f(x') - f(x) \, d\mu(x') \biggr|
\\
& \leq \frac{1}{\mu(A_H(x))} \int_{A_H(x)} |f(x') - f(x)| \, d\mu(x') 
\\
& \leq c_L (d \cdot \overline{h}_0)^{\alpha}
\\
& \leq \varepsilon.
\end{align*}
Consequently we obtain $\|f_{\mathrm{P},H}(x) - f(x)\|_{L_{\infty}(\mu)} \leq \varepsilon$.
\end{proof}

\begin{proof}[of Proposition \ref{L1LInftyRelation}] 
We decompose $\|f_{\mathrm{P},H} - f\|_{L_1(\mu)}$ as follows
\begin{align*}
\|f_{\mathrm{P},H} - f\|_{L_1(\mu)}
& = \int_{B_r} |f_{\mathrm{P},H} - f| \, d\mu
      + \int_{B_r^c} |f_{\mathrm{P},H} - f| \, d\mu
\\
& \leq \mu(B_r) \|(f_{\mathrm{P},H} - f) \eins_{B_r}\|_{L_{\infty}(\mu)}
          + \int_{B_r^c} f_{\mathrm{P},H} \, d\mu
          + \int_{B_r^c} f \, d\mu
\\
& = 2^d r^d\|(f_{\mathrm{P},H} - f) \eins_{B_r}\|_{L_{\infty}(\mu)}
      + \int_{B_r^c} f_{\mathrm{P},H} \, d\mu
      + \mathrm{P}(B_r^c).
\end{align*}
Then \eqref{RHdensity} implies 
\begin{align*}
\int_{B_r^c} f_{\mathrm{P},H}(x) \, dx
& = \int_{B_r^c}  \sum_{j \in \mathcal{I}_H} \frac{\mathrm{P}(A_j) \eins_{A_j}(x)}{\mu(A_j)} 
                          + \frac{\mathrm{P}(B_r^c) \eins_{B_r^c}}{\mu(B_r^c)} \, d\mu(x)
\\
& = \frac{\mathrm{P}(B_r^c)}{\mu(B_r^c)} \int_{B_r^c} \eins_{B_r^c}(x) \, d\mu(x)
\\
& = \frac{\mathrm{P}(B^c_r)}{\mu(B_r^c)} \cdot \mu(B_r^c) 
\\
& = \mathrm{P}(B^c_r).
\end{align*}
Combining the above two estimates, we obtain the desired conclusion. 
\end{proof}

\subsubsection{Proofs Related to Section \ref{subsubsction::VCDimension}}

\begin{proof}[of Lemma \ref{FundamentalLemma}] 
\emph{(i)} Since $B_r = \bigcup_{j\in \mathcal{I}_H} A_j$, we have
\begin{align*}
\|f_{\mathrm{D},H} - f_{\mathrm{P},H}\|_{L_1(\mu)}
& = \int_{\mathbb{R}^d} |f_{\mathrm{D},H} - f_{\mathrm{P},H}| \, d\mu
\\
& = \int_{\mathbb{R}^d} \biggl| \sum_{j\in \mathcal{I}_H} \frac{1}{\mu(A_j)} (\mathrm{D}(A_j) - \mathrm{P}(A_j)) \eins_{A_j(x)}
       + \frac{1}{B_r^c}(\mathrm{D}(B_r^c) - \mathrm{P}(B_r^c))\eins_{B_r^c}\biggr| \, d\mu
\\
& = \sum_{j \in \mathcal{I}_H} \int_{A_j} \frac{1}{\mu(A_j)} |\mathrm{D}(A_j) - \mathrm{P}(A_j)|\, d\mu
        + \frac{1}{\mu(B_r^c)}\int_{B_r^c}|\mathrm{D}(B_r^c) - \mathrm{P}(B_r^c)|\, d\mu
\\
& = \sum_{j \in \mathcal{I}_H} |\mathrm{D}(A_j) - \mathrm{P}(A_j)| + |\mathrm{D}(B_r^c) - \mathrm{P}(B_r^c)|
\\
& = \sum_{j \in \mathcal{I}_H} |\mathbb{E}_{\mathrm{D}} \eins_{A_j} - \mathbb{E}_{\mathrm{P}} \eins_{A_j}| 
                                             + |\mathbb{E}_{\mathrm{D}}\eins_{B_r^c} - \mathbb{E}_{\mathrm{P}}\eins_{B_r^c}|
\\
& = \sum_{j \in \mathcal{I}_H \cup \{0\}} |\mathbb{E}_{\mathrm{D}} \eins_{A_j} - \mathbb{E}_{\mathrm{P}} \eins_{A_j}|.
\end{align*}	

\emph{(ii)} Using \eqref{equ::RHdensity} and \eqref{equ::EMRHdensity}, we get
\begin{align*}
\|f_{\mathrm{D},H} - f_{\mathrm{P},H}\|_{L_{\infty}(\mu)}
& = \sup_{j\in \mathcal{I}_H\cup \{0\}}\sup_{x \in A_j}  |f_{\mathrm{D},H}(x) - f_{\mathrm{P},H}(x)|
\\
& = \sup_{j\in \mathcal{I}_H\cup \{0\}}\sup_{x \in A_j}  \biggl| \frac{\mathrm{D}(A_j)}{\mu(A_j)} - \frac{\mathrm{P}(A_j)}{\mu(A_j)}\biggr|\\
&= \sup_{j\in \mathcal{I}_H\cup \{0\}} \frac{|\mathrm{D}(A_j) - \mathrm{P}(A_j)|}{\mu(A_j)}.
\end{align*}
This proves the assertion.
\end{proof}

To prove Lemma \ref{VCindex}, we need the following fundamental lemma concerning with the VC dimension of purely random partitions which follows the idea put forward by \cite{bremain2000some} of the construction of purely random forest. To this end, let $p \in \mathbb{N}$ be fixed and $\pi_p$ be a partition of $\mathcal{X}$ with number of splits $p$ and $\pi_{(p)}$ denote the collection of all partitions $\pi_p$.

\begin{lemma} \label{VCindexPre}
The VC dimension of $\mathcal{B}_p$ defined by
\begin{align} \label{Bp}
\mathcal{B}_p := \biggl\{ B : B = \bigcup_{j \in J} A_j, J \subset \{ 0, 1, \ldots, p \}, A_j \in \pi_p \subset \pi_{(p)} \biggr\}.
\end{align}
can be upper bounded by $d p + 2$.
\end{lemma}

\begin{proof}[of Lemma \ref{VCindexPre}] 
The proof will be conducted by dint of geometric constructions, and we proceed by induction. 

\begin{figure*}[htbp]
\centering
\begin{minipage}[b]{0.18\textwidth}
\centering
\includegraphics[width=\textwidth]{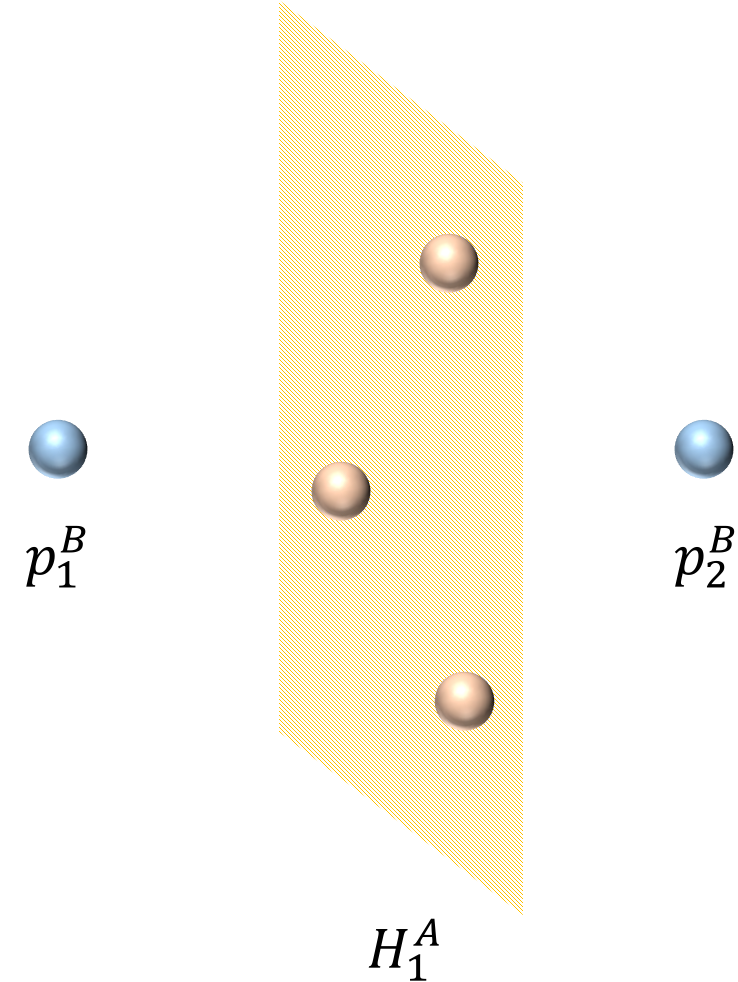}
$p=1$
\centering
\label{fig::p=1}
\end{minipage}
\qquad
\begin{minipage}[b]{0.25\textwidth}
\centering
\includegraphics[width=\textwidth]{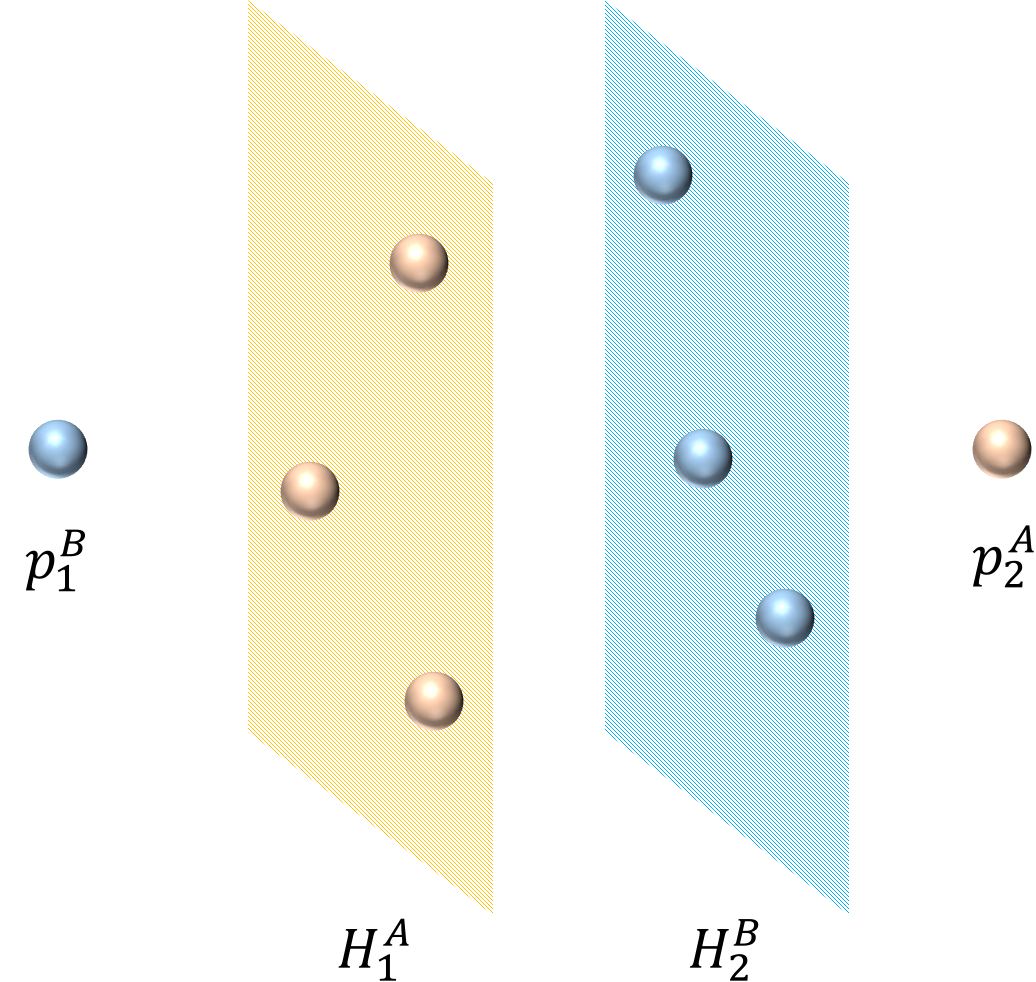}
$p=2$
\label{fig::p=2}
\end{minipage}
\qquad
\begin{minipage}[b]{0.43\textwidth}
\centering
\includegraphics[width=\textwidth]{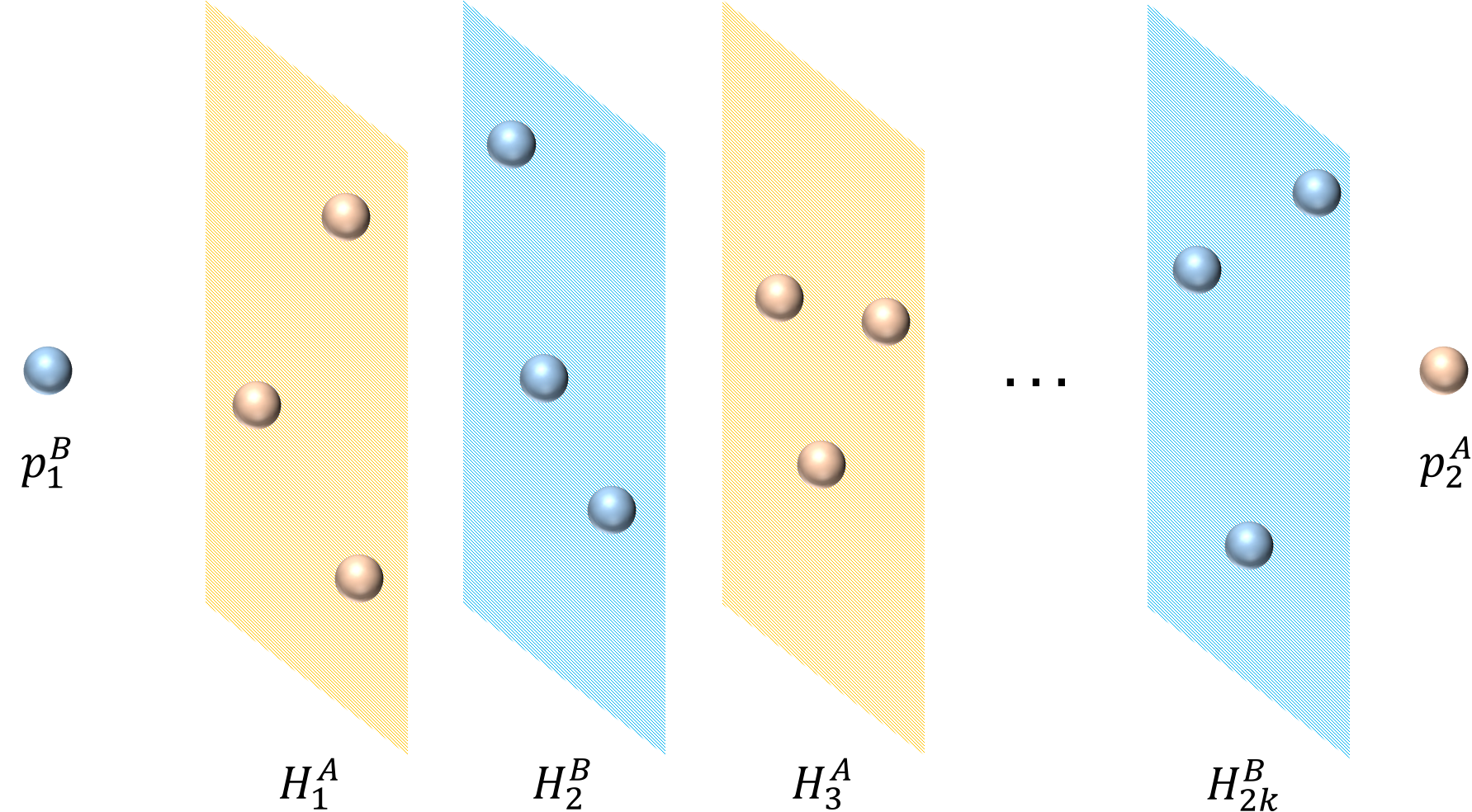}
$p=2k$
\label{fig::p=2k}
\end{minipage}
\caption{We take one case with $d=3$ as an example to illustrate the geometric interpretation of the VC dimension. The yellow balls represent samples from class $A$, blue ones are from class $B$ and slices denote the hyper-planes formed by samples. }
\label{fig::VC}
\end{figure*}

We begin by observing a partition with number of splits $p = 1$. On account that the dimension of the feature space is $d$, the smallest number of points that cannot be divided by $p = 1$ split is $d + 2$. Specifically, considering the fact that $d$ points can be used to form $d - 1$ independent vectors and therefore a hyper-plane of a $d$-dimensional space, we now focus on the case where there is a hyper-plane consisting of $d$ points all from the same class labeled as $A$, and there are two points from the other class $B$ on either side of the hyper-plane. We denote the hyper-plane by $H_1^A$ for brevity. In this case, points from two classes cannot be separated by one split, i.e. one hyper-plane, which means that $\mathrm{VC}(\mathcal{B}(\pi_1)) \leq d + 2$.

We next turn to consider the partition with number of splits $p = 2$ which is an extension of the above case. Once we pick one point out of the two located on either side of the above hyper-plane $H_1^A$, a new hyper-plane $H_2^B$ parallel to $H_1^A$ can be constructed by combining the selected point with $d - 1$ newly-added points from class $B$. Subsequently, a new point from class $A$ is added to the side of the newly constructed hyper-plane $H_2^B$. Notice that the newly added point should be located on the opposite side to $H_1^A$. Under this situation, $p = 2$ splits can never separate those $2 d + 2$ points from two different classes. As a result, we prove that $\mathrm{VC}(\mathcal{B}(\pi_2)) \leq 2 d + 2$.

If we apply induction to the above cases, the analysis of VC index can be extended to the general case where $p \in \mathbb{N}$. What we need to do is to add new points continuously to form $p$ mutually parallel hyperplanes with any two adjacent hyper-planes being built from different classes. Without loss of generality, we assume that $p = 2k+1$, $k \in \mathbb{N}$, and there are two points denoted by $p_1^B, p_2^B$ from class $B$ separated by $2 k + 1$ alternately appearing hyper-planes. Their locations can be represented by $p_1^B, H_1^A, H_2^B, H_3^A, H_4^B, \ldots, H_{(2k+1)}^A, p_2^B$. According to this construction, we demonstrate that the smallest number of points that cannot be divided by $p$ splits is $d p + 2$, which leads to $\mathrm{VC}(\mathcal{B}(\pi_p)) \leq d p + 2$.

It should be noted that our hyper-planes can be generated both vertically and obliquely, which is in line with our splitting criteria for the random partitions. This completes the proof. 
\end{proof}

\begin{proof}[of Lemma \ref{VCindex}]  
The proof will be conducted by dint of geometric constructions.

For the first assertion, we choose a data set $A \subset \mathbb{R}^d$ with $\#(A) = 2^d + 2$ and consider firstly the general case that there exists $x \in A$ such that $x \in \mathrm{Co}(A \setminus \{x\})$, that is, $x$ lies in the convex hull of the set $A \setminus \{x\}$. Then there exists a set $A_1 \subset (A \setminus \{x\})$ such that
\begin{align*}
\#(A_1) = \#(A) - 2
\quad \text{ and } \quad
x \in \mathrm{Co}(A_1).
\end{align*}
Then for a fixed $B \in \pi_h$ with $A_1 \subset A \cap B$, there always holds
\begin{align*}
A_1 \cup \{x\} \subset A \cap B.
\end{align*}
Clearly, there exists no $B \in \pi_h$ such that $A \cap B = A_1$ and therefore $\pi_h$ cannot shatter $A$.

It remains to consider the case when $x \not\in \mathrm{Co}(A \setminus \{x\})$ holds for all $x \in A$. Obviously, the convex hull of $A$ forms a hyper-polyhedron whose vertices are the points of $A$. Note that the hyper-polyhedron can be regarded as an undirected graph, therefore as usual, we define the distance $d(x_1 , x_2)$ between a pair of samples $x_1$ and $x_2$ on the graph by the shortest path between them. Clearly, there exists a starting point $x_0 \in A$ such that $\deg(x) = 2^{d - 1}$.  Then we construct another data set $A_2 \neq A_1$ by
\begin{align*}
A_2 = \{ y : d(x_0, y) \mod 2 = 1, y \in A \}. 
\end{align*}
Again, for a fixed $B \in \pi_h$ such that $A_2 \subset A \cap B$, we deduce that there exists no $B \in \pi_h$ such that $A \cap B = A_2$ and therefore $\pi_h$ cannot shatter $A$ as well. By Definition  \ref{def::VCdimension}, we immediately obtain
\begin{align*} 
\mathrm{VC}(\pi_h) \leq 2^d+2.
\end{align*}

Next, we turn to prove the second assertion. The choice $k := \lfloor \frac{2r\sqrt{d}}{\underline{h}_0} \rfloor+1$ leads to the partition of $B_r$ of the form $\pi_k := \{ A_{i_1, \ldots, i_d} \}_{i_j = 1,\ldots,k}$ with
\begin{align} \label{def::cells}
A_{i_1, \ldots, i_d} 
:= \prod_{j=1}^d A_{i_j}
:= \prod_{j=1}^d \biggl[ - r + \frac{2r(i_j-1)}{k}, -r+\frac{2ri_j}{k} \biggr).
\end{align}
Obviously, we have $|A_{i_j}| \leq \frac{\underline{h}_0}{\sqrt{d}}$. Let $D$ be a data set with 
\begin{align*}
\#(D) = (d (2^d - 1) + 2) \biggl( \biggl\lfloor \frac{2r\sqrt{d}}{\underline{h}_0} \biggr\rfloor + 1 \biggr)^d. 
\end{align*}
Then there exists at least one cell $A$ with 
\begin{align} \label{DcapANo}
\#(D \cap A) \geq d(2^d-1)+2.
\end{align}
Moreover, for all $x, x' \in A$, the construction of the partition \eqref{def::cells} implies $\|x - x'\| \leq \underline{h}_0$.
Consequently, at most one vertex of $A_j$ induced by histogram transform $H$ lies in $A$, since the bandwidth of $A_j$ is larger than $\underline{h}_0$. 
Therefore, 
\begin{align*} 
{\Pi_h}_{|A} := \{B \cap A : B \in \Pi_h\}
\end{align*}
forms a partition of $A$ with $\#({\Pi_h}_{|A}) \leq 2^d$. It is easily seen that this partition can be generated by $2^d-1$ splitting hyper-planes. In this way, Lemma \ref{VCindexPre} implies that ${\Pi_h}_{|A}$ can only shatter a dataset with at most $d(2^d-1)+1$ elements. Thus \eqref{DcapANo} indicates that ${\Pi_h}_{|A}$ fails to shatter $D \cap A$ and therefore
$\Pi_h$ can not shatter the data set $D$ as well. By Definition \ref{def::VCdimension}, we immediately get 
\begin{align*} 
\mathrm{VC}(\Pi_h) 
\leq (d(2^d-1)+2) \biggl( \biggl\lfloor \frac{2r\sqrt{d}}{\underline{h}_0} \biggr\rfloor+1 \biggr)^d
\end{align*}
and the assertion is thus proved.
\end{proof}

\begin{proof}[of Lemma \ref{ScriptBhCoveringNumber}]
The first assertion concerning covering numbers of $\pi_h$ follows directly from Theorem 9.2 in \cite{Kosorok2008introcuction}. For the second estimate, we find the upper bound \eqref{VCMathcalBh}  of $\mathrm{VC}(\Pi_h)$ satisfies
\begin{align*} 
\bigl( d(2^d-1) + 2 \bigr) 
\biggl(\frac{2r\sqrt{d}}{\underline{h}_0}+1\biggr)^d 
\leq \bigl( (d+1)2^d \bigr) \biggl(\frac{3r\sqrt{d}}{\underline{h}_0}\biggr)^d
\leq 2d\cdot 2^d\biggl(\frac{3r\sqrt{d}}{\underline{h}_0}\biggr)^d
:= \biggl( \frac{c_dr}{\underline{h}_0}\biggr)^d,
\end{align*}
where the constant $c_d := 2^{1+\frac{1}{d}}\cdot 3\cdot d^{\frac{1}{d}+\frac{1}{2}}$. Again, Theorem 9.2 in \cite{Kosorok2008introcuction} yields the second assertion and thus completes the proof.
\end{proof}

\begin{proof}[of Proposition \ref{OracleInequalityL1}] 
By Lemma \ref{FundamentalLemma}, we have
\begin{align*}
\|f_{\mathrm{D}, H} - f_{\mathrm{P}, H} \|_{L_1(\mu)}
& =  \sum_{j \in \mathcal{I}_H} |\mathbb{E}_{\mathrm{D}} \eins_{A_j} - \mathbb{E}_{\mathrm{P}} \eins_{A_j}| 
                                              + |\mathbb{E}_{\mathrm{D}} \eins_{B_r^c} - \mathbb{E}_{\mathrm{P}} \eins_{B_r^c}|
\\
& \leq \sup_{\pi_H \in \Pi_H} \sum_{A \in \pi_H} |\mathbb{E}_{\mathrm{D}} \eins_A - \mathbb{E}_{\mathrm{P}} \eins_A|
                                              + |\mathbb{E}_{\mathrm{D}} \eins_{B_r^c} - \mathbb{E}_{\mathrm{P}} \eins_{B_r^c}|,
\end{align*}
where $\pi_H$ is a partition of $B_r$ and $\Pi_H$ denotes the collection of all partitions $\pi_H$. We define
\begin{align*}
\mathcal{B}(\pi_H)
:= \biggl\{ B : B = \bigcup_{j \in I} A_j, I \subset \mathcal{I}_H, A_j \in \pi_H \biggr\}
\end{align*}
as the collection of all $2^{|\mathcal{I}_H|}$ sets that can be expressed as the union of cells of $\pi_H$. Moreover, $\Pi_h$ defined as in \eqref{Bh} also can be used to denote the collection of all such unions, as $\pi_H$ ranges through $\Pi_H$,
that is,
\begin{align*}
\Pi_h = \bigl\{ \mathcal{B}(\pi_H) : \pi_H \in \Pi_H\bigr\}
\end{align*}
For a fixed $\pi_H$, we define
\begin{align*}
\widetilde{A} = \bigcup_{A \in \pi_H : \mathbb{E}_{\mathrm{D}} \eins_A \geq \mathbb{E}_{\mathrm{P}} \eins_A} A.
\end{align*}
Then we have
\begin{align*}
\sum_{A \in \pi_H} |\mathbb{E}_{\mathrm{D}} \eins_A - \mathbb{E}_{\mathrm{P}} \eins_A|
= 2 ( \mathbb{E}_{\mathrm{D}} \eins_{\widetilde{A}} - \mathbb{E}_{\mathrm{P}} \eins_{\widetilde{A}})
\leq 2 \sup_{B \in \Pi_h} |\mathbb{E}_{\mathrm{D}} \eins_B - \mathbb{E}_{\mathrm{P}} \eins_B|
\end{align*}
and
\begin{align}\label{EstErrorDecomp}
\|f_{\mathrm{D},H} - f_{\mathrm{P},H}\|_{L_1(\mu)}
& \leq \sup_{\pi_H \in \Pi_H} \sum_{A \in \pi_H} |\mathbb{E}_{\mathrm{D}} \eins_A - \mathbb{E}_{\mathrm{P}} \eins_A| 
                                                      + |\mathbb{E}_{\mathrm{D}} \eins_{B_r^c} - \mathbb{E}_{\mathrm{P}} \eins_{B_r^c}|
\nonumber\\
& \leq 2 \sup_{\pi_H \in \Pi_H} \sup_{B \in \mathcal{B}(\pi_H)} |\mathbb{E}_{\mathrm{D}} \eins_B - \mathbb{E}_{\mathrm{P}} \eins_B| \ + |\mathbb{E}_{\mathrm{D}} \eins_{B_r^c} - \mathbb{E}_{\mathrm{P}} \eins_{B_r^c}|
\nonumber\\
& = 2 \sup_{B \in \Pi_h} |\mathbb{E}_{\mathrm{D}} \eins_B - \mathbb{E}_{\mathrm{P}} \eins_B| 
                                    + |\mathbb{E}_{\mathrm{D}} \eins_{B_r^c} - \mathbb{E}_{\mathrm{P}} \eins_{B_r^c}|.
\end{align}

Let us first estimate
\begin{align*}
|\mathbb{E}_{\mathrm{D}} \eins_B - \mathbb{E}_{\mathrm{P}} \eins_B|
\end{align*}
by using Bernstein's inequality. For this purpose, we consider the map
\begin{align*}
\xi_i := \eins_B (x_i) - \mathbb{E}_{\mathrm{P}} \eins_B,
\end{align*}
where $x_i$ is the $i$-th sample. Then, we verify the following conditions: Obviously, we have $\mathbb{E}_{\mathrm{P}^n} \xi_i = 0$ and $\|\xi_i\|_{\infty} \leq 1$. Moreover, simple estimates imply
\begin{align}
\mathbb{E}_{\mathrm{P}^n} \xi_i^2 
\leq \mathbb{E}_{\mathrm{P}} \eins_B^2 - (\mathbb{E}_{\mathrm{P}} \eins_B)^2
& = \mathbb{E}_{\mathrm{P}} \eins_B - (\mathbb{E}_{\mathrm{P}} \eins_B)^2
\nonumber\\
& = \mathrm{P}(B) - \mathrm{P}(B)^2 = \mathrm{P}(B)(1-\mathrm{P}(B)) \leq \frac{1}{4}.
\label{VarianceBound}
\end{align}
Finally, it is easy to see that $(\xi_i)$ are independent with respect to $\mathrm{P}^n$. Therefore, we can apply Bernstein's inequality and obtain that for all $n \geq 1$, with probability at most $2 e^{-\tau}$, there holds
\begin{align} \label{FirstBernsteinEstimate}
|\mathbb{E}_{\mathrm{D}} \eins_B - \mathbb{E}_{\mathrm{P}} \eins_B|
= \biggl| \frac{1}{n} \sum_{i=1}^n \xi_i \biggr| 
\geq \sqrt{\frac{\tau}{2n}} + \frac{2 \tau}{3 n}.
\end{align}
We choose $B_1, \ldots, B_{m_1} \in \Pi_h$ such that $\{ B_1, \ldots, B_{m_1} \}$ is an $\varepsilon$-net of $\Pi_h$ with respect to $\|\cdot\|_{L_1(\mathrm{D})}$. Note that here we have $m_1 = \mathcal{N} (\eins_{\Pi_h}, \|\cdot\|_{L_1(\mathrm{D})}, \varepsilon)$ as in Lemma \ref{ScriptBhCoveringNumber}. Using \eqref{FirstBernsteinEstimate} and a union bound argument, we obtain
\begin{align} \label{BernEstimateMax}
\sup_{j = 1, \ldots, m_1} |\mathbb{E}_{\mathrm{D}} \eins_{B_j} - \mathbb{E}_{\mathrm{P}} \eins_{B_j}|
\leq \sqrt{\frac{\tau}{2n}} + \frac{2\tau}{3n}
\end{align}
with probability $\mathrm{P}^n$ at least $1 - 2 m_1 e^{-\tau}$.

Now, for any $B \in \Pi_h$, since $\{ B_1, \ldots, B_{m_1} \}$ is an $\varepsilon$-net of $\Pi_h$, there exists a $B_j$ such that
\begin{align*}
\bigl| |\mathbb{E}_{\mathrm{D}} \eins_B - \mathbb{E}_{\mathrm{P}} \eins_B| 
         - |\mathbb{E}_{\mathrm{D}} \eins_{B_j} - \mathbb{E}_{\mathrm{P}} \eins_{B_j}| \bigr|
& \leq |\mathbb{E}_{\mathrm{D}} \eins_B - \mathbb{E}_{\mathrm{P}} \eins_B 
          - (\mathbb{E}_{\mathrm{D}} \eins_{B_j} - \mathbb{E}_{\mathrm{P}} \eins_{B_j})|
\\
& \leq |\mathbb{E}_{\mathrm{D}} \eins_B - \mathbb{E}_{\mathrm{D}} \eins_{B_j}|
          + |\mathbb{E}_{\mathrm{P}} \eins_B - \mathbb{E}_{\mathrm{P}} \eins_{B_j}|.
\end{align*}
In the following, we estimate the terms on the right hand of the above inequality separately. For the first term, there holds
\begin{align*}
|\mathbb{E}_{\mathrm{D}} \eins_B - \mathbb{E}_{\mathrm{D}} \eins_{B_j}|
& = \biggl| \frac{1}{n} \sum_{i=1}^n (\eins_B (x_i) - \eins_{B_j} (x_i)) \biggr|
\\
& \leq \frac{1}{n} \sum_{i=1}^n |\eins_B (x_i) - \eins_{B_j} (x_i)|
    = \|\eins_B - \eins_{B_j}\|_{L_1(\mathrm{D})}
    \leq \varepsilon.
\end{align*}
As for the second term, we have
\begin{align*}
|\mathbb{E}_{\mathrm{P}} \eins_B - \mathbb{E}_{\mathrm{P}} \eins_{B_j}|
& = \biggl| \int_{\mathbb{R}^d} f \cdot (\eins_B - \eins_{B_j}) \, d\mu \biggr|
\\
& \leq \int_{\mathbb{R}^d} f \cdot |\eins_B - \eins_{B_j}| \, d\mu
\\
& = \|\eins_B - \eins_{B_j}\|_{L_1(\mu)}
    = \mathbb{E}_{\mathrm{P}} (\|\eins_B - \eins_{B_j}\|_{L_1(\mathrm{D})})
    \leq \varepsilon.
\end{align*}
Consequently, we obtain
\begin{align*}
|\mathbb{E}_{\mathrm{D}} \eins_B - \mathbb{E}_{\mathrm{P}} \eins_B|
\leq |\mathbb{E}_{\mathrm{D}} \eins_{B_j} - \mathbb{E}_{\mathrm{P}} \eins_{B_j}| + 2 \varepsilon.
\end{align*}
This together with \eqref{BernEstimateMax} implies that for any $B \in \Pi_h$, there holds
\begin{align} \label{BernEstimate2}
|\mathbb{E}_{\mathrm{D}} \eins_B - \mathbb{E}_{\mathrm{P}} \eins_B|
\leq \sqrt{\frac{\tau}{2n}} + \frac{2\tau}{3n} + 2 \varepsilon
\end{align}
with probability $\mathrm{P}^n$ at least $1 - 2 m_1 e^{-\tau}$. Next, we focus on bounding the second term on the right side of \eqref{EstErrorDecomp}. For a fixed $r > 0$, we consider the map
\begin{align*}
\tilde{\xi}_i := \eins_{B_r^c} (x_i) - \mathbb{E}_{\mathrm{P}} \eins_{B_r^c},
\end{align*}
where $x_i$ is the $i$-th sample. Then, we verify the following conditions: Obviously, we have $\mathbb{E}_{\mathrm{P}^n} \tilde{\xi}_i = 0$ and $\|\tilde{\xi}_i\|_{\infty} \leq 1$. Moreover, simliar estimates as \eqref{VarianceBound} imply
\begin{align*}
\mathbb{E}_{\mathrm{P}^n} \tilde{\xi}_i^2
\leq \mathrm{P}(B_r^c) (1 - \mathrm{P}(B_r^c))
\leq \frac{1}{4}.
\end{align*}
Also, we can see that $(\tilde{\xi}_i)$ are independent with respect to $\mathrm{P}^n$. Therefore, we can apply Bernstein's inequality once again and obtain that for all $n \geq 1$, with probability $\mathrm{P}^n$ at most $2 e^{-\tau}$, there holds
\begin{align} \label{BernEstimateComplement}
|\mathbb{E}_{\mathrm{D}} \eins_{B_r^c} - \mathbb{E}_{\mathrm{P}} \eins_{B_r^c}|
\geq \sqrt{\frac{\tau}{2n}} + \frac{2\tau}{3n}.
\end{align}
Combining \eqref{BernEstimate2} and \eqref{BernEstimateComplement} yields
\begin{align*}
\|f_{\mathrm{D},H} - f_{\mathrm{P},H}\|_{L_1(\mu)}
\leq \sqrt{\frac{9\tau}{2n}} + \frac{2\tau}{n} + 4 \varepsilon
\end{align*}
with probability $\mathrm{P}^n$ at least $1 - 2 (m_1+1) e^{-\tau}$. By a simple variable transformation, we see that with probability $\mathrm{P}^n$ at least $1 - e^{-\tau}$, there holds
\begin{align} \label{EstimationErrorInbetween}
\|f_{\mathrm{D},H} - f_{\mathrm{P},H}\|_{L_1(\mu)}
\leq \sqrt{\frac{9 (\tau + \log(2 m_1 + 2))}{2n}} + \frac{2(\tau+\log(2m_1+2))}{n} + 4 \varepsilon.
\end{align}
Next, we estimate the term $\log(2m_1+2)$ with $m_1 = \mathcal{N}(\eins_{\Pi_h}, \|\cdot\|_{L_1(\mathrm{D})}, \varepsilon)$. Lemma \ref{ScriptBhCoveringNumber} implies that for all $\varepsilon \in (0, 1 / \max \{ e, 2K+2 \})$, there holds
\begin{align}
\log(2m_1+2)
& \leq \log \biggl( 2K \biggl( \frac{c_dr}{\underline{h}_{0,n}} \biggr)^d (4 e)^{(\frac{c_dr}{\underline{h}_{0,n}})^d} \biggl( \frac{1}{\varepsilon} \biggr)^{(\frac{c_dr}{\underline{h}_{0,n}})^d} + 2 \biggr)
\nonumber\\
& \leq \log \biggl( (2K+2) \biggl( \frac{c_dr}{\underline{h}_{0,n}} \biggr)^d (4 e)^{(\frac{c_dr}{\underline{h}_{0,n}})^d} \biggl( \frac{1}{\varepsilon} \biggr)^{(\frac{c_dr}{\underline{h}_{0,n}})^d} \biggr)
\nonumber\\
& = \log (2K+2) + d \log \biggl( \frac{c_dr}{\underline{h}_{0,n}} \biggr) + \biggl( \frac{c_dr}{\underline{h}_{0,n}} \biggr)^d \log (4e) + \biggl( \frac{c_dr}{\underline{h}_{0,n}} \biggr)^d \log \biggl( \frac{1}{\varepsilon} \biggr)
\nonumber\\
& \leq 12 \biggl( \frac{c_dr}{\underline{h}_{0,n}} \biggr)^d  \log \biggl( \frac{1}{\varepsilon} \biggr),  
\label{ElemEstimate}
\end{align}
where the last inequality is based on the following basic inequalities: 
\begin{align*}
\log (2K+2) 
& \leq \log (1/\varepsilon) \leq (c_dr/\underline{h}_{0,n})^d\log (1/\varepsilon),
\\ 
d\log (c_dr/\underline{h}_{0,n}) 
& \leq (c_dr/\underline{h}_{0,n})^d \leq (c_dr/\underline{h}_{0,n})^d\log (1/\varepsilon) ,
\\ 
(c_dr/\underline{h}_{0,n})^d \log (4e) 
& \leq (c_dr/\underline{h}_{0,n})^d\log(e^3) 
   \leq 9(c_dr/\underline{h}_{0,n})^d \log(1/\varepsilon). 
\end{align*}
Now, when choosing $\varepsilon = 1/n$ and plugging \eqref{ElemEstimate} into \eqref{EstimationErrorInbetween}, we obtain
\begin{align*}
\|f_{\mathrm{D},H} - f_{\mathrm{P},H}\|_{L_1(\mu)}
\leq \sqrt{\frac{9(\tau+12(c_dr/\underline{h}_{0,n})^d \log n)}{2n}} + \frac{2(\tau+(c_dr/\underline{h}_{0,n})^d \log n)}{n} + \frac{4}{n}
\end{align*}
with probability $\mathrm{P}^n$ at least $1 - e^{-\tau}$. With the transformation $\tau := \log n$ we get the conclusion.
\end{proof}

\begin{proof}[of Proposition \ref{OracleInequalityInftyNorm}]
Since the density function $f$ considered has a bounded support, we choose $r$ large enough so that the entire support can be contained in $B_r$. According to Lemma \ref{FundamentalLemma}, we have
\begin{align}\label{eq::inftyrep}
\|f_{\mathrm{D},H} - f_{\mathrm{P},H}\|_{L_{\infty}(\mu)}
= \sup_{j \in \mathcal{I}_H} \frac{|\mathbb{E}_{\mathrm{D}} \eins_{A_j} - \mathbb{E}_{\mathrm{P}} \eins_{A_j}|}{\mu(A_j)}.
\end{align}
Let $\pi_h$ be as in \eqref{WidetildeMathcalB}.
Then we have
\begin{align}\label{eq::pih}
\sup_{j \in \mathcal{I}_H} \frac{|\mathbb{E}_{\mathrm{D}} \eins_{A_j} - \mathbb{E}_{\mathrm{P}} \eins_{A_j}|}{\mu(A_j)} 
\leq \sup_{A \in \pi_h} \frac{|\mathbb{E}_{\mathrm{D}} \eins_A - \mathbb{E}_{\mathrm{P}} \eins_A|}{\mu(A)} .
\end{align}
For a fixed $A \in \pi_h$, we estimate
\begin{align*}
\frac{|\mathbb{E}_{\mathrm{D}} \eins_A - \mathbb{E}_{\mathrm{P}} \eins_A|}{\mu(A)} 
\end{align*}
by using Bernstein's inequality. For this purpose, we consider the map
\begin{align*}
\zeta_i := \frac{\eins_A (x_i) - \mathbb{E}_{\mathrm{P}} \eins_A}{\mu(A)},
\end{align*}
where $x_i$ is the $i$-th sample. It is easy to see that $(\zeta_i)$ are independent with respect to $\mathrm{P}^n$. Then, for all cell $A \in \pi_h$ with $\mu(A) \geq \underline{h}_{0,n}^d$, we verify the following conditions: Obviously, we have $\mathbb{E}_{\mathrm{P}^n} \zeta_i = 0$ and 
\begin{align*}
\|\zeta_i\|_{\infty} \leq \frac{1}{\mu(A)} \leq \frac{1}{\underline{h}_{0,n}^d}. 
\end{align*}
Moreover, elementary considerations yield
\begin{align*}
\mathbb{E}_{\mathrm{P}^n} \zeta_i^2 
\leq \frac{\mathbb{E}_{\mathrm{P}} \eins_A^2}{\mu^2(A)} 
= \frac{\mathbb{E}_{\mathrm{P}} \eins_A}{\mu^2(A)} 
= \frac{\mathrm{P}(A)}{\mu^2(A)} 
& = \frac{1}{\mu^2(A)} \int_A f(x) \, d\mu(x)
\\
& \leq \frac{\|f\|_{L_{\infty}(\mu)}}{\mu^2(A)} \int_A 1 \, d\mu(x)
= \frac{\|f\|_{L_{\infty}(\mu)}}{\mu(A)} 
\leq \frac{\|f\|_{L_{\infty}(\mu)}}{\underline{h}_{0,n}^d}.
\end{align*}
Therefore, we can apply Bernstein's inequality and obtain that for all $n \geq 1$, with probability at most $2 e^{-\tau}$, there holds
\begin{align} \label{BernEstimateInfty1}
\frac{|\mathbb{E}_{\mathrm{D}} \eins_A - \mathbb{E}_{\mathrm{P}} \eins_A|}{\mu(A)} 
\geq \sqrt{\frac{2 \|f\|_{L_{\infty}(\mu)} \tau}{n \underline{h}_{0,n}^d}} + \frac{2\tau}{3n\underline{h}_{0,n}^d}.
\end{align}
For any probability measure $\mathrm{Q}$, with
\begin{align*}
\tilde{\varepsilon} := \frac{\underline{h}_{0,n}^{2d} \varepsilon}{\underline{h}_{0,n}^d + \mu(B_r)}
\end{align*}
we choose $A_1, \ldots, A_{m_2} \in \pi_h$ such that $\{ A_1, \ldots, A_{m_2} \}$ is an $\tilde{\varepsilon}$-net of $\pi_h$ with respect to $\|\cdot\|_{L_1(\mathrm{Q})}$, where $m_2 = \mathcal{N}(\pi_h, \|\cdot\|_{L_1(\mathrm{Q})}, \tilde{\varepsilon})$. 
Then the estimate \eqref{BernEstimateInfty1} together with a union bound argument yields that
\begin{align} \label{BernEstimateInftyMax}
\sup_{j \in \{ 1, \ldots, m_2 \}} \frac{|\mathbb{E}_{\mathrm{D}} \eins_{A_j} - \mathbb{E}_{\mathrm{P}} \eins_{A_j}|}{\mu(A_j)} 
\leq \sqrt{\frac{2\|f\|_{L_{\infty}(\mu)} \tau}{n\underline{h}_{0,n}^d}} + \frac{2\tau}{3n\underline{h}_{0,n}^d}
\end{align}
holds with probability $\mathrm{P}^n$ at least $1 - 2 m_2 e^{-\tau}$.
Moreover, the definition of an $\tilde{\varepsilon}$-net implies that
for any $A \in \pi_h$, there exists an $A_j$, $j \in \{ 1, \ldots, m_2 \}$ such that
\begin{align*}
\|\eins_A - \eins_{A_j}\|_{L_1(\mathrm{Q})} \leq \tilde{\varepsilon}.
\end{align*}
Therefore,
\begin{align} 
\biggl\| \frac{\eins_A}{\mu(A)} - \frac{\eins_{A_j}}{\mu(A_j)} \biggr\|_{L_1(\mathrm{Q})}
& = \biggl\| \frac{\eins_A}{\mu(A)} - \frac{\eins_A}{\mu(A_j)} + \frac{\eins_A}{\mu(A_j)} - \frac{\eins_{A_j}}{\mu(A_j)}  \biggr\|_{L_1(\mathrm{Q})}
\nonumber\\
& \leq \frac{|\mu(A_j) - \mu(A)|}{\mu(A) \mu(A_j)} \cdot \|\eins_A\|_{L_1(\mathrm{Q})}
+ \frac{\|\eins_A - \eins_{A_j}\|_{L_1(\mathrm{Q})}}{\mu(A_j)},
\label{EstimateMuAAj}
\end{align}
where $\mu(A) \geq \underline{h}_{0,n}^d$ and $\mu(A_j) \geq \underline{h}_{0,n}^d$. Now, we bound two terms on the right hand of the above inequality separately. For the second term, it can be apparently seen that
\begin{align} \label{EstimateMuAAj1}
\frac{\|\eins_A - \eins_{A_j}\|_{L_1(\mathrm{Q})}}{\mu(A_j)} 
\leq \frac{\tilde{\varepsilon}}{\mu(A_j)} 
\leq \frac{\tilde{\varepsilon}}{\underline{h}_{0,n}^d}.
\end{align}
On the other hand, if $\mathrm{Q}$ is the uniform distribution on $B_r$, then we have
\begin{align*}
\frac{|\mu(A)-\mu(A_j)|}{\mu(B_r)} 
& = \biggl| \int \eins_A - \eins_{A_j} \, d\mu \biggr|
\\
& \leq \frac{1}{\mu(B_r)} \int |\eins_A - \eins_{A_j}| \, d\mu
   = \|\eins_A - \eins_{A_j}\|_{L_1(\mathrm{Q})}
\leq \tilde{\varepsilon}.
\end{align*}
Consequently we obtain
\begin{align} \label{EstimateMuAAj2}
|\mu(A) - \mu(A_j)| \leq \tilde{\varepsilon} \mu(B_r).
\end{align}
If $\mathrm{Q} = \mathrm{D}$, combining \eqref{EstimateMuAAj2} with \eqref{EstimateMuAAj1}, we can bound \eqref{EstimateMuAAj} by
\begin{align} \label{EstimateMuAAj3}
\biggl\| \frac{\eins_A}{\mu(A)} - \frac{\eins_{A_j}}{\mu(A_j)} \biggr\|_{L_1(\mathrm{D})}
\leq \frac{\underline{h}_{0,n}^d + \mu(B_r)}{\underline{h}_{0,n}^{2d}} \tilde{\varepsilon}
= \varepsilon.
\end{align}
Similarly, if $\mathrm{Q} = \mathrm{P}$, it can be deduced that
\begin{align} \label{EstimateMuAAjP}
\biggl\| \frac{\eins_A}{\mu(A)} - \frac{\eins_{A_j}}{\mu(A_j)} \biggr\|_{L_1(\mu)} \leq \varepsilon.
\end{align}
Then, \eqref{EstimateMuAAj3} and \eqref{EstimateMuAAjP} imply that for any $A \in \pi_h$, there holds
\begin{align*}
\biggl| \frac{|\mathbb{E}_{\mathrm{D}} \eins_A - \mathbb{E}_{\mathrm{P}} \eins_A|}{\mu(A)} 
- \frac{|\mathbb{E}_{\mathrm{D}} \eins_{A_j} - \mathbb{E}_{\mathrm{P}} \eins_{A_j}|}{\mu(A_j)} \biggr|
& \leq \biggl| \frac{\mathbb{E}_{\mathrm{D}} \eins_A - \mathbb{E}_{\mathrm{P}} \eins_A}{\mu(A)} 
- \frac{\mathbb{E}_{\mathrm{D}} \eins_{A_j} - \mathbb{E}_{\mathrm{P}} \eins_{A_j}}{\mu(A_j)} \biggr|
\\
& \leq \biggl| \frac{\mathbb{E}_{\mathrm{D}} \eins_A}{\mu(A)}  
- \frac{\mathbb{E}_{\mathrm{D}} \eins_{A_j}}{\mu(A_j)} \biggr|
+ \biggl| \frac{\mathbb{E}_{\mathrm{P}} \eins_A}{\mu(A)} 
- \frac{\mathbb{E}_{\mathrm{P}} \eins_{A_j}}{\mu(A_j)} \biggr|
\\
& \leq \biggl\| \frac{\eins_A}{\mu(A)} - \frac{\eins_{A_j}}{\mu(A_j)} \biggr\|_{L_1(\mathrm{D})}
+ \biggl\| \frac{\eins_A}{\mu(A)} - \frac{\eins_{A_j}}{\mu(A_j)} \biggr\|_{L_1(\mu)}
\\
& \leq 2 \varepsilon
\end{align*}
and consequently we have
\begin{align} \label{eq::resident}
\frac{|\mathbb{E}_{\mathrm{D}} \eins_A - \mathbb{E}_{\mathrm{P}} \eins_A|}{\mu(A)} 
\leq \frac{|\mathbb{E}_{\mathrm{D}} \eins_{A_j} - \mathbb{E}_{\mathrm{P}} \eins_{A_j}|}{\mu(A_j)}  
+ 2 \varepsilon.
\end{align}
This together with \eqref{BernEstimateInftyMax} implies that for any $A \in \pi_h$, there holds
\begin{align} \label{EstimateInbetween2}
\frac{|\mathbb{E}_{\mathrm{D}} \eins_A - \mathbb{E}_{\mathrm{P}} \eins_A|}{\mu(A)} 
\leq \sqrt{\frac{2 \|f\|_{L_{\infty}(\mu)} \tau}{n\underline{h}_{0,n}^d}} + \frac{2\tau}{3n\underline{h}_{0,n}^d} + 2 \varepsilon
\end{align}
with probability $\mathrm{P}^n$ at least $1 - 2  m_2 e^{-\tau}$. By a simple variable transformation, we see that for any $\mu(A_j) \geq \underline{h}_{0,n}^d$, there holds
\begin{align} \label{EstimErrorInfty}
\|f_{\mathrm{D},H} - f_{\mathrm{P},H}\|_{L_{\infty}(\mu)}
\leq \sqrt{\frac{2 \|f\|_{L_{\infty}(\mu)} (\tau + \log(2m_2))}{n\underline{h}_{0,n}^d}}
+ \frac{2(\tau+\log(2m_2))}{3n\underline{h}_{0,n}^d} + 2 \varepsilon
\end{align}
with probability $\mathrm{P}^n$ at least $1 - e^{-\tau}$. Next, we estimate the term $\log(2m_2)$ with $m_2 := \mathcal{N}(\eins_{\pi_h}, \|\cdot\|_{L_1(\mathrm{Q})}, \tilde{\varepsilon})$. The estimate \eqref{CollectionCoveringNumber} implies that for any $\varepsilon \in (0, 1/\max \{ e, 2K, \mu(B_r) \})$, there holds
\begin{align}
\log(2m_2)
& \leq \log (2K (2^d+2) (4e)^{2^d+2} ( 1/\tilde{\varepsilon})^{2^d+1})
\nonumber\\
&\leq \log (2K (2^d+2) (4e)^{2^d+2}( \mu(B_r) + \underline{h}_{0,n}^d)^{2^d+1}(1/\varepsilon\underline{h}_{0,n}^{2d})^{2^d+1})
\nonumber\\
&= \log(2K) + \log(2^d + 2) + (2^d + 2)\log(4e)
+ (2^d + 1)\log(\mu(B_r) + \underline{h}_{0,n}^d)
\nonumber\\
&\quad + (2^d + 1)\log(1/\varepsilon) + 2(2^d + 1)\log(1/\underline{h}_{0,n}^d)
\nonumber\\
&\leq 8 \cdot 2^{d+1} \log(1/\varepsilon) + 2\cdot 2^{d+1}\log(1/\underline{h}_{0,n}^d)
\nonumber\\
&\leq 2^{d+4} \log(1/\varepsilon) +  2^{d+2}\log(1/\underline{h}_{0,n}^d),
        \label{EstimateInbetween3}
\end{align}
where the last inequality is based on the following basic inequalities: 
\begin{align*}
\log(2K) 
& \leq \log(1/\varepsilon) \leq 2^{d+1} \log(1/\varepsilon),
\\ 
\log(2^d+2) 
& \leq 2^d+2 \leq 2^{d+1} \leq 2^{d+1} \log(1/\varepsilon),
\\ 
(2^d+2) \log(4e) 
& \leq 2^{d+1} \log(e^3) \leq 3 \cdot 2^{d+1} \log(1/\varepsilon),
\\ 
(2^d + 1) \log(\mu(B_r)+ \underline{h}_{0,n}^d) 
& \leq 2^{d + 1} \log(2\mu(B_r)) \leq 2\cdot 2^{d + 1} \log(1/\varepsilon). 
\end{align*}
If $x\in B^{+}_{r, \sqrt{d} \cdot \overline{h}_{0,n}}$, then $\mu(A(x))\geq \overline{h}_{0,n}^d$. Now, when choosing $\varepsilon = 1/n$ and plugging \eqref{EstimateInbetween3} into \eqref{EstimErrorInfty}, there holds 
\begin{align}
\|f_{\mathrm{D},H} - f_{\mathrm{P},H}\|_{L_{\infty}(\mu)}
& \leq \sqrt{\frac{2 \|f\|_{L_{\infty}(\mu)} (\tau + 2^{d+4} \log n + 2^{d+2}\log(1/\underline{h}_{0,n}^d))}{n\underline{h}_{0,n}^d}}
\nonumber\\
& \phantom{=}
         + \frac{2(\tau+2^{d+4}\log n+2^{d+2}\log(1/\underline{h}_{0,n}^d))}{3n\underline{h}_{0,n}^d} + \frac{2}{n}
           \label{EstimErrorInfty2}
\end{align}
for all $x \in B^+_{r, \sqrt{d} \cdot \overline{h}_{0,n}}$ with probability $\mathrm{P}^n$ at least $1 - e^{-\tau}$.
\end{proof}

\subsubsection{Proofs Related to Section \ref{subsec::C0}}

\begin{proof}[of Theorem \ref{ConsistencyL1}]
Since $\overline{h}_{0,n} \to 0$, there exists $n_1(\varepsilon) \in \mathbb{N}$ such that for all $n > n_1 $, Proposition \ref{ApproximationError} implies that
\begin{align*}
\|f_{\mathrm{P},H} - f\|_{L_1(\mu)} \leq \varepsilon. 
\end{align*}
Moreover, Proposition \ref{OracleInequalityL1} tells us that if $\log n/(n\underline{h}_{0,n}^d) \to 0$, there exists $n_2 > 0$ such that for all $n > n_2$, there holds 
\begin{align*}
\|f_{\mathrm{D},H} - f_{\mathrm{P},H}\|_{L_1(\mu)}
\leq \varepsilon. 
\end{align*}
Combining the above two inequalities, we obtain the assertion. 
\end{proof}

\begin{proof}[of Theorem \ref{theorem::ConvergenceRatesL1}]
\textit{(i)} Combining the estimates in Proposition \ref{OracleInequalityL1}, Proposition \ref{ApproximationError}  and Proposition \ref{L1LInftyRelation}, we know that with probability $\nu_n$ at least $1 - e^{-\tau}$, there holds
\begin{align*}
\|f_{\mathrm{D},H_n} - f\|_{L_1(\mu)}
& \leq \sqrt{\frac{9(\tau+12(c_dr/\underline{h}_0)^d \log n)}{2n}} 
          + \frac{2(\tau+(c_dr/\underline{h}_0)^d \log n)}{n} + \frac{4}{n}
\\
&\phantom{=} 
         + c_L r^d \overline{h}_0^{\alpha} + 2 \mathrm{P} \bigl( B^c_r \bigr).
\end{align*}
When taking $\tau := \log n$, we have 
\begin{align*}
\|f_{\mathrm{D},H_n} - f\|_{L_1(\mu)}\lesssim \sqrt{\frac{r^d \log n}{n\underline{h}_0^d}} + r^d \overline{h}_0^{\alpha} + \mathrm{P} \bigl( B^c_r \bigr).
\end{align*}
Therefore, with probability $\mathrm{P}^n\otimes \mathrm{P}_H$ at least $1-1/n$ there holds
\begin{align*}
\|f_{\mathrm{D},H_n} - f\|_{L_1(\mu)}
\lesssim \sqrt{\frac{r^d \log n}{n\underline{h}_{0,n}^d}} + r^{- \eta d} + r^d \overline{h}_{0,n}^{\alpha}.
\end{align*}
By choosing
\begin{align*}
\underline{h}_{0,n} 
& := (\log n / n)^{\frac{d+\eta}{\eta(2\alpha+d)+d(\alpha+d)}},
\\
r_n 
& := (n / \log n)^{\frac{\alpha}{d\eta(2\alpha+d)+d(\alpha+d)}},
\end{align*}
there holds 
\begin{align*}
\|f_{\mathrm{D},H_n} - f\|_{L_1(\mu)} 
\lesssim (\log n / n)^{\frac{\alpha\eta }{\eta(2\alpha+d)+(\alpha+d)}}.
\end{align*}

\textit{(ii)} Similar to case \textit{(i)}, one can show that with probability $\mathrm{P}^n\otimes \mathrm{P}_H$ at least $1-1/n$ there holds
\begin{align*}
\|f_{\mathrm{D},H_n} - f\|_{L_1(\mu)}
\lesssim  \sqrt{\frac{r^d \log n}{n\underline{h}_{0,n}^d}} + e^{- a r^\eta} + r^d \overline{h}_{0,n}^{\alpha}.
\end{align*}
By choosing
\begin{align*}
\underline{h}_{0,n} 
& := (\log n / n)^{\frac{1}{2\alpha+d}} (\log n)^{- \frac{d}{\eta} \cdot \frac{1}{2\alpha+d}},
\\
r_n 
& := \bigl( \alpha \log n / (a(2\alpha+d)) \big)^{\frac{1}{\eta}},
\end{align*}
we obtain 
\begin{align*}
\|f_{\mathrm{D},H_n} - f\|_{L_1(\mu)}
\lesssim (\log n / n)^{\frac{\alpha}{2\alpha+d}} (\log n)^{\frac{d}{\eta} \cdot \frac{\alpha+d}{2\alpha+d}}.
\end{align*}

\textit{(iii)} Once again, similar to case \textit{(i)}, it can be showed that with probability $\mathrm{P}^n\otimes \mathrm{P}_H$ at least $1-\frac{1}{n}$, there holds
\begin{align*}
\|f_{\mathrm{D},H_n} - f_{\mathrm{P},H_n}\|_{L_1(\mu)}
\lesssim \sqrt{\frac{\log n}{n\underline{h}_{0,n}^d}} + \overline{h}_{0,n}^\alpha.
\end{align*}
With $\underline{h}_{0,n}$ chosen as  
\begin{align*}
\underline{h}_{0,n} := (\log n / n)^{\frac{1}{2\alpha+d}},
\end{align*}
we obtain 
\begin{align*}
\|f_{\mathrm{D},H_n} - f_{\mathrm{P},H_n}\|_{L_1(\mu)} 
\lesssim (\log n/n)^{\frac{\alpha}{2\alpha+d}}.
\end{align*}
The proof of Theorem \ref{theorem::ConvergenceRatesL1} is thus completed. 
\end{proof}

\begin{proof}[of Theorem \ref{ConvergenceRatesLInfty}]
Similar as the proof of Theorem \ref{theorem::ConvergenceRatesL1} \textit{(iii)}, we conclude that the desired estimate is an easy consequence if we combine the estimates in Proposition \ref{OracleInequalityInftyNorm} and Proposition \ref{ApproximationError} \textit{(ii)} and choose 
\begin{align*}
\underline{h}_{0,n} :=  (\log n/n)^{\frac{1}{2\alpha+d}}.
\end{align*}
We omit the details of the proof here.
\end{proof}

\begin{proof}[of Theorem \ref{ensembleL1}]
Considering the relationship between the $L_1(\mu)$-error of the histogram transform ensemble and the $L_1(\mu)$-errors of single histogram transforms, there holds
\begin{align*}
\|f_{\mathrm{D},\mathrm{E}} - f\|_{L_1(\mu)}
= \biggl\| \frac{1}{T} \sum_{t=1}^T (f_{\mathrm{D},H_t} - f) \biggr\|_{L_1(\mu)}
\leq \frac{1}{T} \sum_{t=1}^T \|f_{\mathrm{D},H_t} - f\|_{L_1(\mu)}.
\end{align*}
In the following, we only present the analysis of the case \emph{(i)} in the three tail probability distributions, since the proof of \emph{(ii)} and \emph{(iii)} are quite similar.

According to Theorem \ref{theorem::ConvergenceRatesL1}, for any $t \in \{ 1, \ldots, T \}$, with 
\begin{align*}
r_n := (n/\log n)^{\frac{\alpha}{d\eta(2\alpha+d)+d(\alpha+d)}}, 
\end{align*}
then there exists a constant $c$ such that
\begin{align*}
\|f_{\mathrm{D},H_t} - f\|_{L_1(\mu)}
> c (\log n/n)^{\frac{\alpha\eta}{(2\alpha+d)\eta+(\alpha+d)}}
=: \mathcal{E} 
\end{align*}
holds with probability $\nu_n$ at least $1 - 1/n$. Then the union bound yields that, for all $\tau > 0$, there holds
\begin{align*}
\nu_n ( \|f_{\mathrm{D},\mathrm{E}} - f\|_{L_1(\mu)} > \mathcal{E} )
\leq \sum_{t=1}^T \nu_n ( \|f_{\mathrm{D},H_t} - f\|_{L_1(\mu)} > \mathcal{E} )
\leq T/n
\end{align*}
and consequently 
\begin{align*}
\|f_{\mathrm{D},\mathrm{E}} - f\|_{L_1(\mu)}
\leq c (\log n/n)^{\frac{\alpha\eta}{(2\alpha+d)\eta+(\alpha+d)}}
\end{align*}
holds with probability $\nu_n$ at least $1 - T/n$.
\end{proof}

\begin{proof}[of Theorem \ref{ensembleLinfty}] 
Let us first consider the relationship between the $L_{\infty}(\mu)$-error of the histogram transform ensemble 
density estimator and the $L_{\infty}(\mu)$-errors of single histogram transform estimators contained in the ensemble. For all $x \in B^+_{r, \sqrt{d} \cdot \overline{h}_{0,n}}$, there holds
\begin{align*}
\|f_{\mathrm{D},\mathrm{E}} - f\|_{L_{\infty}(\mu)}
= \biggl\| \frac{1}{T} \sum_{t=1}^T (f_{\mathrm{D},H_t} - f) \biggr\|_{L_{\infty}(\mu)}
\leq \frac{1}{T} \sum_{t=1}^T \|f_{\mathrm{D},H_t} - f\|_{L_{\infty}(\mu)}.
\end{align*}
The desired estimate can be obtained by choosing the same bandwidth for each partition in $\{ \pi_{H_t} \}_{t=1}^T$, which is
\begin{align*}
\overline{h}_{0,n} =(\log n / n)^{\frac{1}{2\alpha+d}}.
\end{align*}
We omit the details of the proof here for it is similar to that of Theorem \ref{ensembleL1}. 
\end{proof}

\subsection{Proofs of Results in the Space $C^{1, \alpha}$}

The following Lemma presents the explicit representation of $A_H(x)$ which will  play a key role later in the proofs of subsequent sections.

\begin{lemma}\label{binset}
Let the histogram transform $H$ be defined as in \eqref{HistogramTransform} and $A'_H$, $A_H$ be as in \eqref{TransBin} and \eqref{equ::InputBin} respectively. Then for all $x \in \mathbb{R}^d$, the set $A_H(x)$ can be represented as 
\begin{align*}
A_H(x) = \bigl\{ x + (R \cdot S)^{-1} z \ : \  z \in [-b', 1 - b'] \bigr\},
\end{align*}
where $b' := H(x) - \lfloor H(x) \rfloor \sim \mathrm{Unif}(0, 1)^d$.
\end{lemma}

\begin{proof}[of lemma \ref{binset}]
For all $x \in \mathbb{R}^d$, we have $b' \sim \mathrm{Unif}(0,1)^d$ according to the definition of $H$. Moreover, for all $x' \in A'_H(x)$, we define
\begin{align*}
z := H(x') - H(x) = (R \cdot S) (x' - x).
\end{align*}
Then we have
\begin{align*}
x' = x + (R \cdot S)^{-1} z.
\end{align*}
Moreover, since $\lfloor H(x') \rfloor = \lfloor H(x) \rfloor$, we have $z \in [-b', 1 - b']$. 
\end{proof}

\subsubsection{Proofs Related to Section \ref{subsubsec::AppError1}}

\begin{proof}[of Proposition \ref{ApproximationError::LTwo}]
For any $x \in B_r$, the independence of $\{ f_{\mathrm{P},H_t}(x) \}_{t=1}^T$ implies
\begin{align} \label{EnsembleApproxErrorDecomp}
\mathbb{E}_{\mathrm{P}_H} \bigl( f_{\mathrm{P},\mathrm{E}}(x) - f(x) \bigr)^2
= \bigl( \mathbb{E}_{\mathrm{P}_H} ( f_{\mathrm{P},H}(x) ) - f(x) \bigr)^2
+ \frac{1}{T} \cdot \mathrm{Var}_{\mathrm{P}_H}(f_{\mathrm{P},H}(x)).
\end{align}

Let us consider the first term in the RHS of \eqref{EnsembleApproxErrorDecomp}.
Lemma \ref{binset} implies that for any $x' \in A_H(x)$, there exist a random vector $u \sim \mathrm{Unif}[0,1]^d$ and a vector $v \in [0,1]^d$ such that 
\begin{align} \label{xPrimex}
x' = x + S^{-1} R^{\top} (- u + v).
\end{align}
Therefore, we have
\begin{align}
dx' = \det \biggl( \frac{dx'}{dv} \biggr) dv
& = \det \biggl( \frac{d(x + S^{-1} R^{\top}(- u + v))}{dv} \biggr) dv
\nonumber\\
& = \det (R S^{-1}) dv
= \biggl( \prod_{i=1}^d h_i \biggr) dv.
\label{JacobiTrans}
\end{align}
Taking the first-order Taylor expansion of $f(x')$ at $x$, we get
\begin{align} \label{TaylorExpansion}
f(x') - f(x) = \int_0^1 \bigl( \nabla f(x + t(x' - x)) \bigr)^{\top} (x' - x) \, dt. 
\end{align}
Moreover, we obviously have
\begin{align} \label{Trivial}
\nabla f(x)^{\top} (x' - x)
= \int_0^1 \nabla f(x)^{\top} (x' - x) \, dt. 
\end{align}
Thus, \eqref{TaylorExpansion} and \eqref{Trivial} imply that for any $f \in C^{1, \alpha}$, there holds
\begin{align*}
\bigl| f(x') - f(x) - \nabla f(x)^{\top} (x' - x) \bigr|
&= \biggl| \int_0^1 \bigl( \nabla f(x + t(x' - x)) - \nabla f(x) \bigr)^{\top} (x' - x) \, dt \biggr|
\\
& \leq \int^1_0 c_L (t \|x' - x\|_2)^{\alpha} \|x' - x\|_2 \, dt
\\
& \leq c_L \|x' - x\|^{1+\alpha}. 
\end{align*}
This together with \eqref{xPrimex} yields
\begin{align*}
\bigl| f(x') - f(x) - \nabla f(x)^{\top} S^{-1} R^{\top} (- u + v) \bigr|
\leq c_L \overline{h}_0^{1+\alpha}
\end{align*}
and consequently there exists a constant $c_{\alpha} \in [-c_L, c_L]$ such that
\begin{align} \label{TaylorEntwicklung}
f(x') - f(x) = \nabla f(x)^{\top} S^{-1}R^{\top} (- u + v) + c_{\alpha} \overline{h}_0^{1+\alpha}.
\end{align}
For any $x \in B_{r,\sqrt{d} \cdot \overline{h}_0}^+$, there holds
\begin{align*}
f_{\mathrm{P},H}(x) 
= \frac{\mathrm{P}(A_H(x))}{\mu(A_H(x))}
= \frac{1}{\mu(A_H(x))} \int_{A_H(x)} f(x') \, dx'.
\end{align*}
This together with \eqref{TaylorEntwicklung} and \eqref{JacobiTrans} yields
\begin{align}
f_{\mathrm{P},H}(x) - f(x)
& = \frac{1}{\mu(A_H(x))} \int_{A_H(x)} f(x') \, dx' - f(x)
\nonumber\\
& = \frac{1}{\mu(A_H(x))} \int_{A_H(x)} \bigl( f(x') - f(x) \bigr) \, dx'
\nonumber\\
& = \frac{\prod_{i=1}^d h_i}{\mu(A_H(x))}  
\int_{[0,1]^d} \Bigl( \nabla f(x)^{\top} S^{-1} R^{\top} (- u + v) + c_{\alpha} \overline{h}_0^{1+\alpha} \Bigr) \, dv
\nonumber\\
& = \bigg( \int_{[0,1]^d} (- u + v)^{\top} \, dv \biggr) R S^{-1} \nabla f(x) + c_{\alpha} \overline{h}_0^{1+\alpha}
\nonumber\\
& = \biggl( \frac{1}{2} - u \biggr)^{\top} R S^{-1} \nabla f(x) + c_{\alpha} \overline{h}_0^{1+\alpha}.
\label{StepOne}
\end{align}
Since the random variables $(u_i)_{i=1}^d$ are independent and identically distributed as $\mathrm{Unif}[0, 1]$, we have
\begin{align} \label{CrossTermPropertyy1}
\mathbb{E}_{\mathrm{P}_H} \bigg(\frac{1}{2}-u_i \bigg) = 0,
\qquad \qquad
i = 1, \ldots, d.
\end{align}
Combining \eqref{StepOne} with \eqref{CrossTermPropertyy1}, we obtain
\begin{align}
\mathbb{E}_{\mathrm{P}_H} ( f_{\mathrm{P},H}(x) - f(x) )
= 0 + c_{\alpha} \overline{h}_0^{1+\alpha}
= c_{\alpha} \overline{h}_0^{1+\alpha}.
\end{align}
and consequently
\begin{align}\label{equ::biasbound}
\bigl( \mathbb{E}_{\mathrm{P}_H} ( f_{\mathrm{P},H_1}(x) ) - f(x) \bigr)^2
\leq c_L^2 \overline{h}_0^{2(1+\alpha)}.
\end{align}
For the second term in the RHS of \eqref{EnsembleApproxErrorDecomp}, using the assumption $f \in C^{1,\alpha}$, we get
\begin{align} 
\mathrm{Var}_{\mathrm{P}_H}(f_{\mathrm{P},H_1}(x))
& \leq \mathbb{E}_{\mathrm{P}_H} \bigl( f_{\mathrm{P},H_1}(x) - f(x) \bigr)^2
\nonumber\\
& \leq c_L^2 \mathbb{E}_{\mathrm{P}_H} \bigl( \mathrm{diam}(A_{H_1}(x)) \bigr)^2
   \leq d c_L^2 \overline{h}_0^2.
          \label{VarEstimate}
\end{align}
Combining \eqref{EnsembleApproxErrorDecomp} with \eqref{equ::biasbound} and \eqref{VarEstimate}, we obtain
\begin{align*}
\mathbb{E}_{\mathrm{P}_H} \bigl( f_{\mathrm{P},\mathrm{E}}(x) - f(x) \bigr)^2
\leq c_L^2 c_{0,n}^{-2d} \overline{h}_0^{2(1+\alpha)} + \frac{1}{T} \cdot d c_L^2 \overline{h}_0^2,
\end{align*}
which proves the assertion.
\end{proof}

\begin{proof}[of Proposition \ref{ApproximationError::LTwoCounter}]
Lemma \ref{binset} implies that for any $x' \in A_H(x)$, there exist a random vector $u \sim \mathrm{Unif}[0,1]^d$ and a vector $v \in [0,1]^d$ such that 
\begin{align*} 
x' = x + S^{-1} R^{\top} (- u + v).
\end{align*}
Then \eqref{StepOne} yields
\begin{align}    \label{StepOneOne}
(f_{\mathrm{P},H}(x) - f(x))^2 
= \biggl( \biggl( \frac{1}{2} - u \biggr)^{\top} R S^{-1} \nabla f(x) + c_{\alpha} \overline{h}_0^{1+\alpha} \biggr)^2.
\end{align}
The orthogonality \eqref{RotationMatrix} of the rotation matrix $R$ tells us that
\begin{align} \label{CrossTermProperty1}
\sum_{i=1}^d R_{ij} R_{ik} = 
\begin{cases}
1, & \text{ if } j = k, \\
0,& \text{ if } j \neq k
\end{cases}
\end{align} 
and consequently we have
\begin{align} \label{CrossTermProperty2}
\sum_{i=1}^d \sum_{j \neq k} R_{ij} R_{ik} h_j h_k \cdot \frac{\partial f(x)}{\partial x_j} \cdot \frac{\partial f(x)}{\partial x_k}
= \sum_{j \neq k} h_j h_k \cdot \frac{\partial f(x)}{\partial x_j} \cdot \frac{\partial f(x)}{\partial x_k} \sum_{i=1}^d R_{ij} R_{ik}
= 0.
\end{align}
Since the random variables $(u_i)_{i=1}^d$ are independent and identically distributed as $\mathrm{Unif}[0, 1]$, we have
\begin{align} \label{CrossTermProperty3}
\mathbb{E}_{\mathrm{P}_H} \bigg(\frac{1}{2}-u_i \bigg)=0,
\qquad \qquad
i = 1, \ldots, d,
\end{align}
and
\begin{align} \label{CrossTermProperty4}
\mathbb{E}_{\mathrm{P}_H} \bigg(\frac{1}{2}-u_i \bigg)^2=\frac{1}{12},
\qquad \qquad
i = 1, \ldots, d.
\end{align}
Then, for all $x \in B_{r,\sqrt{d} \cdot \overline{h}_0}^+ \cap \mathcal{A}_f$, \eqref{CrossTermProperty1}, \eqref{CrossTermProperty2}, \eqref{CrossTermProperty3}, and \eqref{CrossTermProperty4} yield
\begin{align}
\mathbb{E}_{\mathrm{P}_H}\biggl( \biggl( \frac{1}{2} - u \biggr)^{\top} R S^{-1} \nabla f(x) \biggr)^2 
& = \mathbb{E}_{\mathrm{P}_H}\biggl( \sum_{i=1}^d \biggl( \frac{1}{2} - u_i \biggr) \sum_{j=1}^d R_{ij} h_j \frac{\partial f(x)}{\partial x_j} \bigg)^2 
\nonumber\\
& = \sum_{i=1}^d \mathbb{E}_{\mathrm{P}_H} \biggl( \frac{1}{2} - u_i \biggr)^2 \biggl(\sum_{j=1}^d R_{ij} h_j \frac{\partial f(x)}{\partial x_j} \bigg)^2
\nonumber\\
& = \frac{1}{12} \mathbb{E}_{\mathrm{P}_H} \sum_{i=1}^d \sum_{j=1}^d R_{ij}^2h_j^2\bigg(\frac{\partial f(x)}{\partial x_j} \bigg)^2
\nonumber\\
& \geq \frac{d}{12}\underline{c}_f'^2\underline{h}_0^2 \geq \frac{d}{12}\underline{c}_f'^2c_0^2\overline{h}_0^2.
            \label{ErrorTermMain}
\end{align}
Combining \eqref{StepOne} with \eqref{ErrorTermMain} and using \eqref{CrossTermProperty3}, we see that for all $x \in B_{r,\sqrt{d} \cdot \overline{h}_0}^+ \cap \mathcal{A}_f$, if
\begin{align*}
h_0 \leq \biggl( \frac{\sqrt{d} \underline{c}'_f c_0}{4 \sqrt{3} c_L} \biggr)^{\frac{1}{\alpha}},
\end{align*}
then we have
\begin{align} \label{ApproxiamtionSingle}
\mathbb{E}_{\mathrm{P}_H} (f_{\mathrm{P},H}(x)-f(x))^2 
\geq \frac{d}{16}\underline{c}_f'^2c_0^2\overline{h}_0^2,
\end{align}
where the constant $c_0$ is as in Assumption \ref{assumption::h}. This completes the proof.
\end{proof}

\subsubsection{Proofs Related to Section \ref{subsubsec::EstError1}}

\begin{proof}[of Proposition \ref{OracleInequality::LOne}]
By (\ref{eq::inftyrep}) and (\ref{eq::pih}), we obtain that for $t = 1, \ldots, T$,
\begin{align*}
\|f_{\mathrm{D},H_t} - f_{\mathrm{P},H_t}\|_{L_{\infty}(\mu)}
= \sup_{j \in \mathcal{I}_{H_t}} \frac{|\mathbb{E}_{\mathrm{D}} \eins_{A_j} - \mathbb{E}_{\mathrm{P}} \eins_{A_j}|}{\mu(A_j)} 
\leq \sup_{A \in \pi_h} \frac{|\mathbb{E}_{\mathrm{D}} \eins_A - \mathbb{E}_{\mathrm{P}} \eins_A|}{\mu(A)} .
\end{align*}
Moreover, we have
\begin{align}\label{eq::rewriteE}
\sup_{t=1,\ldots,T}	\|f_{\mathrm{D},H_t} - f_{\mathrm{P},H_t}\|_{L_{\infty}(\mu)} \leq \sup_{A \in \pi_h} \frac{|\mathbb{E}_{\mathrm{D}} \eins_A - \mathbb{E}_{\mathrm{P}} \eins_A|}{\mu(A)} .
\end{align}
Similar to the proof of Proposition \ref{OracleInequalityInftyNorm}, for any $A\in\pi_h$, we also need to bound the term
\begin{align*}
\frac{|\mathbb{E}_{\mathrm{D}} \eins_A - \mathbb{E}_{\mathrm{P}} \eins_A|}{\mu(A)}.
\end{align*}
Choose $A_1, \cdots, A_{m_2} \in \pi_h$ such that $\{A_1, \cdots, A_{m_2}\}$ is an $\tilde{\varepsilon}$-net of $\pi_h$ with respect to $\|\cdot\|_{L_1(Q)}$, where $m_2 = \mathcal{N}(\pi_h, L_1(Q), \tilde{\varepsilon})$.
Similar to the proof of Proposition \ref{OracleInequalityInftyNorm}, we get the same result as \eqref{eq::resident}:
for any $A\in \pi_h$, there exists $j\in\{1,\cdots,m_2\}$ such that
\begin{align}
\frac{|\mathbb{E}_{\mathrm{D}} \eins_A - \mathbb{E}_{\mathrm{P}} \eins_A|}{\mu(A)} 
\leq \frac{|\mathbb{E}_{\mathrm{D}} \eins_{A_j} - \mathbb{E}_{\mathrm{P}} \eins_{A_j}|}{\mu(A_j)}  
+ 2 \varepsilon.
\end{align}
where 
\begin{align*}
\tilde{\varepsilon} := \frac{\underline{h}_{0,n}^{2d} \varepsilon}{\underline{h}_{0,n}^d + \mu(B_r)}
\end{align*}
This together with \eqref{eq::rewriteE} yields
\begin{align*}
\sup_{t=1,\ldots,T}	\|f_{\mathrm{D},H_t} - f_{\mathrm{P},H_t}\|_{L_{\infty}(\mu)}
\leq \sup_{j=1,\ldots,m_2}  \frac{|\mathbb{E}_{\mathrm{D}} \eins_{A_j} - \mathbb{E}_{\mathrm{P}} \eins_{A_j}|}{\mu({A_j})}
+2\varepsilon.
\end{align*}
By Bernstein's inequality and union bound argument, we obtain the same result as  \eqref{BernEstimateInftyMax}, that is,
with probability $\mathrm{P}^n$ at least $1-2m_2 e^{-\tau}$, there holds 
\begin{align*}
\sup_{j \in \{ 1, \ldots, m_2 \}} \frac{|\mathbb{E}_{\mathrm{D}} \eins_{A_j} - \mathbb{E}_{\mathrm{P}} \eins_{A_j}|}{\mu(A_j)} 
\leq \sqrt{\frac{2\|f\|_{L_{\infty}(\mu)} \tau}{n\underline{h}_{0,n}^d}} + \frac{2\tau}{3n\underline{h}_{0,n}^d}.
\end{align*}
By a simple variable transformation, we see that for any $\mu(A_j) \geq \underline{h}_{0,n}^d$, 
\begin{align} \label{EstimErrorInftyE}
\sup_{t=1,\ldots,T}	\|f_{\mathrm{D},H_t} - f_{\mathrm{P},H_t}\|_{L_{\infty}(\mu)}
\leq \sqrt{\frac{2 \|f\|_{L_{\infty}(\mu)} (\tau + \log(2m_2))}{n\underline{h}_{0,n}^d}}
+ \frac{2(\tau+\log(2m_2))}{3n\underline{h}_{0,n}^d} + 2 \varepsilon
\end{align}
holds with probability $\mathrm{P}^n$ at least $1 - e^{-\tau}$.
Next, we estimate the term $\log(2m_2)$ with $m_2 := \mathcal{N}(\eins_{\pi_h}, \|\cdot\|_{L_1(\mathrm{Q})}, \tilde{\varepsilon})$. 
The upper bound \eqref{CollectionCoveringNumber} of $\mathcal{N}(\eins_{\pi_h}, \|\cdot\|_{L_1(\mathrm{Q})}, \tilde{\varepsilon})$  implies that for any $\varepsilon \in (0, 1/\max \{ e, 2K, \mu(B_r) \})$, there holds
\begin{align*}
\log(2m_2)
\leq 2^{d+4} \log(1/\varepsilon) +  2^{d+2}\log(1/\underline{h}_{0,n}^d),
\end{align*}
This together with \eqref{EstimErrorInftyE} implies that for all $x \in B^+_{r, \sqrt{d} \cdot \overline{h}_{0,n}}$, there holds
\begin{align*}
\sup_{t=1,\ldots,T}	\|f_{\mathrm{D},H_t} - f_{\mathrm{P},H_t}\|_{L_{\infty}(\mu)}
& \leq \sqrt{\frac{2 \|f\|_{L_{\infty}(\mu)} (\tau + 2^{d+4} \log(1/\varepsilon) +  2^{d+2}\log(1/\underline{h}_{0,n}^d))}{n\underline{h}_{0,n}^d}}
\nonumber\\
& \phantom{=}
+ \frac{2(\tau+2^{d+4}\log(1/\varepsilon)+2^{d+2}\log(1/\underline{h}_{0,n}^d))}{3n\underline{h}_{0,n}^d} + 2\varepsilon
\end{align*}
with probability $\mathrm{P}^n$ at least $1 - e^{-\tau}$.
Taking $\varepsilon = 1/n$, then for any 
$n > N_0 := \max\{e, 2K, \mu(B_r)\}$,
we have 
\begin{align*}
\sup_{t=1,\ldots,T}	\|f_{\mathrm{D},H_t} - f_{\mathrm{P},H_t}\|_{L_{\infty}(\mu)}
& \leq \sqrt{\frac{2 \|f\|_{L_{\infty}(\mu)} (\tau + 2^{d+4} \log(1/\varepsilon) +  2^{d+2}\log(1/\underline{h}_{0,n}^d))}{n\underline{h}_{0,n}^d}}
\nonumber\\
& \phantom{=}
+ \frac{2(\tau+2^{d+4}\log(1/\varepsilon)+2^{d+2}\log(1/\underline{h}_{0,n}^d))}{3n\underline{h}_{0,n}^d} + \frac{2}{n}
\end{align*}
with probability $\mathrm{P}^n$ at least $1 - e^{-\tau}$.
Finally, we need to estimate the term 
\begin{align*}
\| f_{\mathrm{D},\mathrm{E}} - f_{\mathrm{P},\mathrm{E}}\|_{L_{\infty}(\mu)}.
\end{align*}
For all $\delta > 0$ and all $x \in B^+_{r, \sqrt{d} \cdot \overline{h}_{0,n}}$
, we have
\begin{align*}
\mathrm{P} \biggl( \biggl\| \frac{1}{T}\sum^T_{t=1} f_{\mathrm{D},H_t}-\frac{1}{T}\sum^T_{t=1} f_{\mathrm{P},H_t} \biggr\|_{L_{\infty}(\mu)} \leq \delta \biggr)
& \geq \mathrm{P} \biggl( \frac{1}{T} \sum^T_{t=1} \|f_{\mathrm{D},H_t}-f_{\mathrm{P},H_t}\|_{L_{\infty}(\mu)}
\leq \delta \biggr)
\\
& \geq \mathrm{P} \biggl( \sup_{t=1,\ldots,T}	\|f_{\mathrm{D},H_t} - f_{\mathrm{P},H_t}\|_{L_{\infty}(\mu)}
\leq \delta \biggr).
\end{align*}
Consequently, for any $n > N_0$ and for all $x \in B^+_{r, \sqrt{d} \cdot \overline{h}_{0,n}}$, we obtain 
\begin{align}\label{equ::sampleerrorP}
\|f_{\mathrm{D},\mathrm{E}}-f_{\mathrm{P},\mathrm{E}}\|_{L_{\infty}(\mu)}
& \leq \sqrt{\frac{2 \|f\|_{L_{\infty}(\mu)} (\tau + 2^{d+4} \log n +  2^{d+2}\log(1/\underline{h}_{0,n}^d))}{n\underline{h}_{0,n}^d}}
\nonumber\\
& \phantom{=}
+ \frac{2(\tau+2^{d+4}\log n+2^{d+2}\log(1/\underline{h}_{0,n}^d))}{3n\underline{h}_{0,n}^d} + \frac{2}{n}
\end{align}
with probability $\mathrm{P}^n$ at least $1 - e^{-\tau}$.
\end{proof}

\begin{proof}[of Proposition \ref{OracleInequality::LTwoCounter}]
Recall that for a fixed histogram transform $H$, the set $\pi_H$ is defined as the collection of all cells in the partition induced by $H$, that is, $\pi_H := \{ A_j \}_{j \in \mathcal{I}_H}$. To estimate the first term in \eqref{equ::L2Decomposition}, we observe that for any $x \in B_r$, there holds
\begin{align}
\mathbb{E}_{\mathrm{P}^n} \bigl( (f_{\mathrm{D},H}(x) - f_{\mathrm{P},H}(x))^2 | \pi_H \bigr)
& = \mathrm{Var}_{\mathrm{P}^n} \bigl( f_{\mathrm{D},H}(x) | \pi_H \bigr)
\nonumber\\
& = \mathrm{Var}_{\mathrm{P}^n} \biggl( \frac{1}{n \mu(A_H(x))} \sum_{i=1}^n \eins_{\{ x_i \in A_H(x) \}} \bigg| \pi_H \biggr)
\nonumber\\
& \geq \frac{1}{n^2\overline{h}_{0,n}^{2d}} \sum_{i=1}^n \mathrm{P}(A_H(x)) (1 - \mathrm{P}(A_H(x)))
\nonumber\\
& = \frac{1}{n\overline{h}_{0,n}^{2d}} \mathrm{P}(A_H(x))(1 - \mathrm{P}(A_H(x))),
    \label{ConditionalSampleError}
\end{align}
where $\mathbb{E}_{\mathrm{P}^n}(\cdot | \pi_H)$ and $\mathrm{Var}_{\mathrm{P}^n}(\cdot | \pi_H)$ denote the conditional expectation and conditional variance with respect to $\mathrm{P}^n$ on the partition $\pi_H$, respectively. 

Lemma \ref{binset} implies that for any $x' \in A_H(x)$, there exist a random vector $u \sim \mathrm{Unif}[0,1]^d$ and a vector $v \in [0,1]^d$ such that 
\begin{align*} 
x' = x + S^{-1} R^{\top} (- u + v).
\end{align*}
By \eqref{TaylorEntwicklung} and \eqref{JacobiTrans}, we have
\begin{align}
\mathrm{P}(A_H(x)) 
& = \int_{A_H(x)} f(x') \, dx'
\nonumber\\
& = \biggl( \prod_{i=1}^d h_i \biggr) 
\int_{[0,1]^d} \bigl( f(x) + \nabla f(x)^{\top} S^{-1} R^{\top} (- u + v) + c_{\alpha} \overline{h}_0^{1+\alpha} \bigr) \, dv
\nonumber\\
& = \biggl( \prod_{i=1}^d h_i \biggr) \biggl( 
f(x) + \int_{[0,1]^d} \nabla f(x)^{\top} S^{-1} R^{\top} (- u + v) \, dv + c_{\alpha} \overline{h}_0^{1+\alpha} \biggr)
\nonumber\\
& = \biggl( \prod_{i=1}^d h_i \biggr) \biggl( 
f(x) + \bigg(\int_{[0,1]^d} (-u+v)^{\top} \ dv\bigg) R S^{-1} \nabla f(x) + c_{\alpha} \overline{h}_0^{1+\alpha} \biggr)
\nonumber\\
& = \biggl( \prod_{i=1}^d h_i \biggr) \biggl( 
f(x) + \biggl( \frac{1}{2} - u \biggr)^{\top} R S^{-1} \nabla f(x) + c_{\alpha} \overline{h}_0^{1+\alpha} \biggr).
      \label{PAHx}
\end{align}
Elementary Analysis tells us that for any $a_1, \ldots, a_d \in \mathbb{R}$, there holds
\begin{align*} 
\frac{a_1 + \ldots + a_d}{d} 
\leq \sqrt{\frac{a_1^2 + \ldots + a_d^2}{d}},
\end{align*}
which implies that
\begin{align*} 
\biggl| \biggl( \frac{1}{2} - u \biggr)^{\top} R S^{-1} \nabla f(x) \biggr|
\leq d \cdot \frac{1}{2} \cdot \sqrt{d} \cdot  \overline{h}_0 \cdot c_L
= \frac{d \sqrt{d} c_L}{2} \cdot  \overline{h}_0.
\end{align*}
This together with \eqref{PAHx} yields that for all $x \in B_{r, \sqrt{d} \cdot \overline{h}_0}^+ \cap \mathcal{A}_f$, there hold
\begin{align}
\mathrm{P}(A_H(x)) 
\leq \overline{h}_0^d \biggl( \overline{c}_f 
               + \frac{d \sqrt{d} c_L}{2} \cdot  \overline{h}_0 
                    + c_{\alpha} \overline{h}_0^{1+\alpha} \biggr)
\end{align}
and
\begin{align}
\mathrm{P}(A_H(x)) 
\geq \overline{h}_0^d \biggl( \underline{c}_f
              - \frac{d \sqrt{d} c_L}{2} \cdot  \overline{h}_0
                   + c_{\alpha} \overline{h}_0^{1+\alpha} \biggr).
\end{align}
Then for any $n > N'$ with $N'$ as in \eqref{MinimalNumberSample}, we have
\begin{align} \label{PAHxSize}
\frac{1}{2} \underline{c}_f \overline{h}_0^d
\leq \mathrm{P}(A_H(x)) 
\leq 2 \overline{c}_f \overline{h}_0^d
\leq \frac{1}{2}.
\end{align}
Combining \eqref{ConditionalSampleError} with \eqref{PAHxSize}, we obtain
\begin{align*}
\mathbb{E}_{\mathrm{P}^n} \bigl( ( f_{\mathrm{D},H}(x) - f_{\mathrm{P},H}(x) )^2 | \pi_H \bigr)
& \geq \frac{\mathrm{P}(A_H(x)) (1 - \mathrm{P}(A_H(x)))}{n\overline{h}_{0,n}^{2d}}
\\
& \geq \frac{\mathrm{P}(A_H(x))}{2 n\overline{h}_{0,n}^{2d}}
   \geq \frac{\underline{c}_f \overline{h}_{0,n}^{d}}{4 n \overline{h}_{0,n}^{2d}}
   = \frac{\underline{c}_f}{4 n \overline{h}_{0,n}^{d}}.
\end{align*}
Consequently, for all $x \in B_{r, \sqrt{d} \cdot \overline{h}_0}^+ \cap \mathcal{A}_f$ and all $n \geq N'$, there holds
\begin{align} \label{EstimationSingle}
\mathbb{E}_{\mathrm{P}^n} \bigl( ( f_{\mathrm{D},H}(x) - f_{\mathrm{P},H}(x) )^2\bigr) 
\geq \frac{\underline{c}_f}{4 n\overline{h}_{0,n}^{d}}.
\end{align}
Thus, we proved the assertion.
\end{proof}

\subsubsection{Proofs Related to Section \ref{sec::LowerBoundSingles}}

\begin{proof}[of Theorem \ref{cor::ensemblerateP}]
Proposition \ref{ApproximationError::LTwo} together with Proposition \ref{OracleInequality::LOne} yields that
for all $x \in B^+_{r, \sqrt{d} \cdot \overline{h}_{0,n}}$, there holds
\begin{align*}
\|f_{\mathrm{D},\mathrm{E}} - f\|_{L_{\infty}(\mu)}
& \leq e^{\tau/2} c_L^2 \biggl( c_{0,n}^{-2d} \overline{h}_{0,n}^{2(1+\alpha)} + \frac{d}{T} \cdot  \overline{h}_{0,n}^2 \biggr)
\\
& \phantom{=}
+ \sqrt{\frac{2 \|f\|_{L_{\infty}(\mu)} (\tau + 2^{d+4} \log n +  2^{d+2}\log(1/\underline{h}_{0,n}^d))}{n\underline{h}_{0,n}^d}}
\nonumber\\
& \phantom{=}
+ \frac{2(\tau+2^{d+4}\log n+2^{d+2}\log(1/\underline{h}_{0,n}^d))}{3n\underline{h}_{0,n}^d} + \frac{2}{n}
\end{align*} 
in the sense of $L_2(\mathrm{P}_H)$-norm with probability $\mathrm{P}^n$ at least $1-e^{-\tau}$. Then, for all $x \in B^+_{r, \sqrt{d} \cdot \overline{h}_{0,n}}$, by choosing
\begin{align*}
\underline{h}_{0,n} & := (n/\log n)^{-\frac{1}{2 (1+\alpha) + d}},
\\
T_n & := n^{\frac{2 \alpha}{2 (1+\alpha) + d}},
\end{align*}
we obtain 
\begin{align}\label{RatesLoneEnsemble}
\|f_{\mathrm{D},\mathrm{E}} - f\|_{L_{\infty}(\mu)}
\lesssim   (\log n/n)^{\frac{1+\alpha}{2(1+\alpha)+d}}
\end{align}	
in the sense of $L_2(\mathrm{P}_H)$-norm with probability $\mathrm{P}_n$ at least $1 - e^{-\tau}$.
\end{proof}

\begin{proof}[of Theorem \ref{thm::LowerBoundSingles}]
Recall the error decomposition \eqref{equ::L2Decomposition} of single random histogram transform density estimator.
Then \eqref{ApproxiamtionSingle} and \eqref{EstimationSingle} yield that for all $x \in B_{r, \sqrt{d} \cdot \overline{h}_0}^+ \cap \mathcal{A}_f$ and all $n > N_0$, there holds
\begin{align*}
\mathbb{E}_{\mathrm{P}_H \otimes \mathrm{P}^n} (f_{\mathrm{D},H}(x)-f(x))^2 
\geq \frac{d}{16} \underline{c}_f'^2 c_0^2 \cdot \overline{h}_{0,n}^2 + \frac{\underline{c}_f}{4 n\overline{h}_{0,n}^{d}}.
\end{align*}
By choosing 
\begin{align*}
\overline{h}_{0,n} := n^{-\frac{1}{2+d}}, 
\end{align*}
we obtain
\begin{align*}
\mathbb{E}_{\nu_n} (f_{\mathrm{D},H}(x)-f(x))^2 
\gtrsim n^{-\frac{2}{2+d}},
\end{align*}
which proves the assertion.
\end{proof}

\section{Conclusion}\label{sec::conclusion}

In this paper, we study the density estimation problem with a nonparametric strategy named histogram transform ensembles (HTE) density estimator which provides an effective means taking full advantage of the nature of ensemble learning, large diversity due to the random histogram transform and the inherent local adaptiveness of histogram estimators. The main results presented in this paper cover the universal consistency under $L_1(\mu)$-norm and the convergence rates in terms of several norms and convergence types under different reasonable tail assumptions.
More precisely, for the case where $f$ lies in the space $C^{0,\alpha}$, the convergence rates for both single and ensemble estimators  are proved to be optimal up to certain logarithmic factor in the sense of $L_1(\mu)$- and $L_{\infty}(\mu)$-norm. 
In contrast, for the subspace $C^{1,\alpha}$, almost optimal convergence rates can merely be established for the ensembles and the lower bound of single estimators illustrates the benefits of histogram transform ensembles over single density estimator. In experiments, we propose an adaptive version of HTE. Numerical comparisons between our adaptive HTE and other density estimators further illustrate the satisfactory performance with respect to the estimation accuracy.

\small{}

\end{document}